\documentclass{article}
\usepackage[utf8]{inputenc}
\usepackage{multirow}
\usepackage{lineno}

\usepackage[overload]{empheq}
\usepackage{amsfonts}
\usepackage{amsmath}
\usepackage{amssymb}
\usepackage{bbm}
\usepackage{dsfont}
\usepackage{diagbox}
\usepackage{graphicx}
\usepackage{tikz}
\usepackage{caption}
\usepackage{subcaption}
\usepackage{placeins, microtype}
\usepackage{latexsym,amsthm,xcolor,graphicx, longtable, tabu}
\usepackage{hyperref}
\usepackage{mathtools}
\mathtoolsset{showonlyrefs}

\makeatletter
\def\thm@space@setup{
  \thm@preskip=0.4cm 
  \thm@postskip=\thm@preskip 
}
\makeatother

\newtheorem{thm}{Theorem}[section]
\newtheorem{definition}[thm]{Definition}

\newtheorem{lemma}[thm]{Lemma}

\newtheorem{remark}[thm]{Remark}
\newtheorem{prop}[thm]{Proposition}
\newtheorem{corollary}[thm]{Corollary}

\def\S{\mathcal S}
\def\R{\mathbb{R}}
\def\Z{\mathbb{Z}}
\def\eqd{\,{\buildrel d \over =}\,}
\def\P{{\mathbb P}}     
\def\E{{\mathbb E}} 
\newcommand{\cl}{{\text{cl}}}
\newcommand{\calE}{\mathcal{E}}
\newcommand{\calF}{\mathcal{F}}
\newcommand{\TV}{\text{TV}}
\newcommand{\calK}{\mathcal{K}}
\newcommand{\calP}{\mathcal{P}}
\newcommand{\Law}{\mathcal{L}}
\newcommand{\Leb}{\text{Leb}}
\newcommand{\calC}{\mathcal{C}}
\newcommand{\calR}{\mathcal{R}}
\newcommand{\Rm}{\mathbb{R}}
\newcommand{\calB}{\mathcal{B}}
\newcommand{\calG}{\mathcal{G}}

\newcommand{\expE}{\mathbb{E}}
\newcommand{\Ind}{\mathbbm{1}}
\newcommand{\Pm}{\mathbb{P}}
\newcommand{\Zm}{\mathbb{Z}}
\newcommand{\bfS}{\mathbb{S}}

\newcommand{\st}{{\text{st}}}

\newcommand{\ra}{\rightarrow}

\newcommand{\taud}{\tau_{\partial}}

\newcommand{\fix}{\text{fix}}

\newcommand{\commentout}[1]{}

\long\def\metanote#1#2{{\color{#1}\
\ifmmode\hbox\fi{\sffamily\mdseries\upshape [#2]}\ }}

\long\def\WF#1{\metanote{red!70!black}{{\tiny Louis} #1}}

\def\<{{\langle}} 
\def\>{{\rangle}}

\numberwithin{equation}{section}
\newcounter{keepeqno}
\usepackage{fancyhdr}
\title{Quasi-stationary behavior of the stochastic FKPP equation \\ on the circle
}

\author{Wai-Tong (Louis) Fan\footnote{
Department of Mathematics, Indiana University, Bloomington, IN, USA. (waifan@iu.edu)} 
\footnote{
Department of Organismic and Evolutionary Biology, Harvard University,  MA, USA. (lfan@cmsa.fas.harvard.edu)} 
\and 
Oliver Tough\footnote{
Department of Mathematical Sciences, University of Bath, UK. (okt24@bath.ac.uk)}
}

\date{\today}

\begin{document}

\captionsetup{width=0.85\textwidth}

\maketitle

\abstract{ 
We consider the stochastic Fisher-Kolmogorov-Petrovsky-Piscunov (FKPP) equation on the circle $\mathbb{S}$,
\begin{equation*}
   	\partial_t u(t,x)      \,= \frac{\alpha}{2}\Delta u +\beta\,u(1-u) + \sqrt{\gamma\,u(1-u)}\,\dot{W}, \qquad (t,x)\in(0,\infty)\times \mathbb{S},
\end{equation*}
where $\dot{W}$ is space-time  white noise. While any solution will eventually be absorbed at one of two states, the constant 1 and the  constant 0 on the circle, essentially nothing had been established about the absorption time (also called the fixation time in population genetics), or about the long-time behavior prior to absorption. We establish the existence and uniqueness of the quasi-stationary distribution (QSD) for the solution of the stochastic FKPP. Moreover, we show that the solution conditioned on not being absorbed at time $t$ converges to this unique QSD as  $t\to\infty$, for any initial distribution, and characterize the leading-order asymptotics for the tail distribution of the fixation time. We obtain explicit calculations in the neutral case ($\beta=0$), quantifying the effect of spatial diffusion on fixation time. We explicitly express the fixation rate in terms of the migration rate $\alpha$ for all $\alpha\in (0,\infty)$, finding in particular that the fixation rate is given by $\gamma[1-\frac{\gamma}{12\alpha}+\mathcal{O}(\frac{\gamma^2}{\alpha^2})]$ for fast migration and $\pi^2\alpha[1-\frac{8\alpha}{\gamma}+\mathcal{O}(\frac{\alpha^2}{\gamma^2})]$ for slow migration. Our proof relies on the observation that the absorbed (or killed) stochastic FKPP is dual to a system of $2$-type branching-coalescing Brownian motions killed when one type dies off,  and on leveraging the relationship between these two killed processes. 
}

\section{Introduction}

Let $\mathbb{S}\simeq [0,1)$ be the circle whose circumference is equal to 1.
In this paper, we consider  the stochastic Fisher-Kolmogorov-Petrovsky-Piscunov
(FKPP) equation
\begin{equation}\label{fkpp_X}
   	\partial_t u(t,x)      \,= \frac{\alpha}{2}\Delta u +\beta\,u(1-u) + \sqrt{\gamma\,u(1-u)}\,\dot{W}, \qquad (t,x)\in(0,\infty)\times \mathbb{S},
\end{equation}
where $\dot{W}=\{\dot{W}(t,x)\}_{(t,x)\in
[0,\infty)\times \mathbb{S}}$ is space-time Gaussian white noise, and  $\alpha\in(0,\infty)$, $\beta\in\R$ and $\gamma\in(0,\infty)$ are constants. 
We adopt Walsh's theory \cite{MR876085} to regard the stochastic partial differential equation (SPDE) \eqref{fkpp_X} as shorthand for the integral equation 
\begin{align}\label{E:MildSol_FKPP}
u_t(x)= \int_{\mathbb{S}} p(t,x,y)\,u_0(y)\,m(dy) &+ \int_0^t\int_{\mathbb{S}}p(t-s,x,y)\,\beta u_s(z) (1-u_s(y))\,m(dy)\,ds   \notag\\
&+ \int_{\mathbb{S}\times [0,t]}p(t-s,x,y)\,
\sqrt{\gamma u_s(y) (1-u_s(y))}\,dW(y,s), 
\end{align}
where 
$p(t,x,y)=\frac{1}{\sqrt{2\pi \alpha t}} \sum_{k\in\Z} e^{\frac{-(y-x+ k)^2}{2\alpha t}}$  is the transition density of a Brownian motion on $\mathbb{S}\simeq [0,1)$ with variance $\alpha$, with respect to the 1-dimensional Lebesgue measure $m(dy)$. Roughly speaking,
a stochastic process $u=(u_t)_{t\geq 0}$ is said to be a  \textbf{mild solution} to equation \eqref{fkpp_X} with initial condition $u_0$ if $u$ satisfies \eqref{E:MildSol_FKPP}. Hereafter, we denote $u_t(x)=u(t,x)$ (not a partial derivative). See Section \ref{appendix section:mild solutions} and \cite{MR876085} for details such as the meaning of the stochastic integral in \eqref{E:MildSol_FKPP}. Whenever we refer to the stochastic FKPP \eqref{fkpp_X} in this paper, we refer to its mild solution $u=(u_t)_{t\in\R_{\geq 0}}$, which exists and is unique in law.

The stochastic FKPP  \eqref{fkpp_X} and the analogous equation on the real line $\R$ have attracted intense mathematical study; see for instance \cite{doering2003interacting, hobson2005duality,mueller2009some, MR1271224, mueller2011effect}. It arose as a model in population genetics \cite{shiga1988stepping}, and is a prototypical model for front propagation in reaction-diffusion systems. It has found wide applications in physical chemistry, biophysics and other scientific fields; see the survey \cite{panja2004effects}.
Its importance stems from the fact that it is the universal scaling limit of various microscopic particle models such as the stepping stone model, the biased voter model and interacting stochastic ODE \cite{mueller1995stochastic, durrett2016genealogies, fan2017stochastic}. 

The reaction diffusion equation $\partial_t u(t,x)      \,= \frac{\alpha}{2}\Delta u +\beta\,u(1-u)$, the (deterministic) FKPP equation, was originally derived independently and at the same time by Fisher \cite{fisher1937wave} and Kolmogorov, Petrovskii and Piskunov \cite{Kolmogorov1937} as a model for the spread of an advantageous gene. In this equation, $u(t,x)$ is the population density at time $t$ and location $x$ for the individuals with the favored gene (call them type 1 individuals), and $1-u$ is the remaining population density of the other type (call them type 0). The term $\frac{\alpha}{2}\Delta u$ captures the spatial motion of the individuals; the term $\beta u(1-u)$ captures the average increase (when $\beta>0$) of the density of type 1 individuals due to random interaction between the two types, where $\beta$ is called the selection strength; such random interaction between types, ignored in \cite{fisher1937wave} but captured in the stochastic FKPP, has variance proportional to $u(1-u)$ and gives rise to the term $\sqrt{\gamma\,u(1-u)}\,\dot{W}$ in \eqref{fkpp_X}. The reciprocal of $\gamma$ can be viewed as proportional to the local effective  population size \cite{hallatschek2008gene, durrett2016genealogies}.
In the setting of the biased voter model of  \cite{durrett2016genealogies}, $\gamma^{-1} \propto \frac{M}{\alpha L}$ where $M$ is the total number of individuals within a small interval with length $1/L$ in space, when $M$ and $L$ are both large.

Equation \eqref{fkpp_X} has two absorbing states, namely the  functions \textbf{1} and \textbf{0} on $\mathbb{S}$ that are  constant 1 and 0 respectively.  Absorption here is also referred to as fixation (more precisely, fixation of either type) - it represents the phenomenon in population genetics of one genetic type being fixed, while the other disappears from the population.
We define the fixation time $\tau_{\fix}$ to be the time at which fixation occurs,
\begin{equation}\label{Def:taufix}
    \tau_{\fix}:=\inf\{t\geq 0:\,u_t=\textbf{0}\quad\text{or}\quad u_t= \textbf{1}\},
\end{equation}
with the convention that $\inf\{\emptyset\}=\infty$. while fixation is inevitable, at any given time there is a positive chance that this has not yet happened - indeed fixation may typically take a long time and the probability distribution of the system may stabilize near a  ``quasi-stationary distribution" (QSD) for a long time before fixation.
In what appears to be the first reference to the notion of a QSD in the literature, Wright posited in 1931 \cite[p.111]{Wright1931} that:
\begin{quote}
``As time goes on, divergences in the frequencies of factors may be expected to increase more and more until at last some are either completely fixed or  completely lost from the population. The distribution  curve of gene frequencies should,  however, approach  a definite  form if the genes  which have been  wholly fixed or lost are left out of consideration.''
\end{quote}
This limiting ``definite form'' is what is today referred to as a \textit{quasi-stationary distribution} (actually \textit{quasi-limiting distribution}, to be pedantic). The first of our main results, Theorem \ref{T:main1}, establishes that the stochastic FKPP conditioned on non-fixation at that time converges to  a unique quasi-stationary distribution, as time tends to infinity. 

\medskip
\noindent
{\bf QSD and fixation time for the classical Wright Fisher diffusion. }
The mathematical study of QSDs started with the foundational work of Yaglom on subcritical Galton-Watson processes \cite{Yaglom1947};  see the surveys \cite{van2013quasi,meleard2012quasi} and the book \cite{collet2013quasi} for background on QSDs. 
The quasi-stationary and quasi-limiting behaviours of finite-dimensional diffusions is now well-understood \cite{champagnat2023general}; for completeness we provide a short proof that the Wright-Fisher diffusion converges to its unique quasi-stationary distribution in the appendix (see Proposition \ref{prop:WF diffusion convergence to QSD}). Whereas this proof is based on a standard spectral theory argument, which may be applied broadly to finite-dimensional diffusions, it will be clear that such an argument \textit{fails in the infinite-dimensional setting} of the stochastic FKPP \eqref{fkpp_X}.

In the next few paragraphs, we give a brief account of the classical $1$-dimensional Wright-Fisher diffusion, 
\begin{equation}\label{1dWF}
dX_t=\beta X_t(1-X_t)dt+\sqrt{\gamma\,X_t(1-X_t)}\,dB_t,
\end{equation}
where $B$ is the standard Brownian motion in $\R$, which is the spatially well-mixed version of the stochastic FKPP (Lemma \ref{L:ODE}). We also provide some rigorous results in Section \ref{SS:1-dim}. For example, from Proposition \ref{prop:WF diffusion convergence to QSD}, the principal eigenvalue gives the rate of fixation, whereas the spectral gap gives the rate of convergence to a QSD. 
In \cite{Ewens1963,Ewens1964}, Ewens considered this Wright-Fisher diffusion, assuming that at fixation the process is returned to some arbitrary state. He examined the stationary distribution of this returned process, which he referred to as a ``pseudo-transcient distribution''. The quasi-stationary distribution of the Wright-Fisher diffusion was then examined by Seneta shortly thereafter in \cite{seneta1966quasi}. A brief review of the relationship between the pseudo-transcient distribution and the quasi-stationary distribution can be found in \cite[p.4-5]{van2013quasi}. Applications of QSDs in  genetics are discussed in  \cite{seneta1966quasi} (see the discussion on pages 266-277).

For the neutral Wright-Fisher diffusion  $dX_t=\sqrt{X_t(1-X_t)}dB_t$, it is well-known (see \cite[p.259]{seneta1966quasi}) 
that the principal right eigenfunction is given by $h^{\text{WF}}(x)=6x(1-x)$, whereas the unique QSD is given by the uniform distribution, i.e. $\pi^{\text{WF}}=\text{Unif}((0,1))$. Furthermore, 
the infinitesimal generator of the killed process, denoted by $L^{\text{WF}}$, has a pure discrete spectrum (see Proposition \ref{prop:WF diffusion convergence to QSD}) which is given by \cite[Lemma 3.5]{Tran2013}
\[
\sigma(L^{\text{WF}})=\left\{-{n\choose 2}:n\geq 2\right\}.
\]
Since the principal eigenvalue (namely $-1$) of the generator  is algebraically simple  and the spectral gap is $-2$, it follows from Proposition \ref{prop:WF diffusion convergence to QSD} that 
\begin{equation}\label{AbsTime_1_wellmix}
\Pm_{x}(\tau_{\fix}>t)\sim 6x(1-x)e^{-t}\qquad \text{as }t\to\infty
\end{equation}
for all $x\in (0,1)$,
and that
for all $\delta\in (0,1)$ there exists a constant $C\in(0,\infty)$ and a time $T(x)$ that depends continuously  upon $x\in (0,1)$
 such that 
\begin{align}
\lvert\lvert \Law_{x}(X_t\lvert \tau_{\partial}>t)-\text{Unif}((0,1)) \rvert\rvert_{\TV}\leq \frac{Ce^{-(1+\delta)t}}{x(1-x)},\quad \text{for all}\quad T(x)\leq t< \infty.
\end{align}
Therefore, the conditional law  $\Law_x(X_t\lvert \tau_{\partial}>t)$ converges to the QSD faster than the speed at which $\Pm_x(\tau_{\partial}>t)$ converges to $0$. It follows that the distribution of the number of individuals with a given gene type, on the event of non-fixation, resembles the uniform distribution while there is still a substantial probability of non-fixation. For more on the time to absorption for the Wright-Fisher diffusion,
see \cite[Section 5.4]{ewens2004mathematical} and the references therein.

\begin{remark}\rm
Observe that the spectrum $\sigma(L^{\text{WF}})$ corresponds precisely to the jump rates of Kingman's coalescent, the dual process of the neutral Wright-Fisher diffusion. In light of the duality introduced in Section \ref{S:duality}, we see that this is not a coincidence. In Section \ref{S:duality}, the eigenvalues of the killed stochastic FKPP will be seen to correspond to that of a killed $2$-type branching-coalescing Brownian motion. Applying the same idea to the Wright-Fisher diffusion, one may arrive at the jump rates of Kingman's coalescent. 
\end{remark}

\medskip
\noindent
{\bf QSDs in infinite dimension. }
In contrast to processes in finite dimensions, very little is known about convergence to a quasi-stationary distribution for stochastic PDEs, or infinite-dimensional processes more generally. 

Superprocesses under various conditionings - most commonly conditioning on survival for all time - have been widely studied, for instance \cite{Overbeck1993,Etheridge2003,Champagnat2008}. Convergence to a quasi-stationary distribution for subcritical superprocesses conditioned on survival was recently established in \cite{Liu2021}. In a recent preprint \cite{adams2022quasi}, Adams established convergence to a quasi-stationary distribution for reaction-diffusion equations perturbed by additive cylindrical noise. Although both \cite{adams2022quasi} and the present paper deal with reaction-diffusion type equations, the results are disjoint and the proof strategies are totally different; the proof in \cite{adams2022quasi} proceeds by spectral arguments whereas the present article relies instead on moment duality. This is a necessary consequence of the difference in the noise term. The quasi-stationary behavior of the subcritical contact process on $\Z^d$ has been examined by a number of authors in \cite{Ferrari1996,Sturm2014,Andjel2015,arrejoria2020quasi}. Collectively, they establish convergence to a unique QSD modulo spatial translation for all initial conditions. In particular, it is shown in \cite{arrejoria2020quasi} that uniqueness of QSD holds even though the process does not ``come down from infinity". This is in contrast to the equivalence between the uniqueness of QSDs and  “coming down from infinity” for birth and death processes \cite{van1991quasi, bansaye2016speed}, and some finite dimensional diffusions \cite[Theorem 7.3]{cattiaux2009quasi}; see also \cite{champagnat2016exponential, champagnat2023general} for some general results relating the existence and uniqueness of QSDs to the speed at which the process comes down from infinity.

We also note that another type of Yaglom limit has recently been considered by a number of authors in \cite{Powell2019,Harris2022,Maillard2022} for critical branching processes with absorption. As a result of the criticality, the number of particles alive at time $t$, conditional on the particles not having died out, grows to infinity as $t\ra\infty$. They obtain various scaling limits for the distribution of the particles conditional on survival. These results are analogous to classical results of Yaglom \cite{Yaglom1947} on critical Galton-Watson processes. To the authors' knowledge, the above constitutes the extent of the literature on quasi-limiting behaviour of Markov processes in infinite dimensions. 


\medskip
\noindent
{\bf Contributions. }
We establish the existence and uniqueness of the quasi-stationary distribution (QSD) for the solution of the stochastic FKPP, and show that this QSD is the attractor for all initial distributions.
Moreover, we characterize the leading-order asymptotics for the tail distribution of the fixation time and obtain some explicit calculations in the neutral case  $\beta=0$, yielding insight into the effect of spatial diffusion on fixation time. For example, in \eqref{asym_fixRate} we show that the fixation rate increases to $\gamma$ as $\alpha\uparrow\infty$ and decreases to 0 as  $\alpha \downarrow 0$, asymptotically like $\gamma-\frac{\gamma^2}{12 \alpha}+\mathcal{O}(\frac{\gamma^3}{\alpha^2})$ and $\pi^2 [\alpha-8\frac{\alpha^2}{\gamma}]+\mathcal{O}(\frac{\alpha^3}{\gamma^2})$ respectively. In particular, the  fixation rate of the stochastic FKPP converges  to that  of the well-mixed case as the diffusion constant $\alpha\uparrow\infty$. See  Remark \ref{Rk:insight} for more detail and Figure \ref{fig:lambda} for an illustration.

Our proof relies on the observation that the killed stochastic FKPP is dual to a system of $2$-type branching-coalescing Brownian motions killed when one type dies off.
Proving existence and uniqueness of QSD for either process for all $\beta\in\R_{\ge 0}$ is challenging since both processes are in infinite dimensional setting. 
The key innovation in our proof is that we first
obtain the QSD for the  $2$-type coalescing Brownian motions for the case $\beta=0$ (i.e. without branching, for which the particle system is finite-dimensional) and then level up to the desired results for all $\beta\in\R_{\ge 0}$
by leveraging the relationship between the two killed processes. We expect that this proof strategy can be applied to study the QSDs of other infinite dimensional systems, in particular those stochastic PDEs which possess  moment duality established in \cite{MR1813840}.
Of independent interest, we establish while proving our main results that both the stochastic FKPP $u$ and the killed (or absorbed) process $u^{\rm kill}$ possess the Feller property. 

Our results and approach can readily be generalized to a bounded interval, or more generally to any compact metric space on which  existence of a mild solution is known.

\medskip
\noindent
{\bf Organization. } In Section \ref{S:MainResults}, we will give the rigorous statements of our main results for the stochastic FKPP. We will then describe the duality for the killed stochastic FKPP and the quasi-stationary behavior of the dual process in Section \ref{S:duality}, which shall be crucial to our proofs. This will then offer new insights into the the stochastic FKPP, including a characterization both of its QSD and of the tail distribution of the fixation time. We shall then, in Section \ref{S:Insights}, obtain explicit calculations in the neutral case. We will then turn to our proofs. An overview of the proof of our main theorems shall be given in Section \ref{S:Idoof}, highlighting in particular how we take advantage of the duality between the killed stochastic FKPP and a killed 2-type BCBM. This will then be followed by the proofs of all stated results in Section \ref{S:Proofs}. Finally, we will conclude with the appendix.




\section{Main results for stochastic FKPP}\label{S:MainResults}


To state our main results precisely, we need some notation. Given a separable metric space $E$ and a set $A\subset \R$, we write $\calC(E;A)$ and $\calB(E;A)$ for the spaces of continuous functions and Borel functions respectively from $E$ to $A$. 
We use a subscript ``$b$" to denote $\calC_b(E;A)$ and $\calB_b(E;A)$ for the spaces of bounded continuous functions and bounded Borel functions respectively, both equipped with the uniform norm. For $f\in\calB_b(E;A)$ we denote by  $\|f\|_{\infty}:=\sup_{x\in E}|f(x)|$ the uniform norm. We let $\calP(E)$ be the space of probability measures on $E$, with the topology of weak convergence of measures. For a suitable pair of measure $\mu$ and  function $H$ on a space $E$, we write $\mu(H):=\langle \mu,\,H\rangle_{E}:=\int_{E}H(\cdot) \mu(d\cdot)$. We write ${\rm d}_{\S}(x,y)$ for the geodesic distance between $x,y\in\bfS=\Rm/\Zm$.

\medskip

It is known  that  for each initial condition $u_0\in \mathcal{B}(\mathbb{S};\,[0,1])$, equation \eqref{fkpp_X} has a mild solution
$u=(u_t)_{t\in\R_{\geq 0}}$ that is unique in distribution and satisfies $u\in \mathcal{C}((0,\infty)\times \bfS;[0,1])$ almost surely. 
From now on, $(u_t)_{t\in\R_{\geq 0}}$ denotes such a solution defined on a probability space $(\Omega, \cal{F}, \P)$.
These solutions also solve the martingale problem associated to \eqref{fkpp_X}. Therefore they define a strong Markov process on $ \mathcal{B}(\mathbb{S};[0,1])$, which belongs to $\mathcal{C}(\mathbb{S};[0,1])$ at all strictly positive times, almost surely. 
These facts can be found, for instance, in \cite{MR1271224} and \cite[Remark 1]{hobson2005duality}.  For each $\mu\in \calP(\mathcal{B}(\mathbb{S};\,[0,1]))$ we let $\P_{\mu}$ be the probability measure under which the initial state $u_0$ has distribution $\mu$. For each $f\in \mathcal{B}(\mathbb{S};\,[0,1])$
we let $\P_{f}$ be the probability measure under which the initial state is $u_0=f$.

Recall the \textbf{fixation time} $\tau_{\fix}$ defined in \eqref{Def:taufix}. Note that  $\tau_{\fix}=0$ almost surely under $\P_{u_0}$ if $u_0\in \bar{\textbf{0}} \cup \bar{\textbf{1}}$, where
 $\bar{\textbf{0}}$ (resp. $\bar{\textbf{1}}$)
is the subset of Borel functions on $\mathbb{S}$  that are almost everywhere  constant with value 0 (resp. 1). 
Hence we omit $\bar{\textbf{0}} \cup \bar{\textbf{1}}$ from the possible initial conditions, and consider
$$\calB_{\ast}:=\mathcal{B}(\mathbb{S};[0,1])\setminus (\bar{\textbf{0}} \cup \bar{\textbf{1}})\,=\,\left\{f\in \mathcal{B}(\mathbb{S};[0,1]):\,\int_{\mathbb{S}}f(x)m(dx)\in \{0,\,1\}\right\}.$$  

The \textbf{killed (or absorbed) stochastic FKPP}   is the process $u^{\rm kill}$ defined  by 
\begin{equation}\label{fkpp_X_kill}
u^{\rm kill}_t=
\begin{cases}
u_t, & \quad \text{if }t<\tau_{\fix}\\
\Delta, & \quad \text{if }t\geq \tau_{\fix},
\end{cases},
\end{equation}
where $\Delta$ is a separate isolated cemetery state and any measurable function is extended to be 0 at $\Delta$.
Since we are only interested in the behaviour prior to fixation, we will abuse notation by writing the killed stochastic FKPP as $(u_t)_{0\leq t<\tau_{\fix}}$.  


A Borel probability measure $\pi\in \calP(\calB_{\ast})$ is a \textbf{quasi-stationary distribution (QSD)}  for $(u_t)_{0\leq t<\tau_{\fix}}$ if
\begin{equation}\label{Def:QSD}
\P_{\pi}(u_t \in \,\cdot\;|\,\tau_{\fix}>t)\,=\,\pi(\cdot)\quad \text{for}\quad t\geq 0.
\end{equation}



Where the killing time is unambiguous, we shall sometimes, for brevity, refer to the ``QSD of a process'' rather than the ``QSD of a killed process''. For all strictly positive times  and prior to fixation (i.e. for $t\in(0,\tau_{\rm fix})$), $u_t$ takes values in  
\[
\calC_{\ast}:=\mathcal{C}(\mathbb{S};[0,1])\setminus \{{\textbf{0}},{\textbf{1}}\}\,=\,\left\{f\in \mathcal{C}(\mathbb{S};[0,1]):\,0<f(x)<1 \text{ for some }x\in\mathbb{S}\right\}.
\]
Hence a QSD $\pi$ is supported on $\calC_{\ast}$ if it exists; see Remark \ref{RK:B_ast}. Our main results, Theorems \ref{T:main1} and \ref{thm:fixation time},
hold for any fixed constants $\alpha\in(0,\infty)$, $\beta\in\R$ and $\gamma\in(0,\infty)$.
\begin{thm}\label{T:main1}
The stochastic FKPP equation \eqref{fkpp_X} has a unique quasi-stationary distribution $\pi\in \calP(\calC_{\ast})$. Furthermore, we have the  convergence 
\begin{equation}\label{E:main1}
\P_{\mu}(u_t \in \,\cdot\;|\,\tau_{\fix}>t)\,\to \,\pi(\cdot)\quad\text{in}\quad \calP(\calC_{\ast}) \quad \text{as}\quad t\to\infty,
\end{equation}
for any initial distribution $\mu\in \calP(\calB_{\ast})$.
\end{thm}

Our next theorem characterises the leading-order asymptotics of the fixation time as $t\to \infty$. We must firstly introduce some notation and background. 
For $t\in\R_{\geq 0}$, we define
\begin{equation}\label{eq:submarktranssemigroupFKPP}
\begin{split}
&P_t(f,\cdot):=\,\Pm_{f}(u_t\in \cdot,\tau_{\fix}>t),\qquad  f\in \calB_{\ast},\\
&\mu P_t(\cdot):=\,\Pm_{\mu}(u_t\in \cdot,\tau_{\fix}>t),\qquad \mu\in \calP(\calB_{\ast}),\\ 
&P_tF(f):=\,\expE_{f}[F(u_t)\Ind(\tau_{\fix}>t)],\quad F\in \calB_b(\calB_{\ast}).
\end{split}
\end{equation}
Then $\{P_t:\, t\in\R_{\geq 0}\}$ are sub-Markovian transition kernels, with $(P_t)_{t\geq 0}$ being a sub-Markovian transition semigroup; it provides the transition semigroup for the killed stochastic FKPP $(u_t)_{0\leq t<\tau_{\fix}}$. We note that $\mu P_t(F) = \mu (P_t F)$. We also note that $(P_t)_{t\geq 0}$ is only dependent upon the stochastic FKPP prior to fixation.
Clearly, a measure $\pi\in\calP(\calB_{\ast})$ is a quasi-stationary distribution for the stochastic FKPP if and only if it is a \textbf{left eigenmeasure} of $(P_t)_{t\geq 0}$ with a positive eigenvalue, in the sense that there exists $\lambda>0$ such that $\pi P_t=\lambda^t\pi$ for all $t>0$. In this case we refer to $\lambda$ as the eigenvalue of $\pi$ and write it as $\Lambda(\pi)$ - it is the principal eigenvalue of $P_1$.
Under a QSD $\pi$, $\tau_{\rm fix}$ is an exponential variable and 
\begin{equation}\label{lambda_t}
\Pm_{\pi}(\tau_{\fix}>t)=\lambda^t=e^{-\kappa\,t} \quad\text{for }t\in(0,\infty),
\end{equation}
where $\kappa:=-\ln \lambda$ is called the \textbf{fixation rate}.
In particular, $\lambda=\Pm_{\pi}(\tau_{\fix}>1)>0$. These and other general facts about QSD can be found in \cite{meleard2012quasi}.

We similarly define a bounded, non-negative Borel function on $\calB_{\ast}$, $h\in \calB_b(\calB_{\ast};\Rm_{\geq 0})$, to be a \textbf{right eigenfunction} for $(P_t)_{t\geq 0}$ 
with positive eigenvalue $\lambda>0$ if $P_th(f)=\lambda^th(f)$ for all $f\in \calB_{\ast}$, $t\geq 0$. 
We write  $\Lambda(h)$ for the eigenvalue of a right eigenfunction $h$. Therefore, 
\begin{equation}\label{eq:eigentriple stochastic FKPP}
\pi P_t=(\Lambda(\pi))^t\pi,\quad P_th=(\Lambda(h))^th,\quad t\geq 0.
\end{equation}

\begin{remark}[Feller property versus Feller semigroup]\label{Rk:eigenvalue is over time 1}\rm
In  Proposition \ref{prop:Feller_FKPP}, we show that the semigroup $(P_t)_{t\geq 0}$ possesses the Feller property on $\calC_b(\calC_{\ast})$. However, we also note that $(P_t)_{t\geq 0}$ is not strongly continuous on $\mathcal{C}_b(\calC_{\ast})$, hence is not a Feller semigroup. Thus, whereas it is typical to define the eigenvalue associated to a right eigenfunction and to a QSD to be the eigenvalue with respect to the infinitesimal generator, it's not clear here that we have an infinitesimal generator defined on a dense subspace of $\mathcal{C}_b(\calC_{\ast})$. Thus we define $\Lambda(\pi)$ to be the eigenvalue of $\pi$ with respect to $P_1$, whereas it is more typical in the literature to instead use the eigenvalue with respect to the infinitesimal generator.
\end{remark}

Theorem \ref{thm:fixation time} below asserts that there exists a unique right eigenfunction, with eigenvalue being the same as that of the unique QSD. Moreover, the eigenpair determines the asymptotic behavior of the fixation time. 
\begin{thm}[Right eigenpair and fixation time]\label{thm:fixation time}
The following hold for the stochastic FKPP equation \eqref{fkpp_X}:
\begin{itemize}
\item[(i)]There exists  a  right eigenfunction $h\in \mathcal{C}_b(\calB_{\ast};\Rm_{>0})$ for $(P_t)_{t\in \R_{\ge 0}}$, with eigenvalue $\Lambda(h)=\Lambda(\pi)\in (0,1)$. 
Moreover $h$  is the unique (up to constant multiple) right eigenfunction for $(P_t)_{t\in \R_{\ge 0}}$ in $\mathcal{C}_b(\calB_{\ast};\Rm_{>0})$; and the restriction $h_{\lvert_{\calC_{\ast}}}$
is the unique (up to constant multiple) right eigenfunction for $(P_t)_{t\in \R_{\ge 0}}$ in $\calC_b(\calC_{\ast};\Rm_{\geq 0})$.
\item[(ii)]
The right eigenfunction $h$ and the eigenvalue $\lambda:=\Lambda(h)$ give the leading-order asymptotics of the fixation time in the sense that
\begin{equation}\label{AbsTime_1}
    \lambda^{-t}\Pm_{\mu}(\tau_{\fix}>t)\rightarrow \frac{\mu(h)}{\pi(h)}\quad\text{as}\quad t\rightarrow \infty\quad\text{for all}\quad \mu\in\calP(\calB_{\ast}).
\end{equation}
Moreover, we  have the lower bound
\begin{equation}\label{eq:lower bound on fixation probability}
    \Pm_{\mu}(\tau_{\fix}>t)\geq \frac{\mu(h)}{\lvert\lvert h\rvert\rvert_{\infty}}\lambda^t \qquad \text{for all}\quad t\in \R_{\ge 0},\quad \mu\in\calP(\calB_{\ast}),
\end{equation}
and the upper bound
\begin{equation}\label{eq:upper bound on fixation probability}
    \Pm_{\mu}(\tau_{\fix}>t)\leq C\lambda^t \qquad \text{for all}\quad t\in \R_{\ge 0},\quad \mu\in\calP(\calB_{\ast}),
\end{equation}
for some uniform constant $C\in (0,\infty)$ which does not depend upon $\mu\in\calP(\calB_{\ast})$ nor $t\in \R_{\ge 0}$.
\end{itemize}
\end{thm}

\begin{remark}[Discontinuous initial conditions]\rm \label{RK:B_ast}
We allow the initial condition of the stochastic FKPP to belong to the larger space $\calB_{\ast}$ in Theorems \ref{T:main1} and \ref{thm:fixation time}, which is desirable  because  discontinuous functions (like step functions) have been used as the initial condition in the literature and in simulations. Under the uniform norm, $\calC_{\ast}$ is a separable, closed subset of $\calB_{\ast}$.  
As mentioned,  $\P_{u_0}(\tau_{\fix}=0)=1$  if $u_0\in \bar{\textbf{0}} \cup \bar{\textbf{1}}$. On the contrary, if  $u_0\in \calB_{\ast}$, then (i) 
 $\P_{u_0}(\tau_{\fix}>0,\;u_s\in \calC_{\ast} \,\forall s\in (0,\tau_{\fix}))=1$
 and (ii) for all $t>0$, we have $\P_{u_0}(\tau_{\fix}>t)>0$ and $\P_{u_0}(u_s\in \calC_{\ast} \,\forall s\in (0,t]\,\lvert\, \tau_{\fix}>t)=1$, 
In particular, $\Law_{u_0}(u_t\lvert \tau_{\fix}>t)\in \calP(\calC_{\ast})$ for all $t>0$ and $u_0\in \calB_{\ast}$. This implies that the QSD $\pi$ is supported on $\calC_{\ast}$ if it exists.
 \end{remark}

\begin{remark}[Fixation time]\rm \label{Rk:fixationTime}
By \eqref{lambda_t}, under the QSD $\pi$, $\tau_{\rm fix}$ is an exponential variable with rate $\kappa$. For an arbitrary initial distribution $\mu\in \calP(\calB_{\ast})$, the upper bound \eqref{eq:upper bound on fixation probability} implies that all moments of $\tau_{\fix}$ under $\Pm_{\mu}$ are finite. In particular, $\Pm_{\mu}(\tau_{\fix}<\infty)=1$ for all $\mu\in \calP(\calB_{\ast})$. 
\end{remark}

\begin{remark}[Convergence/non-convergence in total variation]\rm
We do not know if the convergence in weak topology in \eqref{E:main1} can be strengthened to convergence in total variation norm, as in Theorem \ref{thm:conv to QSD dual}. However, note that in the SPDE setting, we should not typically expect to obtain convergence in total variation, roughly speaking because distinct measures in infinite-dimensional spaces are typically mutually singular \cite[Section 4.2]{hairer2009introduction}. Finally, we note that \eqref{E:main1} and \eqref{AbsTime_1} imply that  for all $\mu\in\calP(\calB_{\ast})$, 
\begin{equation}\label{Conv_subMkv_FKPP}
    e^{-\kappa_0t}\mu P_t(\cdot)=e^{-\kappa_0t}\Pm_{\mu}(u_t\in \cdot\;,\;\tau_{\fix}>t)\to \frac{\mu(h)}{\pi(h)}\,\pi(\cdot) \quad \text{in } \mathcal{M}_{+}(\calC_*)\quad\text{as}\quad t\rightarrow \infty,
\end{equation}
where $\mathcal{M}_{+}(\calC_*)$ is the space of finite non-negative measures on $\calC_*$, equipped with the weak topology.
\end{remark}

\section{Duality for killed processes}\label{S:duality}

While Theorems \ref{T:main1} and \ref{thm:fixation time}
hold for any  $\beta\in\R$, the $\beta<0$ case immediately follows from the $\beta>0$ case by considering $v:=1-u$, so we may without loss of generality assume that $\beta\in \R_{\geq 0}$. 

The stochastic Fisher-KPP (when $\beta\in \R_{\geq 0}$) enjoys a  moment duality relationship with a system of \textbf{branching coalescing Brownian motions  (BCBM)}.  In this particle system, each particle performs an independent Brownian motion on the circle $\mathbb{S}$ at rate $\alpha$ up to its lifetime, each particle splits into two at rate $\beta$, and 
every (unordered) pair of particles $\{i,j\}$ (where $i\neq j$) coalesce independently at rate $\frac{\gamma}{2\alpha}$ according to their intersection local time $L^{(i,j)}=(L^{(i,j)}_t)_{t\in\R_{\ge 0}}$. For $t\in\R_{\ge 0}$, $L^{(i,j)}_t$ is defined as the local time at 0 of the process $s\mapsto  {\rm d}_{\mathbb{S}}(X^i_s,X^j_s)$ up to time $t$, where ${\rm d}_{\mathbb{S}}$ is the geodesic distance on $\mathbb{S}=\Rm/\Zm$; see 
\eqref{localtime_formal}-\eqref{aslimit_localtime}. At the coalescence time, one of the two particles dies. 

Let $\mathcal{I}_t$ be the set of indices of particles alive at time $t$, and $\bar{X}_t := \{X^a_t:\;a\in \mathcal{I}_t\}$ be the multi-set of their spatial locations (i.e. counting multiplicities and ignoring order). It holds that
 \begin{equation}\label{WFdual}
	\E_{u_0}[D(u_t;\,\bar{x})] = {\bf E}_{\bar{x}}[D(u_0;\,\bar{X}_t)]
	\end{equation}
	for all $t\in\mathbb{R}_{\geq 0}$ and for all initial conditions $u_0\in \mathcal{B}(\mathbb{S};\,[0,1])$  and $\bar{X}_0=\bar{x}\in \mathbb{S}^n/\sim$, where
	\begin{equation}\label{Def:DualFcn}
	  D(f; \bar{x}):= \prod_{i=1}^n \left(1-f(x_i)\right)
	\end{equation}
whenever $\bar{x}=\{x_1,x_2\cdots,x_n\}\in \mathbb{S}^n/\sim$ is a multi-set of $n$ points on the circle equivalent up to permutation of indices,  $\E_{u_0}$ is the  expectation under which $u$ starts at $u_0$, and ${\bf E}_{\bar{x}}$ is the expectation under which $\bar{X}$ starts at $X^{(i,0)}_0=x_i$ for $1\leq i\leq n$. The duality relation \eqref{WFdual} was first stated in \cite{shiga1988stepping}, and  proved in \cite{MR1813840,durrett2016genealogies}. This relation,
together with the existence \cite{athreya1998probability, MR1813840} 
of the BCBM imply weak uniqueness  of the stochastic FKPP; see \cite[Lemma 1]{MR1813840}.

\medskip

Our proofs of  Theorems \ref{T:main1} and \ref{thm:fixation time} shall rely on the key observation that the quasi-stationary properties of the stochastic FKPP  are in some sense dual to the quasi-stationary properties of a 
$2$-type branching-coalescing Brownian motion (BCBM), killed at a  stopping time $\tau_{\partial}$ which we introduce in Definition \ref{Def:killTime_BCBM} below. This connection, based on the duality functions $\{\calE^z\}_{z\in\chi}$ to be introduced in \eqref{eq:F for duality relationship}, will allow us to characterise the principal eigentriple $(\pi,\lambda,h)$ of the stochastic FKPP in terms of that of this killed $2$-type BCBM. 

We expect that the proof strategy developed in this paper may be applied to study the QSDs of other infinite dimensional systems, in particular those stochastic PDEs which possess the moment duality established in \cite{MR1813840}.

\medskip

\subsection{Killed $2$-type moment dual}\label{SS:killedDual}

We consider a BCBM in which each particle is given one of two colors, green and red, together with an index in the sets $\mathbb{N}\times \{\rm green\}$ and $\mathbb{N}\times \{\rm red\}$ respectively, at the beginning and at birth (due to branching). The color of a particle stays the same throughout its lifetime and is the same as its parent. When a coalescence event occurs for a pair of particles with two different colors, the red particle disappears while the green particle stays alive.

\begin{definition}[$2$-type branching coalescing Brownian motions]\rm \label{Def:2BCBM}
Let $(\calG_t)_{t\geq 0}$ and $(\calR_t)_{t\geq 0}$ be the index sets of green  particles and red particles respectively, which are alive at each time $t\geq 0$. They are c\`adl\`ag processes taking values in the space of finite subsets of $\mathbb{N}\times \{\rm green\}$ and $\mathbb{N}\times \{\rm red\}$ respectively.
For each $t\in\R_{\ge 0}$ and $i\in \calG_t\cup \calR_t$, we denote by $X^i_t$  the location of the particle with index $i$ at time $t$. The process $\{X^i\}$ evolves according to the following Markovian dynamics:
\begin{enumerate}
    \item Between its birth time and the time it is killed, each $X^i_t$ evolves as an independent Brownian motion on $\mathbb{S}$ of rate $\alpha$, so that $dX^i_t=\sqrt{\alpha}dB^i_t$ for some independent Brownian motion $B^i_t$.
    \item Each particle $X^i_t$ gives birth at rate $\beta$, to a child which has the same colour and is born at the same location. Thus, if $i\in \calG_{t-}$ (respectively $i\in \calR_{t-}$) gives birth to a child, we add a new index $j$ to $\calG_t$ (respectively $\calR_t$), and define $X^j_t:=X^i_{t-}$.
    \item For $i,j\in \calG_t\cup \calR_t$ ($i\neq j$) we denote by $L^{(i,j)}_t$  the intersection local time of $X^i$ and $X^j$ accumulated during the interval $[0,t]$. 
        Then every (unordered) pair of particles $\{X^i,\,X^j\}$ coalesce with rate $\frac{\gamma}{2\alpha}$ according to their intersection local time as follows:
\begin{itemize}
\item[(a)] If $i\in \calG_t$, then $X^i$ is killed at rate 
    \[
    \frac{\gamma}{4\alpha} \sum_{\substack{j\in \calG_t\\j\neq i}}dL^{(i,j)}_t. 
    \]
\item[(b)]  If $i\in\calR_t$, then $X^i$ is killed at rate 
    \[
     \frac{\gamma}{4\alpha} \sum_{\substack{j\in \calR_t\\j\neq i}}dL^{(i,j)}_t+\frac{\gamma}{2\alpha} \sum_{\substack{j\in \calG_t\\j\neq i}}dL^{(i,j)}_t. 
    \]
\end{itemize} 
\end{enumerate}
\noindent
We let $G_t:=\{X^i_t:i\in \calG_t\}$ and $R_t:=\{X^i_t:i\in \calR_t\}$ be respectively the sets of locations of the green particles  and the red particles 
at time $t$. We let $Z_t:=(G_t, R_t)$ and call the process $(Z_t)_{t\in \R_{\geq 0}}$  the \textbf{2-type BCBM} in this paper. 
 \end{definition}

\smallskip
 
The Markov process $(Z_t)_{t\in \R_{\geq 0}}$ exists by \cite{athreya1998probability}. Clearly, $(G_t\cup R_t)_{t\in \R_{\ge 0}}\eqd (\bar{X}_t)_{t\in \R_{\ge 0}}$, where $\bar{X}_t$ is the set of locations of all the particles in \eqref{WFdual}. That is, ignoring the color reduces the system to the usual 1-type BCBM. In particular, the total coalescence rate of the 2-type BCBM is the same as that of the usual 1-type BCBM (and is given by \eqref{TotalCoaRate}).

Our new observation here is that the quasi-stationary behavior of the stochastic FKPP is intimately related to that of the 2-type BCBM prior to the stopping time $\tau_{\partial}$ defined in Definition \ref{Def:killTime_BCBM} below.
We declare this $2$-type particle system to be ``killed'' when there are no more red particles, and
consider the QSD of this $2$-type particle system prior to be the ``killing time" $\tau_{\partial}$.
For $t\geq \taud$, there are no red particles while the green particles continue to evolve as a system of branching-coalescing Brownian motions. See \textbf{Figure \ref{Fig:2BCBM}} for an illustration.

\FloatBarrier

\begin{figure}
\centering
\includegraphics[scale=0.5]{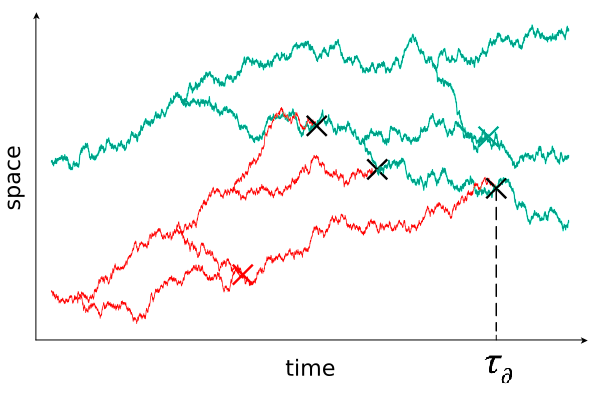}
\caption{A trajectory of the 2-type branching-coalescing Brownain motions (2-type BCBM) starting with one green particle and one red particle, where $\tau_{\partial}$ is the first time when all red particles die off. 
}
\end{figure}\label{Fig:2BCBM}


Indeed, since red particles cannot reappear once they disappear, $\{(G,R):R=\emptyset\}$ is a cemetery set. The $2$-type BCBM with the absorption time $\taud$ therefore defines an absorbed (or killed) Markov process, which we call the \textbf{killed $2$-type BCBM} and denote as $(Z_t)_{0\leq t<\taud}$. We will always assume that the 2-type BCBCM starts with at least one green particle and at least one red particle. Then 
the killed 2-type BCBM prior to killing has state space 
\begin{equation}\label{eq:state space for BCBM}
\chi:=(\cup_{n\geq 1}\bfS^n/\sim)\times(\cup_{m\geq 1}\bfS^m/\sim),
\end{equation}
where $\sim$ is the equivalence relationship on $\bfS^n$ such that
$x=(x_1,\ldots,x_n)\sim x'=(x'_1,\ldots,x'_n)$ if $x$ can be obtained from $x'$ by permuting the coordinates.

\begin{definition}\rm \label{Def:killTime_BCBM}
We consider the first time that all the red particles are killed, namely
\[
\tau_{\partial}:=\inf\{t\geq 0:\, R_t=\emptyset \},
\]
and we call it the \textbf{killing time} of the 2-type BCBM. This then defines the killed $2$-type BCBM $(Z_t)_{0\leq t<\taud}$. A Borel probability measure, $\varphi\in \calP(\chi)$, is called a \textbf{quasi-stationary distribution (QSD)} for the killed $2$-type BCBM $(Z_t)_{0\leq t<\taud}$  if
\begin{equation}
{\bf P}_{\varphi}(Z_t \in \,\cdot\;|\,\taud>t)\,=\,\varphi(\cdot)\quad \text{for}\quad t\in\R_{\geq 0}.
\end{equation}
 \end{definition}

\medskip

Next we will define a family of functions $\{\calE^z\}_{z\in\chi}$ which will serve as the dual functions between the killed processes, and which is large enough to characterise measures on $\calC_{\ast}$.
For $f\in \calB_{\ast}$ and $z=(x,y)=((x_1,\ldots,x_n),(y_1,\ldots,y_m))\in \chi$, we define
\begin{align}
    \calE(f,z):=&\,D(f; x)-D(f; (x,y))
   =\, \Big[\prod_{i=1}^n(1-f(x_i))\Big]\Big[1-\prod_{j=1}^m(1-f(y_j))\Big], \label{eq:F for duality relationship}
\end{align}
where the function $D(f; \bar{x}):= \prod_{i=1}^n \left(1-f(x_i)\right)$ was defined in \eqref{Def:DualFcn}.
We further define 
\[
\calE^{z}(f):=\calE(f,z).
\]

We will apply Lemma \ref{lem:moments determine measures} below  to use  $\{\calE^z\}_{z\in\chi}$ to characterize the QSD of the stochastic FKPP. 
\begin{lemma}\label{lem:moments determine measures} 
Suppose that $\mu_1$ and $\mu_2$ are finite non-negative measures on $\calC_{\ast}$ such that 
$\mu_1(\calE^{z})=\mu_2(\calE^{z})$ for all $z\in \chi$. Then $\mu_1=\mu_2$.
\end{lemma}
The proof of this self-contained lemma is given in the Appendix.

\begin{prop}(Duality between killed processes.) \label{prop:killDual}
Let $\alpha,\gamma\in(0,\infty)$ and $\beta\in\R_{\geq 0}$ be fixed constants. Let $u$ be the  stochastic FKPP and $Z$ be the  2-type BCBM corresponding to these constants. It holds that for all $f\in \calB_{\ast}$, $z\in \chi$ and $t\in \R_{\geq 0}$, 
    \begin{equation}\label{eq:moment duality relationship for 2-type results}
        \expE_{f}[\calE^z(u_t)\Ind(\tau_{\fix}>t)]={\bf E}_{z}[\calE^{\,Z_t}(f)\,\Ind(\taud>t)],
    \end{equation}
where  $\calE:\,\calB_{\ast}\times \chi\to [0,1]$ is  the function defined  in \eqref{eq:F for duality relationship}.
    In other words, $(P_t\calE^{z})(f)=\left(Q_t\calE^{\bullet}(f)\right)(z)$, where $\{Q_t\}$ is the 
    sub-Markovian transition semigroup of the killed 2-type BCBM,
    acting on the function $z\mapsto \calE^{z}(f)$, and $\{P_t\}$ is the 
    sub-Markovian transition semigroup of the killed stochastic FKPP.
\end{prop}

\begin{proof}[Proof of Proposition \ref{prop:killDual}]
We take $u_0\in \mathcal{B}(\mathbb{S};\,[0,1])$ and $z\in \chi$ fixed and arbitrary. The process $(Z_t)_{t\in\R_{\ge 0}}$ is the usual ($1$-type) system of BCBMs starting from $z$, if we ignore the color. Similarly we define $(u_t)_{t\in \Rm_{\ge 0}}$ to be the stochastic FKPP, defined for all time (so defined for $t\geq \tau_{\fix}$). Therefore we have 
\begin{equation}\label{Dual2type_1}
\E_{u_0}[D(u_t,\,z)] 
\stackrel{ (\ref{WFdual})}{=}
{\bf E}_{z}[D(u_0,\,Z_t)].
\end{equation}
On the other hand, the subset $G_t$ also constitutes a ($1$-type) system of BCBMs that is \textit{not affected by the red particles}, hence
\begin{equation}\label{Dual2type_2}
\E_{u_0}[D(u_t,\,g)]
\stackrel{ (\ref{WFdual})}{=}
{\bf E}_{g}[D(u_0;\,G_t)] = {\bf E}_{(g,r)}[D(u_0;\,G_t)].
\end{equation}
Subtracting \eqref{Dual2type_1} from \eqref{Dual2type_2} and recalling \eqref{eq:F for duality relationship}, we obtain that for all $z\in \chi$ we have
\begin{equation}\label{eq:2-type duality without stopping time yet}
    \E_{u_0}[\calE(u_t,\,z)]=  {\bf E}_{z}[\calE(u_0;\,Z_t)].
\end{equation}

If $t\geq \taud$, then $\calE(u_0,Z_t)=0$ for all $u_0$ because $G_t=Z_t$. On the other hand, if $t\geq \tau_{\fix}$, then  $\calE(u_t,z)=0$ for all $z\in\chi$ since $u_t\equiv 0$ or $u_t\equiv 1$.  Therefore, \eqref{eq:moment duality relationship for 2-type results} 
follows from \eqref{eq:2-type duality without stopping time yet}.
\end{proof}

We define the sub-Markovian transition semigroup $(Q_t)_{t\in \R_{\geq 0}}$  
of the killed $2$-type BCBM in the same manner as in the definition for the stochastic FKPP in \eqref{eq:submarktranssemigroupFKPP}, namely 
\begin{equation}
\begin{split}
&Q_t(z,\cdot):=\,{\bf P}_{z}(Z_t\in \cdot,\taud>t),\quad  z\in \chi,\\  
&\mu Q_t(\cdot):=\,{\bf P}_{\mu}(Z_t\in \cdot,\taud>t),\quad \mu\in \calP(\chi),\\ &Q_tf(z):=\,{\bf E}_{z}[f(Z_t)\Ind(\taud>t)],\quad f\in \calB_b(\chi).
\end{split}
\end{equation}
Then as with the stochastic FKPP,  $\varphi\in\calP(\chi)$ is a quasi-stationary distribution for the killed $2$-type BCBM if and only if it is a left eigenmeasure for $(Q_t)_{t\geq 0}$. As with \eqref{eq:eigentriple stochastic FKPP}, we write $\Lambda(\varphi)$ and $\Lambda(\phi)$ respectively for the eigenvalues of a QSD $\varphi$ and a right eigenfunction $\phi$, so that
\begin{equation}
\varphi Q_t=(\Lambda(\varphi))^t\varphi,\quad Q_t\phi=(\Lambda(\phi))^t\phi,\quad t\geq 0.
\end{equation}

Analogously to Theorems \ref{T:main1} and \ref{thm:fixation time} for the stochastic FKPP, we have Theorems \ref{thm:conv to QSD dual} and \ref{thm:right efn of dual} respectively for the dual process.
\begin{thm}\label{thm:conv to QSD dual}
The killed $2$-type BCBM $(Z_t)_{0\leq t<\taud}$ has a unique quasi-stationary distribution $\varphi\in \calP(\chi)$. Furthermore, we have the  convergence
\begin{equation}\label{Converge_BCBM_QSD}
{\bf P}_{\nu}(Z_t \in \,\cdot\;|\,\taud>t)\,\to \,\varphi(\cdot)\quad\text{in total variation}
\end{equation}
as $t\to\infty$, for any initial condition $\nu\in \calP(\chi)$.
\end{thm}

Recall that $\mathcal{C}_b(\chi;\Rm_{>0})$ is the space of
bounded, continuous, everywhere strictly positive function on $\chi$.
\begin{thm}\label{thm:right efn of dual}
There exists $\phi\in \mathcal{C}_b(\chi;\Rm_{>0})$ which is a right eigenfunction of the killed $2$-type BCBM $(Z_t)_{0\leq t<\taud}$ with the same positive eigenvalue as the unique QSD, $\Lambda(\phi)=\Lambda(\varphi)\in (0,1)$. Moreover $\phi$ is the unique (up to constant multiple) right eigenfunction of $(Z_t)_{0\leq t<\taud}$ in $\calB_b(\chi;\Rm_{> 0})$. Furthermore the right eigenfunction $\phi$ and eigenvalue $\lambda:=\Lambda(\phi)$ give the leading-order asymptotics of the killing time for any initial condition, in the sense that
\begin{equation}
\lambda^{-t}{\bf P}_{\nu}(\taud>t)\ra \frac{\nu(\phi)}{\varphi(\phi)}
\end{equation}
as $t\to\infty$, for any initial condition $\nu\in \calP(\chi)$.
\end{thm}

The QSDs and the right eigenfunctions of the two processes (the  stochastic FKPP and the 2-type BCBM) are related as follows. Let $\calE:\,\calB_{\ast}\times \chi\to [0,1]$ be defined in \eqref{eq:F for duality relationship}.
Define the functions
$\Phi:\,\chi\to \R$ and $H:\,\calB_{\ast}\to\R$ by
\label{enum:formula for right efn of dual results}
    \begin{equation}\label{eq:formula for right efn of FKPP in terms of QSD of dual results}
\Phi(z):=\Phi_{\pi}(z):=\int_{\calC_{\ast}}\calE^z(f)\,\pi(df),\qquad z\in\chi
\end{equation}
and
    \begin{equation}\label{eq:formula for right efn of dual in terms of QSD of FKPP results}
   H(f):= H_{\varphi}(f):= \int_{\chi}\calE^z(f)\,\varphi(dz) ,\qquad f\in \calB_{\ast}.
\end{equation}
By Tonelli's theorem,
\begin{equation}\label{eq:duality theorem integral of right efns the same}
\pi(H)=\varphi(\Phi)=\int_{\calC_{\ast}}\int_{\chi}\calE^z(f)\varphi(dz)\pi(df).
\end{equation}

\begin{thm}(QSD and eigenfunction of the dual) \label{theo:duality relationship for quasi-stationarity results}
Let $u=(u_t)_{0\leq t<\tau_{\fix}}$ be the killed stochastic FKPP and  $(Z_t)_{0_{\leq t<\taud}}$ the killed 2-type BCBM corresponding to a given set of constants $\alpha\in(0,\infty)$, $\beta\in\R_{\geq 0}$ and $\gamma\in\R_{\geq 0}$. 
Let $\pi$ and $\varphi$ be the QSDs of $u$ and $Z$ respectively, mentioned in Theorems \ref{T:main1} and \ref{thm:conv to QSD dual}. The following holds:
\begin{enumerate}   
    \item\label{enum:duality theorem right efn BCBM} 
    The function $\Phi$ defined in \eqref{eq:formula for right efn of FKPP in terms of QSD of dual results} belongs to $\calC_b(\chi;\Rm_{>0})$ and is the unique non-negative right eigenfunction (up to constant multiple) of the the killed $2$-type BCBM.
\item The function $H$ defined in \eqref{eq:formula for right efn of dual in terms of QSD of FKPP results} belongs to $\calC_b(\calB_{\ast};\Rm_{>0})$ and is the unique non-negative right eigenfunction (up to constant multiple) of the killed stochastic FKPP.
\item
These QSDs and right eigenfunctions all share the same eigenvalue; that is, 
\begin{equation}\label{eq:duality theorem eigenvalues all the same}
\Lambda(\pi)=\Lambda(H)=\Lambda(\varphi)=\Lambda(\Phi)\in (0,1).
\end{equation}
\end{enumerate}
\end{thm}

Unless otherwise stated, we always normalize the right eigenfunctions by taking $\phi=\Phi$ and $h=H$ defined in \eqref{eq:formula for right efn of FKPP in terms of QSD of dual results} and  \eqref{eq:formula for right efn of dual in terms of QSD of FKPP results} respectively.
We summarize Theorem \ref{theo:duality relationship for quasi-stationarity results} and some related notations for the two processes in Table \ref{T:dual}.
\FloatBarrier
\begin{table}[ht]
\centering
\small
{\tabulinesep=1.8mm
\begin{tabu}{ |c|c|c| }
\hline
& stochastic FKPP & 2-type BCBM \\
\hline
the process & $u=(u_t)_{t\in\R_{\geq 0}}$ & $Z_t=(G_t, R_t),\;t\in\R_{\geq 0}$\\
\hline
killing time & $\tau_{\rm fix}=\inf\{t\in\R_{\geq 0}:u_t=\textbf{0}\;\text{or}\; \textbf{1}\}$ & $\tau_{\partial}=\inf\{t\in\R_{\geq 0}:\, R_t=\emptyset\}$ \\ 
 \hline
state space prior killing & 
$\calB_{\ast}:=\mathcal{B}(\mathbb{S};[0,1])\setminus (\bar{\textbf{0}} \cup \bar{\textbf{1}})$
&  $\chi=(\cup_{n\geq 1}\bfS^n/\sim)\times(\cup_{m\geq 1}\bfS^m/\sim)$ \\
 \hline
sub-Markovian kernel & $P_t(f, \cdot)$ for $f\in \mathcal{B}_{\ast},\,t\in\R_{\geq 0}$ & $Q_t(z, \cdot)$ for $z\in\chi,\,t\in\R_{\geq 0}$\\
  \hline
 QSD & $\pi\in \mathcal{P}(\mathcal{C}_{\ast})$ & $\varphi\in \mathcal{P}(\chi)$ \\ 
\hline
right eigenfunction & $h(f)= \int_{\chi}\calE(f,z)\,\varphi(dz)=\varphi(\calE(f,\cdot))$ & $\phi(z)=\int_{\calC_{\ast}}\calE(f,z)\,\pi(df)=\pi(\calE^z)$  \\ 
\hline
eigenvalue & $\Lambda(\pi)=\Lambda(h)$ & $\Lambda(\varphi)=\Lambda(\phi)$ \\ 
\hline
\end{tabu}
}
\caption{Notation for the stochastic FKPP \eqref{fkpp_X} and the 2-type branching-coalescing Brownian motion (Definition \ref{Def:2BCBM}). We allow the initial condition of the stochastic FKPP to belong to the space $\calB_{\ast}$ which is larger than $\calC_{\ast}:=\mathcal{C}(\mathbb{S};[0,1])\setminus \{{\textbf{0}},{\textbf{1}}\}$. The function $\calE:\,\calB_{*}\times \chi\to[0,1]$ is defined in \eqref{eq:F for duality relationship}. We have $\pi(h)=\varphi(\phi)$ by \eqref{eq:duality theorem integral of right efns the same}.
}\label{T:dual}
\end{table}
\FloatBarrier

The proofs of Theorems \ref{T:main1}, \ref{thm:fixation time},  \ref{thm:conv to QSD dual}, \ref{thm:right efn of dual} and \ref{theo:duality relationship for quasi-stationarity results} will be given in Sections
\ref{S:Idoof}-\ref{S:Proofs}. Before this, in Section \ref{S:Insights}, we present some explicit calculations for the case $\beta=0$ that may offer some insights for the quasi-stationary behavior of the stochastic FKPP. 
None of our proofs in Sections
\ref{S:Idoof}-\ref{S:Proofs} (nor in the Appendix)
depend on any calculation in Section \ref{S:Insights}.

\section{Explicit calculations  when $\beta=0$}\label{S:Insights}

This section aims to offer explicit insights of our general results, and can be skipped if the reader is interested only in the proofs of Theorems \ref{T:main1}, \ref{thm:fixation time},  \ref{thm:conv to QSD dual}, \ref{thm:right efn of dual} and \ref{theo:duality relationship for quasi-stationarity results} at this point.

Recall from \eqref{eq:formula for right efn of FKPP in terms of QSD of dual results} that
$\phi(z)=\int_{\calC_{\ast}}\calE^z(f)\,\pi(df)$ is the unique  (up to a constant multiple)  right engenfunction of the killed 2-type BCBM. Hence, in principle, all moments of the QSD $\pi$ of the stochastic FKPP can be computed from the right eigenfunction $\phi$ of the 2-type BCBM, and these moments uniquely determine $\pi$, by Lemma \ref{lem:moments determine measures}. As in the well-mixed case (Section \ref{SS:1-dim}), explicit calculations for the QSD are only possible when $\beta=0$. 

We note, however, that a general approximation method for killed Markov processes based on interacting particle systems has been established by Villemonais \cite{Villemonais2011}, having been originally introduced in the case of Brownian dynamics by Burdzy, Hołyst and March \cite{Burdzy2000}. This may be employed to numerically sample the QSD of the $2$-type killed BCBM with $\beta>0$, providing numerics for the fixation time of the stochasic FKPP by the duality established in Section \ref{S:duality}.

For the rest of  this section, we assume that $\beta=0$, while $\alpha,\gamma\in \Rm_{>0}$ are fixed and arbitrary. Since $\beta=0$, there is no branching and the $2$-type branching-coalescing Brownian motion is simply a $2$-type coalescing Brownian motion (CBM).

\subsection{QSD and eigenpair for the $2$-type CBM}

Recall from  Theorem \ref{thm:conv to QSD dual} that
the $2$-type CBM $Z=(G,R)$ with $\beta=0$  has a unique QSD $\varphi^0$. As we shall see in Proposition \ref{prop:existence of QSD for FKPP beta=0}, $\varphi^0$ is supported on $\bfS\times\bfS$ and that has a  density function with respect to Lebesgue measure. We further recall from Theorems \ref{thm:right efn of dual} and \ref{theo:duality relationship for quasi-stationarity results} that the killed $2$-type CBM has a unique (up to constant multiple) right eigenfunction $\phi^0\in \mathcal{C}_b(\chi;\Rm_{>0})$ which has the same eigenvalue as the QSD, $\Lambda(\phi^0)=\Lambda(\varphi^0)$. 

In Theorem \ref{thm:explicit expressions when neutral}, we explicitly write down the QSD $\varphi^0$ and the right eigenfunction $\phi^0$, and their common eigenvalue.  Equation \eqref{eq:expression for right e-fn of CBM} gives the restriction of  $\phi^0$ to $\bfS\times \bfS$, which is more explicit than the representation \eqref{E:phiz} of $\phi^0$ on its domain $\chi$.

\begin{thm}[QSD and eigenpair for the 2-type CBM]\label{thm:explicit expressions when neutral}
We suppose that $\beta=0$. Let $\theta_{\ast}\in \big(0,\frac{\pi}{2}\big)$ be the unique solution to
\begin{equation}\label{eq:equation for theta direct calculation}
    \frac{\gamma}{\alpha}=4\theta \tan(\theta).
\end{equation}
Then the unique QSD $\varphi^0$ of the killed $2$-type CBM $Z$ is an element of $\mathcal{P}(\bfS\times\bfS)$ and is explicitly given by $\varphi^0(dx,dy)=\rho^0(x,y)\,dxdy$, whereby
\begin{equation}\label{eq:QSD density for 2-particle CBM}
\rho^0(x,y)=\frac{\theta_{\ast}}{\sin(\theta_{\ast})}\cos\Big(2\theta_{\ast}\Big(\frac{1}{2}-d_{\bfS}(x,y)\Big)\Big).
\end{equation}
Let $\phi^0\in \mathcal{C}_b(\chi;\Rm_{>0})$ be the right eigenfunction of the $2$-type CBM given by \eqref{eq:formula for right efn of FKPP in terms of QSD of dual results}.
The restriction of  $\phi^0$ to $\bfS\times \bfS$ is equal to
\begin{equation}\label{eq:expression for right e-fn of CBM}
\phi^0_{\lvert_{\bfS\times \bfS}}((x,y))=M_*\cos\Big(2\theta_{\ast}\Big(\frac{1}{2}-d_{\bfS}(x,y)\Big)\Big),
\end{equation}
where $M_*\in (0,\infty)$  is a constant. Furthermore,
\begin{equation}\label{E:phiz}
\phi^{0}(z)=M_{\ast}\cos(\theta_{\ast})\,{\bf E}_{z}\left[e^{4\alpha \theta_{\ast}^2\tau^{Z}}\right] \quad \text{for}\quad z\in \chi,
\end{equation}
where $\tau^Z$ is the first time that $Z$ consists of one green particle and one red particle, both at the same position.
The common eigenvalue $\lambda$ of $\varphi^0$ and of $\phi^0$ is given by
\begin{equation}\label{eq:principal eigenvalue no selection}
  \lambda= \Lambda(\phi^0)=\Lambda(\varphi^0)=e^{-4\alpha \theta_{\ast}^2}.
\end{equation}
\end{thm}

\medskip

The constant $M_*$ in \eqref{eq:expression for right e-fn of CBM} will  be determined  in \eqref{eq:Mast constant} below in terms of the usual $1$-type coalescing Brownian motion (CBM) $\bar{X}$  mentioned in \ref{WFdual}. We define $\tau_1$ to be the 
first time when $\bar{X}$ consists of exactly two particles both at the same position. That is,
\begin{equation}\label{Def:tau1}
\tau_{1}:=\inf\{t\in \R_{\ge 0}:\,|\mathcal{I}_t|=2,\;X^{i}_t=X^{j}_t \text{ whenever }i,j\in \mathcal{I}_t\}.
\end{equation}

Given a closed set $F\subseteq \bfS$, we define
$\Sigma_F$
to be the set of infinite sequences in $\mathbb{S}$ 
whose closure and limit set (i.e. the set of accumulation points) are both given by $F$. That is,
\begin{equation}
\Sigma_F:=\{\underline{x}:=(x_1,x_2,\ldots)\in F^{\mathbb{N}}:\;\text{every $x'\in F$ is an accumulation point of $\underline{x}$}\}.
\end{equation}
For instance, 
$\Sigma_{\{0,\,\frac{1}{2}\}}=\{(0,0,\ldots)\}\cup \{(\frac{1}{2},\frac{1}{2},\ldots)\}$,
while $\Sigma_{\bfS}$ consists of all dense sequences in $\mathbb{S}$. 

\begin{lemma}\label{L:Laplace_tau1}
Suppose  $\beta=0$. 
Let $F$ be a closed and non-empty subset of $\mathbb{S}$.
Then the limit
$${\bf E}_{\bar{x}}\left[e^{4\alpha\theta_{\ast}^2\tau_1}\right]:=\lim_{n\to\infty}{\bf E}_{(x_i)_{i=1}^n}\left[e^{4\alpha\theta_{\ast}^2\tau_1}\right] $$ 
exists in $(1,\infty)$.
This number is the same for  all  $\underline{x}=(x_1,x_2,\ldots)\in \Sigma_F$, hence we can
 denote it by ${\bf E}_{F}\left[e^{4\alpha\theta_{\ast}^2\tau_1}\right]$.  Furthermore,
the strict inequality 
\begin{equation}\label{ineq:EF strictly increasing in F}
{\bf E}_{F}\left[e^{4\alpha\theta_{\ast}^2\tau_1}\right]<{\bf E}_{F'}\left[e^{4\alpha\theta_{\ast}^2\tau_1}\right]   
\end{equation}
holds for all non-empty, closed subsets $F$ and $F'$ such that $\emptyset\neq F\subsetneq F'\subseteq \mathbb{S}$.
\end{lemma}

In particular, the limit ${\bf E}_{\bar{x}}\left[e^{4\alpha\theta_{\ast}^2\tau_1}\right]\in (1,\infty)$  
does not depend upon the choice of  $\underline{x}\in \Sigma_{\bfS}$, and is denoted by ${\bf E}_{\mathbb{S}}\left[e^{4\alpha\theta_{\ast}^2\tau_1}\right]$.

\begin{thm}\label{T:M*}
When  $\phi^0$ is given by \eqref{eq:formula for right efn of FKPP in terms of QSD of dual results}, the constant $M_*$ in \eqref{eq:expression for right e-fn of CBM} is equal to
\begin{equation}\label{eq:Mast constant}
    M_{\ast}:=\frac{1}{2\cos(\theta_{\ast})\,{\bf E}_{\mathbb{S}}\left[e^{4\alpha\theta_{\ast}^2\tau_1}\right]} \in (0,\infty).
\end{equation}
\end{thm}

\subsection{QSD and eigenpair for neutral FKPP}

We now apply Theorem \ref{thm:explicit expressions when neutral} to the neutral stochastic FKPP. Immediately from Theorem \ref{theo:duality relationship for quasi-stationarity results}, the eigenvalue of $\pi^0$ and $h^0$ is equal to
\begin{equation}
\lambda:=\Lambda(\pi^0)=\Lambda(h^0)=e^{-4\alpha \theta_{\ast}^2}.
\end{equation}
It follows that the fixation rate, defined as $\kappa=-\ln \lambda$, is given by
\begin{equation}
\kappa=4\alpha\theta_{\ast}^2.
\end{equation}

By \eqref{AbsTime_1} in Theorem \ref{thm:fixation time}, the leading order asymptotics of the fixation time is given, for each $u\in\calB_{\ast}$, by
\begin{align}\label{AbsTime_1_neutral}
\Pm_{u}(\tau_{\fix}>t)
\stackrel{(\ref{AbsTime_1})}{\sim} &\,
e^{-4\alpha \theta_{\ast}^2\,t}\,\frac{h^0(u)}{\pi^0(h^0)} \qquad \text{as } t\to\infty,
\end{align}
where $h^0(u)$ and $\pi^0(h^0)$ are given respectively by
\begin{equation}    
h^0(u)
\stackrel{(\ref{eq:formula for right efn of dual in terms of QSD of FKPP results})}{=}
\int_{\chi}\calE^z(u)\,\varphi^0(dz)
\stackrel{(\ref{eq:QSD density for 2-particle CBM})}{=}
\frac{\theta_{\ast}}{\sin(\theta_{\ast})}\int_{\bfS\times \bfS}(1-u(x))u(y)\cos\Big(2\theta_{\ast}\Big(\frac{1}{2}-d_{\bfS}(x,y)\Big)\Big)\,dx\,dy 
\end{equation}
for $u\in\calB_{\ast}$, and 
\begin{align}    
\pi^0(h^0)
\stackrel{(\ref{eq:duality theorem integral of right efns the same})}{=}
\varphi^0(\phi^0)
&\stackrel{(\ref{eq:expression for right e-fn of CBM})}{=}
\,\int_{\bfS\times\bfS}
M_*\cos\Big(2\theta_{\ast}\Big(\frac{1}{2}-d_{\bfS}(x,y)\Big)\Big)\,\varphi^0(dx,dy)\\
&\stackrel{(\ref{eq:QSD density for 2-particle CBM})}{=}
\,\frac{M_*\theta_{\ast}}{\sin(\theta_{\ast})}\,\int_{\bfS\times\bfS}
\cos^2\Big(2\theta_{\ast}\Big(\frac{1}{2}-
d_{\bfS}(x,y)\Big)\Big)\,dxdy=\,\frac{M_*\theta_{\ast}}{\sin(\theta_{\ast})}\,\frac{1}{2}\,\Big(1+\frac{\sin(2\theta_{\ast})}{2\theta_{\ast}}\Big).
\end{align}

\begin{figure}
    \centering
    \includegraphics[scale=0.4]{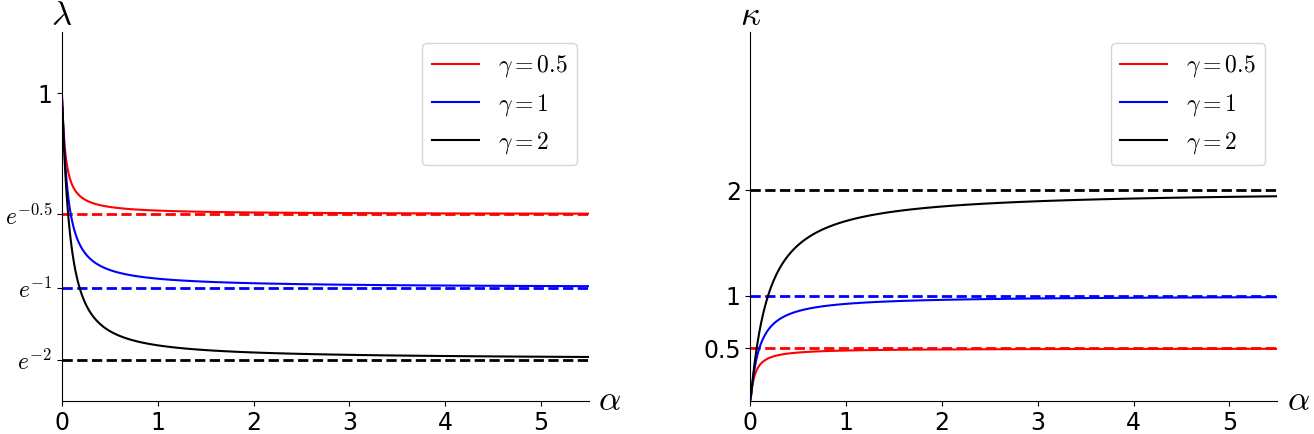}
    \caption{(Left) the eigenvalue $\lambda=e^{-4\alpha \theta_{\ast}^2}$ in \eqref{eq:principal eigenvalue no selection} against $\alpha$ for fixed values of $\gamma$. The eigenvalue $\lambda$ tends to $e^{-\gamma}$ as $\alpha\uparrow\infty$.  (Right) the fixation rate $\kappa=-\ln \lambda=4\alpha \theta_{\ast}^2$ against $\alpha$ for fixed values of $\gamma$. 
    It increases to $\gamma$ (the fixation rate  of the well-mixed case) as $\alpha\uparrow\infty$ and it decreases to 0 as $\alpha \downarrow 0$.}
    \label{fig:lambda}
\end{figure}

\FloatBarrier
In Figure \ref{fig:lambda} (right), we plot the value of the fixation rate $\kappa= -\ln \lambda$ as a function of the diffusion constant $\alpha$. 
By \eqref{eq:equation for theta direct calculation}, we obtain the  asymptotic expansions 
\begin{equation}\label{asym_fixRate}
\kappa= 
\begin{dcases}
\gamma\,\Big[1- \frac{\gamma}{12 \alpha}  +  \frac{\gamma^2}{180 \alpha^2} - \frac{\gamma^3}{3780 \alpha^3}
+ \mathcal{O}_{\frac{\gamma}{\alpha}\ra 0}\Big(\frac{\gamma^4}{\alpha^4} \Big) 
\Big]\qquad &\text{as} \quad\frac{\gamma}{\alpha}\to 0 \qquad \text{(fast diffusion)}\\
\pi^2 \alpha \,\Big[1 -8 \frac{\alpha}{\gamma}+48\frac{\alpha^2}{\gamma^2} +  \mathcal{O}_{\frac{\alpha}{\gamma}\ra 0}\Big(\frac{\alpha^3}{\gamma^3} \Big)\Big] 
\qquad &\text{as} \quad\frac{\alpha}{\gamma} \to 0 \qquad \text{(slow diffusion)}
\end{dcases}
\end{equation}

We observe that when $\gamma\ll \alpha$ (fast diffusion), the fixation rate is determined to leading order only by $\gamma$, whereas when $\gamma\gg \alpha$ (slow diffusion), the fixation rate is determined to leading order only by $\alpha$. 

\begin{remark}[Faster spatial movement speeds up fixation]\label{Rk:insight}\rm
From Figure \ref{fig:lambda} and \eqref{asym_fixRate}, we see that
faster spatial movement speeds up fixation, but there is an upper bound.
As the diffusion coefficient $\alpha\uparrow \infty$, the  fixation rate $\kappa$ of the stochastic FKPP increases to $\gamma$  which agrees with the  fixation rate in the well-mixed case \eqref{AbsTime_1_wellmix}, and
for a large but finite diffusion constant $\alpha$, the fixation rate is smaller by about $\frac{\gamma^2}{12\alpha}$ and so it takes a longer time  to get to fixation compared with the well-mixed case in the sense of \eqref{AbsTime_1_neutral}. On other hand,
as
$\alpha\downarrow 0$, the  fixation rate $\kappa$ decreases to 0 like  $\pi^2\alpha$. The constants $\frac{1}{12}$ and $\pi^2$ may change if the circle is changed to other spaces  on which the stochastic FKPP admits a solution \cite{fan2017stochastic}; we do not know. How fixation rate changes in terms of the geometry of the underlying space is of considerable interest in population genetics.
\end{remark}

\medskip

We now use Theorem \ref{thm:explicit expressions when neutral} to obtain some explicit calculations for the QSD $\pi^0$. We firstly note that by symmetry, $\expE_{u\sim \pi^0}[u(x)]=\frac{1}{2}$ for all $x\in \bfS$. We have from \eqref{eq:formula for right efn of FKPP in terms of QSD of dual results} and \eqref{eq:expression for right e-fn of CBM} that
\[
\expE_{u\sim \pi^0}[(1-u(x))u(y)]=M_{\ast}\cos\Big(2\theta_{\ast}\Big(\frac{1}{2}-d_{\bfS}(x,y)\Big)\Big),\quad x,y\in\bfS.
\]
We may therefore immediately calculate the following:
\begin{align}
 \expE_{u\sim \pi^0}[u(x)u(y)]=&\frac{1}{2}-M_{\ast}\cos\Big(2\theta_{\ast}\Big(\frac{1}{2}-d_{\bfS}(x,y)\Big)\Big),\quad x,y\in\bfS,\\
 \text{Cov}_{u\sim \pi^0}(u(x),u(y))=&\frac{1}{4}-M_{\ast}\cos\Big(2\theta_{\ast}\Big(\frac{1}{2}-d_{\bfS}(x,y)\Big)\Big),\quad x,y\in\bfS,\\
\text{Var}\Big(\int_{\bfS}u(x)dx\Big)=&\frac{1}{4}-\frac{M_{\ast}\sin(\theta_{\ast})}{\theta_{\ast}}.
\end{align}

\medskip

\noindent 
{\bf Local fixation. }
Our next result implies that for any non-empty disjoint closed sets $F_1,\,F_2\subset \mathbb{S}$, 
local fixation can occur on $F_1$ but not on $F_2$ under the QSD  of the neutral stochastic FKPP. This represents the event that there is no genetic diversity on a region $F_1$, while there is genetic variation on $F_2$.  

\begin{definition}
Given $f\in \mathcal{C}(\mathbb{S};[0,1])$ and a non-empty closed set $F\subseteq \bfS$, we say that \textbf{$f$ is fixed on $F$}
if $f\equiv 0$ on $F$ or $f\equiv 1$ on $F$.     
\end{definition}

By the definition of the QSD, $\Pm_{u\sim\pi^0}(u\text{ is not fixed on $\mathbb{S}$})=1$.
As we shall see in section \ref{S:InsightsProofs}, the strict inequalities (between 0 and 1) in Theorem \ref{thm:neutral QSD doesn't fixate on closed set} follow from \eqref{ineq:EF strictly increasing in F}.
\begin{thm}[Local fixation]\label{thm:neutral QSD doesn't fixate on closed set}
Suppose that $\beta=0$. 
For any non-empty closed set $F\subsetneq \bfS$, the probability that the neutral 
stochastic FKPP not being fixed on $F$, under the quasi-stationary distribution $\pi^0$, is given by
\begin{equation}\label{eq:prob neutral QSD not fixed on F}
    \Pm_{u\sim \pi^0}(u\text{ is not fixed on $F$})=\frac{{\bf E}_{F}\left[e^{4\alpha\theta_{\ast}^2\tau_1}\right]}{{\bf E}_{\mathbb{S}}\left[e^{4\alpha\theta_{\ast}^2\tau_1}\right]} \in (0,1).
\end{equation}
Moreover, for all non-empty closed subsets $F$ and $F'$ of $\bfS$ such that $F\subsetneq F'$, it holds that
\begin{equation}\label{eq:probabliity of fixation on set but not on larger set strictly positive}
\Pm(\text{$u$ is fixed on $F$ and not fixed on  $F'$})\in(0,1).
\end{equation}
\end{thm}
Observe that
\eqref{eq:prob neutral QSD not fixed on F} is equivalent to
\begin{equation}\label{eq:prob neutral QSD zero on F}
\Pm_{u\sim \pi^0}(u\equiv 0\text{ on $F$})=\frac{1}{2} \left(1 -\frac{{\bf E}_{F}\left[e^{4\alpha\theta_{\ast}^2\tau_1}\right]}{{\bf E}_{\bfS}\left[e^{4\alpha\theta_{\ast}^2\tau_1}\right]} \right),
\end{equation}
because $\Pm_{u\sim \pi^0}(u\equiv 0\text{ on $F$})=\Pm_{u\sim \pi^0}(u\equiv 1\text{ on $F$})$ by symmetry.

The strict inequalities in Theorem \ref{thm:neutral QSD doesn't fixate on closed set} offer some information about the support of the QSD $\pi^0$. For example, for any given non-empty closed set $F$, $\pi^0$ puts positive mass on paths that are fixed on $F$. In particular, 
$$\Pm_{u\sim \pi^0}(0<u(x)<1\text{ for all }x\in\mathbb{S}) \in[0,1).$$
Moreover, since $m\{x\in \mathbb{S}:\,u(x)=0\}=\lim_{n\to\infty}\int_{\mathbb{S}}(1-u(x))^n\,m(dx)$ for each $u\in\calC_*$, we have
\begin{align*}
\E_{u\sim \pi^0}\left[m\{x\in \mathbb{S}:\,u(x)=0\}\right]=&\,
\int_{\mathbb{S}}\P_{u\sim \pi^0}(u(x)=0)\,m(dx)=\P_{u\sim \pi^0}(u(0)=0)\\
=&\,\frac{1}{2} \left(1 -\frac{{\bf E}_{\{0\}}\left[e^{4\alpha\theta_{\ast}^2\tau_1}\right]}{{\bf E}_{\bfS}\left[e^{4\alpha\theta_{\ast}^2\tau_1}\right]} \right) \in \left(0,\,1/2 \right)
\end{align*}
where the last equality follows from \eqref{eq:prob neutral QSD zero on F} with $F=\{0\}$.

Our results and methods lay the foundation for the study of detailed properties of the QSD for the stochastic FKPP.
In particular, the technique in our proof of Theorem \ref{thm:neutral QSD doesn't fixate on closed set} enables us to obtain further properties of the QSD $\pi^0$. For example,
for any closed disjoint subsets $F_g, \,F_r \subset \bfS$,
\begin{equation}\label{Conv_product3}
\P_{u\sim \pi^0}\left(\{u\equiv 0\text{ on $F_g$}\}\setminus \{u\equiv 0\text{ on $F_r$}\}\right) 
=
M_{\ast}\cos(\theta_{\ast})\,{\bf E}_{(\underline{x},\underline{y})}[e^{4\alpha \theta_{\ast}^2\tau^{Z}}],
\end{equation}
where  $\tau^Z$ is the first time that the 2-type CBM $Z$ consists of one green and one red particle, both at the same position. The elements  
 $\underline{x}=(x_1,\ldots)\in \Sigma_{F_g}$ and
$\underline{y}=(y_1,\ldots)\in \Sigma_{F_r}$ are arbitrarily fixed and will not affect the value of ${\bf E}_{(\underline{x},\underline{y})}[e^{4\alpha \theta_{\ast}^2\tau^{Z}}]:=\lim_{n\to\infty}{\bf E}_{z^{(n)}}[e^{4\alpha \theta_{\ast}^2\tau^{Z}}]$, where $z^{(n)}:=((x_1\ldots,x_n),\,(y_1\ldots,y_n))$.

\medskip

\noindent 
{\bf A martingale. }
The duality that the coalescing Brownian motion enjoys with the neutral stochastic FKPP is a spatial version of the duality that Kingman's coalescent enjoys with Wright-Fisher diffusion; 
the coalescing Brownian motion can be viewed as a spatial version of Kingman's coalescent. Writing $N_t$ for the number of blocks (or lineages) that Kingman's coalescent has at time $t$, it is well-known and easy to check that
\begin{equation}\label{eq:Kingman martingale}
e^{t}\Big(\frac{1}{2}-\frac{1}{N_t+1}\Big)\quad\text{is a martingale.}
\end{equation}
We  recover the corresponding martingale for coalescing Brownian motion in the following.
\begin{corollary}\label{cor:martingale corollary}
Suppose that $\bar{X}_t:=(X^1_t,\ldots,X^{N_t}_t)$ is a system of coalescing Brownian motion (described before \eqref{WFdual} with $\beta=0$). Then
\begin{equation}\label{eq:CBM martingale}
    e^{4\alpha \theta_{\ast}^2t}\Big(\frac{1}{2}-\expE_{u\sim \pi^0}\Big[\prod_{i=1}^{N_t}(1-u(X_t^i))\Big]\Big)\quad\text{is a martingale.}
\end{equation}
\end{corollary}

Indeed, it is well-known that the standard Wright-Fisher diffusion on $[0,1]$ prior to fixation (when it hits $0$ or $1$) has the QSD $\pi^{\text{WF}}=\text{Unif}((0,1))$ \cite{seneta1966quasi}. The analogue of the term 
in brackets in \eqref{eq:CBM martingale} is therefore
\[
\frac{1}{2}-\expE_{u\sim \text{unif}((0,1))}\Big[(1-u)^{N_t}\Big]=\frac{1}{2}-\int_0^1(1-u)^{N_t}du=\frac{1}{2}-\frac{1}{N_t+1}.
\]
Therefore,  \eqref{eq:CBM martingale} is the spatial version of \eqref{eq:Kingman martingale}.

\section{Overview of the proofs of our main results}\label{S:Idoof}

In this section, we provide an overview of our proofs. Our proofs will be decomposed into Propositions 
\ref{prop:duality relationship for quasi-stationarity}-
\ref{prop:uniqueness QSD for FKPP}, which we shall establish later, in turn, in Section \ref{SS:ProofIdoof}. Assuming these propositions, we give the proofs of Theorems \ref{T:main1}, \ref{thm:fixation time},  \ref{thm:conv to QSD dual}, \ref{thm:right efn of dual} and \ref{theo:duality relationship for quasi-stationarity results} 
at the end of this section. The key ideas and structure of our proofs are laid out in Fig. \ref{fig:Idoof}.

Our first proposition establishes a general link between the killed stochastic FKPP and the killed $2$-type BCBM, as in Theorem \ref{theo:duality relationship for quasi-stationarity results}, \textit{without} knowledge about uniqueness of QSDs or that of the right eigenfunctions.
\begin{prop}[QSD gives right eigenfunction of the dual]\label{prop:duality relationship for quasi-stationarity}
Let $u=(u_t)_{0\leq t<\tau_{\fix}}$ be the killed stochastic FKPP and  $(Z_t)_{0_{\leq t<\taud}}$ the killed 2-type BCBM corresponding to a given set of constants $\alpha\in(0,\infty)$, $\beta\in\R_{\geq 0}$ and $\gamma\in\R_{\geq 0}$. 
\begin{enumerate}
    \item[(i)] Suppose that $\pi\in \calP(\calC_*)$ is a quasi-stationary distribution for the  stochastic FKPP and $\Lambda(\pi)>0$ is the corresponding left eigenvalue. Then $\phi(z):=\int_{\calC_{\ast}}\calE^z(f)\,\pi(df)$ defined by \eqref{eq:formula for right efn of FKPP in terms of QSD of dual results} 
     belongs to $\calC_b(\chi;\Rm_{>0})$ and is a  right eigenfunction  of the  the killed $2$-type BCBM. Furthermore, $\Lambda(\phi)=\Lambda(\pi)$.
\item[(ii)] Suppose that $\varphi\in \calP(\chi)$ is a QSD of the $2$-type BCBM  and $\Lambda(\varphi)>0$ is the corresponding left eigenvalue. Then $h(f):=\int_{\chi}\calE^z(f)\,\varphi(dz)$ defined by \eqref{eq:formula for right efn of dual in terms of QSD of FKPP results} belongs to $\calC_b(\calB_{\ast};\Rm_{>0})$ and is a  right eigenfunction of the killed stochastic FKPP. Furthermore, $\Lambda(h)=\Lambda(\varphi)$.
\end{enumerate}
\end{prop}

\medskip

The starting point of our proof of Theorem \ref{T:main1} will be the key observation that, when $\beta=0$, the dual process has a QSD that is supported on the finite dimensional space $\S\times \S$, enabling us to obtain various tightness results. Moreover this QSD is amenable to exact analysis, enabling the precise calculations of Section \ref{S:Insights} (note that these precise calculations are not needed for the proofs of our main results). In the rest of the proof, it shall be necessary to vary $\beta\in \R_{\geq 0}$. Where necessary to avoid ambiguity, we indicate this by a superscript $\beta$. 

\begin{prop}[Existence of QSD for CBM when $\beta=0$]\label{prop:existence of QSD for FKPP beta=0}
The $2$-type CBM with $\beta=0$ (i.e. without branching) has a QSD, denoted by $\varphi^0$,
which is supported on $\bfS\times\bfS$. The restriction $\varphi^0_{\lvert_{\bfS\times \bfS\setminus \Gamma}}$ off the diagonal $\Gamma:=\{(x,y)\in \bfS\times \bfS:x=y\}$  has a  density with respect to Lebesgue measure which is an element of $\calC(\bfS\times \bfS\setminus \Gamma;\Rm_{>0})$.
\end{prop}

It then follows from Proposition \ref{prop:duality relationship for quasi-stationarity}(ii)
that the stochastic FKPP  with $\beta=0$ has a right eigenfunction given by \eqref{eq:formula for right efn of dual in terms of QSD of FKPP results}, which we denote by $h^0$ and which has eigenvalue $\lambda_0:=\Lambda(h^0)=\Lambda(\varphi^0)>0$. By Girsanov's transform (Lemma \ref{L:Girsanov}), there exists a constant $C^{\beta}_t\in(0,\infty)$ such that 
\begin{align}\label{E:LowerBound_tau_generalmu}
\P_{\mu}\big( 
\tau^{\beta}_{\fix}>t\big)   \geq C^{\beta}_t\,\P_{\mu}\left( 
\tau^{0}_{\fix}>t\right)= \frac{C^{\beta}_t\,\lambda_0^t\,\mu(h^0)}{\E_{\mu}[h^0(u^0_t)\,|\,\tau^0_{\fix}>t]}\geq 
 \frac{C^{\beta}_t\,\lambda_0^t\,\mu(h^0)}{\|h^0\|_{\infty}} 
\end{align}
for all $\mu\in\calP(\calC_*)$, where the equality follows since  $h^0$ is a right eigenfunction for $(P_t)_{t\geq 0}$ when $\beta=0$.

We will use  \eqref{E:LowerBound_tau_generalmu} to establish the Feller property (Proposition \ref{prop:Feller_FKPP}) of the killed stochastic FKPP for all $\beta\in\R_{\ge 0}$ and to prove the following proposition.
\begin{prop}[Existence of QSD for FKPP for all $\beta\geq 0$]\label{prop:tightness proposition}
There exists a QSD for the stochastic FKPP, for all $\beta\in \R_{\geq 0}$. Moreover $\{\Law_{\mu}(u_t\lvert \tau_{\fix}>t)\}_{t\geq 1}$ is tight in $\calP(\calC_*)$ for all $\beta\in \R_{\geq 0}$ and $\mu\in\calP(\calB_{\ast})$.
\end{prop}

Proposition \ref{prop:tightness proposition} gives the existence of a QSD for the stochastic FKPP, but not (yet) its uniqueness.
Henceforth we fix some choice of QSD, which we denote by $\pi^{\beta}$. We therefore obtain by Proposition \ref{prop:duality relationship for quasi-stationarity} the existence of a right eigenfunction $\phi^{\beta}\in \calC_b(\chi;\Rm_{>0})$ with eigenvalue $\Lambda(\phi^{\beta})=\Lambda(\pi^{\beta})>0$ for the killed $2$-type BCBM, given by \eqref{eq:formula for right efn of FKPP in terms of QSD of dual results}. We henceforth fix $\phi^{\beta}$ to be this right eigenfunction (with the fixed normalisation given by \eqref{eq:formula for right efn of FKPP in terms of QSD of dual results}). We will use this right eigenfunction to establish the following proposition.

\begin{prop}[Existence and uniqueness of QSD for BCBM for all $\beta\geq 0$]\label{prop:convergence to QSD for 2-type BCBM}
There exists a unique QSD for the killed $2$-type BCBM, which we denote by $\varphi^{\beta}\in\calP(\chi)$, for all $\beta\in \R_{\geq 0}$. Moreover we have that
\begin{equation}\label{E:EqualEigenvalues}
\Lambda(\varphi^{\beta})=\Lambda(\phi^{\beta})
\end{equation}
and, for all $\mu\in\calP(\chi)$,
\begin{equation}\label{eq:Perron-Frobenius for dual}
    \lvert \lvert \lambda_{\beta}^{-t}\mu Q^{\beta}_t(\cdot)-\mu(\phi^{\beta})\varphi^{\beta}(\cdot)\rvert\rvert_{\rm TV}\rightarrow 0\quad\text{as}\quad t\rightarrow \infty,
\end{equation}
whereby $\|\cdot\|_{\rm TV}$ is the total variation distance in $\calP(\chi)$ and $\lambda:=\Lambda(\varphi^{\beta})=\Lambda(\phi^{\beta})$.
\end{prop}

By Proposition \ref{prop:convergence to QSD for 2-type BCBM} and \eqref{eq:formula for right efn of dual in terms of QSD of FKPP results}, there exists a continuous, bounded and strictly positive right eigenfunction $h^{\beta}\in \calC_b(\calB_{\ast};\Rm_{>0})$ of eigenvalue $\Lambda(h^{\beta})=\Lambda(\varphi^{\beta})>0$ for the killed stochastic FKPP, for all $\beta\in \R_{\geq 0}$. We may henceforth define
\begin{equation}\label{eq:eigenvalues all the same lambda defin}
\lambda_{\beta}:=\Lambda(h^{\beta})=\Lambda(\varphi^{\beta})=\Lambda(\phi^{\beta})=\Lambda(\pi^{\beta})>0.
\end{equation}

By Propositions \ref{prop:convergence to QSD for 2-type BCBM} and \ref{prop:killDual}, we obtain uniqueness and a characterization of the QSD for the FKPP for each $\beta\in \R_{\geq 0}$, given by the following proposition.
\begin{prop}[Uniqueness  of QSD for FKPP for all $\beta\geq 0$]\label{prop:uniqueness QSD for FKPP}
For all $\beta\geq 0$, there exists a unique QSD for the stochastic FKPP. This QSD, denoted by $\pi^{\beta}$, satisfies 
\begin{equation}\label{QSD_FKPP_char}
\pi^{\beta}(\calE^{z})=\lim_{t\to\infty}\frac{(\mu P^{\beta}_t)(\calE^{z})}{\lambda_{\beta}^{t}\mu(h^{\beta})} \qquad\text{for }z\in\chi \text{ and }\mu\in \calP(\calB_{\ast}).
\end{equation}
\end{prop}

\begin{figure}
    \centering
    \includegraphics[scale=0.6]{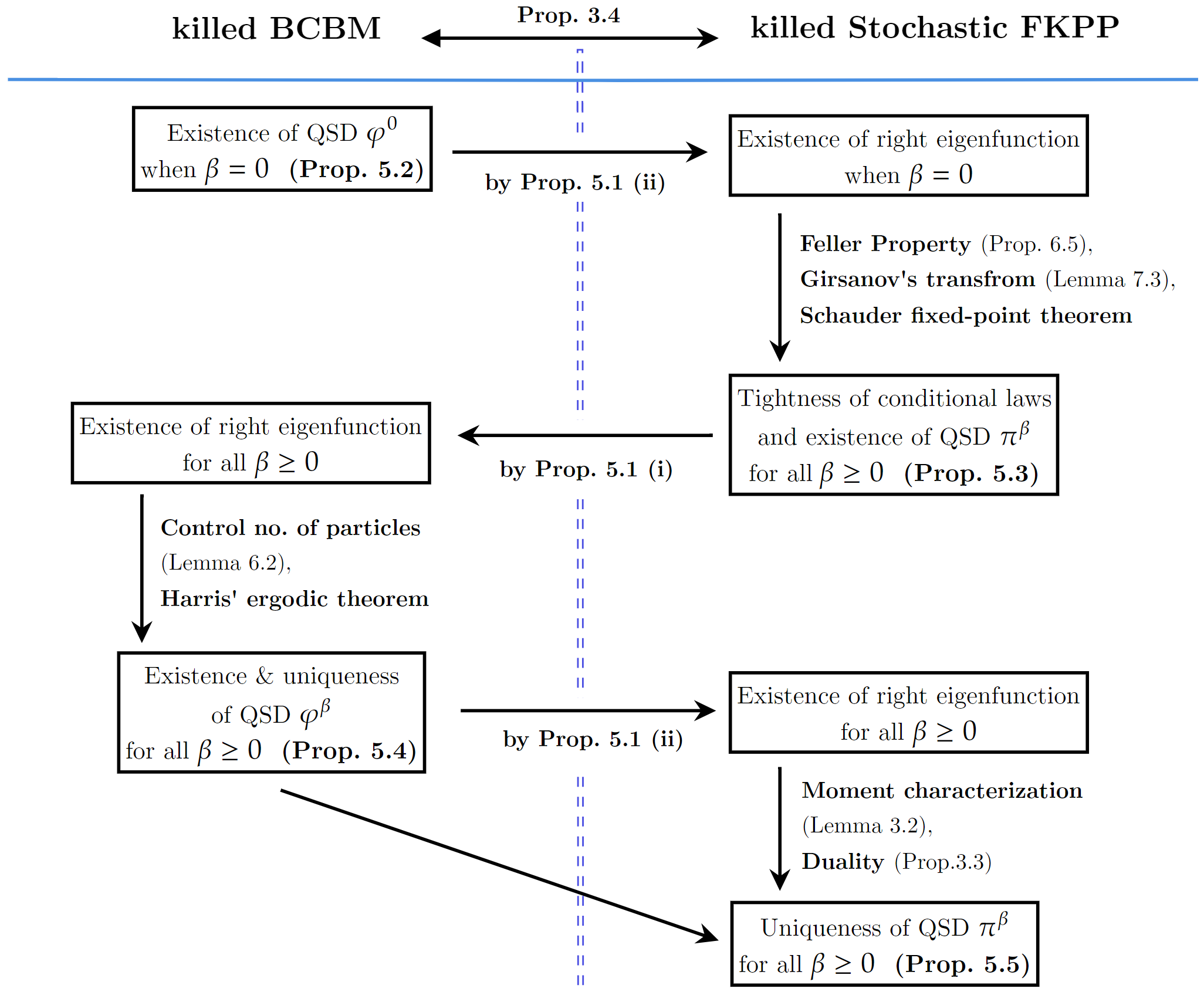}
    \caption{The key idea of our proofs exploits the duality (Proposition \ref{prop:killDual}) between the killed BCBM and the killed stochastic FKPP through
 Propositions 
\ref{prop:duality relationship for quasi-stationarity}-
\ref{prop:uniqueness QSD for FKPP}.}
    \label{fig:Idoof}
\end{figure}

For the rest of this section, we assume Propositions 
\ref{prop:duality relationship for quasi-stationarity}-
\ref{prop:uniqueness QSD for FKPP} and finish the proofs of Theorems \ref{T:main1}, \ref{thm:fixation time},
\ref{thm:conv to QSD dual}, \ref{thm:right efn of dual} and \ref{theo:duality relationship for quasi-stationarity results}. Propositions \ref{prop:duality relationship for quasi-stationarity}-
\ref{prop:uniqueness QSD for FKPP} will be proven, in turn, in Section \ref{S:Proofs}.

\medskip
\noindent
\begin{proof}[\bf Proof of Theorem \ref{T:main1}]
We fix $\beta\geq 0$. The existence and uniqueness of the QSD for the stochastic FKPP has already been established.

We now fix any sequence of times $t\rightarrow \infty$. By the tightness stated in Proposition \ref{prop:tightness proposition}, there exists a further subsequence and a subsequential limit $\widetilde{\pi}\in \calP(\calC_*)$ such that
\begin{equation}\label{E:a subseq limit}
\Law_{\mu}(u_t\lvert \tau_{\fix}>t)\ra \tilde{\pi}\quad\text{in}\quad  \calP(\calC_*) \quad\text{as}\quad t\ra\infty\quad\text{along this subsequence.}
\end{equation}

We have that $(\mu P^{\beta}_t)1=\P_{\mu}(\tau_{\rm fix}>t)$ and 
$(\mu P^{\beta}_t) (\calE^z)
= \P_{\mu}(\tau_{\rm fix}>t) \,\Law_{\mu}(u_t\lvert \tau_{\fix}>t)(\calE^{z})$ for all $z\in \chi$.
It follows that along this subsequence,
\begin{equation}\label{eq:eigenvalue over probability convergence to subsequential limit}
\frac{\lambda_{\beta}^t}{\Pm_{\mu}(\tau_{\fix}>t)}=\frac{\Law_{\mu}(u_t\lvert \tau_{\fix}>t)(\calE^{z})}{\lambda_{\beta}^{-t}(\mu P^{\beta}_t)(\calE^{z})} \;\ra\; \frac{\widetilde{\pi}(\calE^{z})}{\pi^{\beta}(\calE^{z})\mu(h^{\beta})}\qquad\text{as}\quad t\rightarrow \infty\quad\text{for all}\quad z\in \chi.
\end{equation}
In the above, we used \eqref{E:a subseq limit} and \eqref{QSD_FKPP_char} for the convergences of the numerator and the denominators respectively. These fractions are well-defined since $\pi^{\beta}(\calE^z)=\phi^{\beta}(z)>0$ for all $z\in \chi$ and $h^{\beta}(u)>0$ for all $u\in \calB_{\ast}$ by Proposition \ref{prop:duality relationship for quasi-stationarity}.

Since the left-hand side of \eqref{eq:eigenvalue over probability convergence to subsequential limit} does not depend on $z\in \chi$, it follows that either $\tilde{\pi}(\calE^z)>0$ for all $z\in \chi$ or $\tilde{\pi}(\calE^z)=0$ for all $z\in \chi$. The latter possibility would contradict $\tilde{\pi}\in \calP(\calC_{\ast})$ by Lemma \ref{lem:moments determine measures}, so we must have the former. Again using that the left-hand side of \eqref{eq:eigenvalue over probability convergence to subsequential limit} does not depend upon $z\in \chi$, and that the right-hand side belongs to $(0,\infty)$, we call the right hand side $1/\tilde{p}$. We therefore obtain  that 
\[
\lambda_{\beta}^{-t}(\mu P^{\beta}_t)1\rightarrow \tilde p\quad\text{as}\quad t\rightarrow \infty\quad\text{along this subsequence,}
\]
and that this sub-sequential limit satisfies
\[
\tilde{p}\tilde{\pi}(\calE^{z})=\mu(h^{\beta})\pi^{\beta}(\calE^{z}).
\]
It then follows from Lemma \ref{lem:moments determine measures} and the fact that $\tilde{\pi}$ and $\pi^{\beta}$ are both probability measures that 
\[
\tilde{\pi}=\pi^{\beta}\quad\text{and}\quad \tilde{p}=\mu(h).
\]

Therefore the subsequential limits $\tilde{\pi}$ and $\tilde{p}$ do not depend upon the choice of subsequence and are our desired limits, whence we conclude \eqref{E:main1} and \eqref{AbsTime_1}. This concludes the proof of Theorem \ref{T:main1}.
\end{proof}

\medskip
\noindent
\begin{proof}[\bf Proof of Theorem \ref{thm:fixation time}]

The existence of the right eigenfunction $h^{\beta}\in\calC_b(\calB_{\ast};\Rm_{> 0})$ has already been established. The convergence \eqref{AbsTime_1} was obtained in the proof of Theorem \ref{T:main1}. 

We suppose that $h'\in \calC_b(\calB_{\ast};\Rm_{> 0})$ is some other right eigenfunction of eigenvalue $\lambda'> 0$. Then it follows from \eqref{E:main1} and \eqref{AbsTime_1} that for all $u\in \calB_{\ast}$ we have
\[
\Big(\frac{\lambda'}{\lambda_{\beta}}\Big)^th'(u)=\expE_u[h'(u_t)\lvert \tau_{\fix}>t]\,\lambda_{\beta}^{-t}\,\Pm_u(\tau_{\fix}>t)\ra \frac{\pi^{\beta}(h')}{\pi^{\beta}(h^{\beta})}h^{\beta}(u) \qquad \text{as }t\to\infty.
\]
Since $u$ is arbitrary and $\pi^{\beta}(h')>0$, it follows that $\lambda'=\lambda_{\beta}$ and $h'=\frac{\pi^{\beta}(h')}{\pi^{\beta}(h^{\beta})}h^{\beta}$. The proof of the uniqueness of $h^{\beta}_{\lvert_{\calC_{\ast}}}$ in $\calC_b(\calC_{\ast};\Rm_{>0})$ is identical.

We obtain \eqref{eq:lower bound on fixation probability} from
\[
\mu(h^{\beta})\lambda_{\beta}^t=\expE_{\mu}[h^{\beta}(u_t)\lvert \tau_{\fix}>t]\,\Pm_{\mu}(\tau_{\fix}>t)\leq \lvert\lvert h^{\beta}\rvert\rvert_{\infty}\,\Pm_{\mu}(\tau_{\fix}>t).
\]
Finally, using Lemma \ref{lem:properties of fixed point equation for stochastic FKPP} (found in the proof of Proposition \ref{prop:existence of QSD for FKPP beta=0}) we obtain $\epsilon_{\beta}>0$, not dependent upon $\mu\in \calP(\calB_{\ast})$, such that 
\[
\mu(h^{\beta})\lambda_{\beta}^t=\expE_{\mu}[h^{\beta}(u_t)\lvert \tau_{\fix}>t]\,\Pm_{\mu}(\tau_{\fix}>t)\geq (\mu(h^{\beta})\wedge \epsilon_{\beta})\,\Pm_{\mu}(\tau_{\fix}>t)\quad\text{for all}\quad t\geq 0,
\]
whence we obtain \eqref{eq:upper bound on fixation probability}.
\end{proof}

\medskip
\noindent
\begin{proof}[\bf Proof of Theorem \ref{thm:conv to QSD dual}]
Theorem \ref{thm:conv to QSD dual} follows immediately from Proposition \ref{prop:convergence to QSD for 2-type BCBM}. 
\end{proof}

\medskip
\noindent
\begin{proof}[\bf Proof of Theorem \ref{thm:right efn of dual}]
Theorem \ref{thm:right efn of dual} follows, using Proposition \ref{prop:convergence to QSD for 2-type BCBM}, in the same manner as the proof of Theorem \ref{thm:fixation time}. To see that $\phi^{\beta}$ is unique in $\calB_b(\chi;\Rm_{> 0})$, we also use that the convergence in \eqref{Converge_BCBM_QSD} is total variation convergence.
\end{proof}

\medskip
\noindent
\begin{proof}[\bf Proof of Theorem \ref{theo:duality relationship for quasi-stationarity results}]
Parts 1-2 follow from
Proposition \ref{prop:duality relationship for quasi-stationarity}.
The equality \eqref{eq:duality theorem eigenvalues all the same} of eigenvalues follows from \eqref{E:EqualEigenvalues} in Proposition \ref{prop:convergence to QSD for 2-type BCBM}. 
\end{proof}

\medskip

The proofs  for  Propositions 
\ref{prop:duality relationship for quasi-stationarity}-
\ref{prop:uniqueness QSD for FKPP} will be given in
Section \ref{S:Proofs}.

\section{Proof of results}\label{S:Proofs}

We now give the proofs to our stated results in Sections 
\ref{S:Idoof} 
and \ref{S:Insights}, in that order. The proofs for Section 
\ref{S:Idoof} depend on neither the statements nor the proofs for Section \ref{S:Insights}. We begin with some basic results for the 2-type BCBM.

\subsection{Preliminaries for the BCBM}
\label{SS:PrelimBCBM}

Recall the 1-type BCBM $\bar{X}$ in \eqref{WFdual}. We first give some details about  the intersection local time $L^{(i,j)}_t$  of two particles $X^i$ and $X^j$.  Formally (\cite[P.1714]{MR1813840}), 
\begin{equation}\label{localtime_formal}
dL^{(i,j)}_t=2\alpha \,\delta_{X^i_t=X^j_t}\,dt.
\end{equation}
More precisely, by  \cite[Chapter 29]{kallenberg1997foundations}, for all $t\in\R_{\ge0}$ it is the almost sure limit
\begin{equation}\label{aslimit_localtime}
L^{i,j}_{t}=\lim_{\epsilon\downarrow 0}\frac{1}{2\epsilon}\int_0^t \Ind\big\{\Delta^{ij}_s\leq \epsilon\big\}\,d\langle \Delta^{ij}\rangle_s=\lim_{\epsilon\downarrow 0}\frac{2\alpha}{2\epsilon}\int_0^t \Ind\{\Delta^{ij}_s\leq \epsilon\}\,ds,
\end{equation}
where $\Delta^{ij}_s:={\rm d}_{\S}(X^i_s,\,X^j_s)$ is the geodesic distance between the two particles at time $s$.
Let $\mathcal{I}_t$ be the set of indexes of particles alive at time $t$. The total coalescence rate is then
\begin{equation}
\label{TotalCoaRate}
\frac{\gamma}{4\alpha} \sum_{\substack{i,j\in \mathcal{I}_t\\j\neq i}}dL^{(i,j)}_t\;=\;\frac{\gamma}{2\alpha} \sum_{\substack{{\rm unordered \;pair }\\\{i,j\} \,{\rm in } \;\mathcal{I}_t}}dL^{(i,j)}_t.
\end{equation}

\begin{remark}[A common typo]\rm
The factor $2\alpha$ in \eqref{localtime_formal}, coming from the quadratic variation of the difference process $\Delta^{ij}$, was missing in some existing literature such as \cite[P. 3483]{durrett2016genealogies} and \cite[P.139]{hobson2005duality}. 
\end{remark}

Since the pairwise coalescent rate is quadratic in the total number of particles while the branching rate is only linear, the proportion of particles that are alive at any fixed time $t$ should be bounded \textit{uniformly for all initial number} $n$ of particles. Indeed, this proportion tends to zero in expectation, as $n\to\infty$, as we will show in \eqref{supx_Expectednt} in Lemma \ref{L:supx_nt}. 


\begin{lemma}[Number of particles in BCBM]\label{L:supx_nt}
Let  $N_t=n(\bar{X}_t)$ denotes the number of particles in the system of Branching coalescing Brownian motions on $\bfS$ at time $t$. 
For any positive time $t\in (0,\infty)$,
\begin{equation}\label{supx_nt}
\sup_{n\in\mathbb{N}}\sup_{\bar{x}\in \mathbb{S}^n} {\bf P}_{\bar{x}}\left(N_t\geq k\right) \to 0 \quad \text{as}\quad k\to\infty.
\end{equation}
Furthermore, for any $\epsilon\in(0,1)$, there exists $n_{\epsilon,t}\in\mathbb{N}$ such that 
\begin{equation}\label{supx_Expectednt}
\sup_{\bar{x}\in \mathbb{S}^n} {\bf E}_{\bar{x}}\left[N_t\right] \leq \epsilon \, n \quad\text{for}\quad n\geq n_{\epsilon,t}.
\end{equation}
\end{lemma}

\begin{proof}[Proof of Lemma \ref{L:supx_nt}]
We prove this by following the duality argument in \cite[Section 3.1]{hobson2005duality} and strengthening it to all initial conditions $\bar{x}\in \mathbb{S}^n$ and to arbitrary $\beta\in [0,\infty)$. We fix arbitrary $\beta\geq 0$ for the time being. We firstly let $\theta\in(0,1)$ be a constant and  take $u_0 \equiv 1-\theta$ in  \eqref{WFdual} to obtain that
\begin{equation}
{\bf E}_{\bar{x}}\left[\theta^{N_t}\right]=\E_{1-\theta}\left[\prod_{i=1}^n(1-u(t,x_i))\right] \geq \P_{1-\theta}(u_t\equiv 0).
\end{equation}
The right hand side is independent of $n$ and $\bar{x}\in \mathbb{S}^n$. Hence for $k\in\Z_+$,
\begin{align}
{\bf P}_{\bar{x}}\left(N_t\geq k\right)={\bf P}_{\bar{x}}\left(1-\theta^{N_t}\geq 1-\theta^k\right)\leq \frac{1- \P_{1-\theta}(u_t\equiv 0)}{1-\theta^k}.
\end{align}
From this we have
\[
\limsup_{k\to\infty}\sup_{n\in\mathbb{N}}\sup_{\bar{x}\in \mathbb{S}^n} {\bf P}_{\bar{x}}\left(N_t\geq k\right) \leq 1- \P_{1-\theta}(u_t\equiv 0).
\]
Then \eqref{supx_nt} follows since $\P_{1-\theta}(u_t\equiv 0)\to 1$ as $\theta \uparrow 1$, by Lemma \ref{L:extinct_at_t}. 

We now fix $\beta=0$ for the time being. We let $\theta\in(0,1)$ be such that $1- \P_{1-\theta}(u_t\equiv 0)<\epsilon/2$. Then we have
\begin{align}\label{supx_Expectednt_2}
{\bf E}_{\bar{x}}\left[N_t\right]=
\sum_{k=1}^{n}{\bf P}_{\bar{x}}\left(N_t\geq k\right)
\leq  [1- \P_{1-\theta}(u_t\equiv 0)]\,\sum_{k=1}^{n}\frac{1}{1-\theta^k} \leq \frac{\epsilon}{2}\,n(1+\epsilon) \leq \epsilon \,n
\end{align}
for all $n$ large enough (depending on $\theta$), since $\frac{1}{n}\sum_{k=1}^{n}\frac{1}{1-\theta^k}\to 1$ as $n\to\infty$.
Hence  \eqref{supx_Expectednt} holds  for $\beta=0$.

We now consider the $\beta>0$ case. For $\beta\in [0,\infty)$, we add the superscript $\beta$ to write $N^{(\beta)}_t$ in place of $N_t$, and further write 
$L^{(\beta)}_t:=\sum_{i,j}L^{i.j}_t$ for the total intersection local time. We note that 
$N^{(\beta)}_t-N^{(\beta)}_0-\int_0^t\beta N^{(\beta)}_s\,ds + \gamma\,L^{(\beta)}_t$
is a martingale for $t\in \R_{\ge 0}$, so that
\begin{align*}
{\bf E}_{\bar{x}}\left[N^{(\beta)}_t\right]=&\, n+ \beta\int_0^t {\bf E}_{\bar{x}}\left[N^{(\beta)}_s\right]\,ds - \gamma\,{\bf E}_{\bar{x}}\left[L^{(\beta)}_t\right]\\
\leq&\, n+ \beta\int_0^t {\bf E}_{\bar{x}}\left[N^{(\beta)}_s\right]\,ds - \gamma\,{\bf E}_{\bar{x}}\left[L^{(0)}_t\right]\\
\leq &\, n+ \beta\int_0^t {\bf E}_{\bar{x}}\left[N^{(\beta)}_s\right]\,ds - n+ {\bf E}_{\bar{x}}\left[N^{(0)}_t\right]. 
\end{align*}
It follows that
\begin{align}\label{eq:difference between number with and without branching}
{\bf E}_{\bar{x}}\left[N^{(\beta)}_t\right]-{\bf E}_{\bar{x}}\left[N^{(0)}_t\right]\leq \beta\int_0^t \left(\expE_{\bar{x}}\left[N^{(\beta)}_s\right]-\expE_{\bar{x}}\left[N^{(0)}_s\right]\right)ds+\beta \int_0^t{\bf E}_{\bar{x}}\left[N^{(0)}_s\right]ds.
\end{align}

Since the last term $\beta \int_0^t\expE_{\bar x}[N_s^{(0)}]ds$ is non-decreasing in $t$, it follows from Gronwall's inequality 
(the version for the case of the forcing term being non-decreasing) 
that
\begin{align*}
{\bf E}_{\bar{x}}\left[N^{(\beta)}_t\right]-{\bf E}_{\bar{x}}\left[N^{(0)}_t\right]\leq \beta e^{\beta t}\int_0^t{\bf E}_{\bar{x}}\left[N^{(0)}_s\right]
\,ds.
\end{align*}

\commentout{\WF{It seems Gromwall gives a little stronger inequality
\begin{align*}
{\bf E}_{\bar{x}}\left[N^{(\beta)}_t\right]-{\bf E}_{\bar{x}}\left[N^{(0)}_t\right]\leq \beta e^{\beta t}\int_0^t{\bf E}_{\bar{x}}\left[N^{(0)}_s\right]e^{-\beta s}
\,ds,
\end{align*}
though the extra term $e^{-\beta s}$ will not help.
}}

The integral on the right is  $\int_0^t{\bf E}_{\bar{x}}\left[N^{(0)}_s\right]ds \leq nh+t{\bf E}_{\bar{x}}\left[N^{(0)}_h\right]$ whenever $0\leq h\leq t$,
because $s\mapsto {\bf E}_{\bar{x}}\left[N^{(0)}_s\right]$ is non-increasing and starts from $n$.
It then follows that
\begin{align*}
{\bf E}_{\bar{x}}\left[N^{(\beta)}_t\right]-{\bf E}_{\bar{x}}\left[N^{(0)}_t\right]
\leq  
\beta e^{\beta t}\left(nh+t{\bf E}_{\bar{x}}[N^{(0)}_h]\right).
\end{align*}

We now fix $t>0$ and fix $h=h(\epsilon,t)<t$ such that $\beta e^{\beta t}h<\frac{\epsilon}{3}$. Then 
\begin{align*}
{\bf E}_{\bar{x}}\left[N^{(\beta)}_t\right]\leq {\bf E}_{\bar{x}}\left[N^{(0)}_t\right]+\frac{\epsilon}{3}n+ \beta e^{\beta t}t{\bf E}_{\bar{x}}[N^{(0)}_h].
\end{align*}
Having already proven \eqref{supx_Expectednt} in the case of $\beta=0$, we use this to see that the first and third terms on the right are both at most $\frac{\epsilon}{3}n$ for all $n$ sufficiently large, whence we obtain \eqref{supx_Expectednt} for all $\beta>0$.

\end{proof}

\medskip

\begin{remark}[Entrance law]\rm \label{Rk:EntranceLaw}
For an infinite sequence $\underline{x}=(x_1,x_2,\ldots)\in \mathbb{S}^{\mathbb{N}}$, we denote by  ${\bf P}_{\bar{x}}$ (resp. ${\bf E}_{\bar{x}}$) the probability (resp. expectation) under which the usual $1$-type CBM $\bar{X}$ 
starts at $(X^{(i,0)}_0)=(x_i)$, and  when a coalescence events occurs for a pair $X^i$ and $X^j$ where $i<j$, then $X^j$ dies and $X^i$ survives (i.e. the \textit{lower} rank survives). 
A construction for such a ranked CBM starting with countably infinitely many particles can be found in \cite[Section 2]{tribe1995large} 
 and \cite{barnes2022coming}. In this construction, the final particle to survive must be $X^1$.
\end{remark}

\medskip

\subsection{Proofs for Section \ref{S:Idoof}}\label{SS:ProofIdoof}

The space  $\mathcal{C}(\mathbb{S};[0,1])$ with sup-norm is a Polish space (since $\mathbb{S}$ is compact), but it is not  locally compact since $C(\bfS;\Rm)$ is infinite-dimensional. Since $\mathcal{C}(\mathbb{S};[0,1])$ is a separable metric space, Prokhorov's theorem ensures that
a subset of $\calP(\mathcal{C}(\mathbb{S};[0,1]))$ is tight if and only if its closure is sequentially compact in the space $\calP(\mathcal{C}(\mathbb{S};[0,1]))$ equipped with the topology of weak convergence of measures. 

\subsubsection*{Proof of Proposition \ref{prop:duality relationship for quasi-stationarity}}

\noindent
(i). 
Assume  that $\pi\in\calP(\calC)$ is a QSD for the stochastic FKPP of eigenvalue $\lambda=\Lambda(\pi)>0$. We thereby obtain $\phi$ given by \eqref{eq:formula for right efn of FKPP in terms of QSD of dual results}.

The fact that $0\leq \phi\leq 1$ is clear by construction, so it is bounded in particular. Moreover the continuity of $\phi$ follows from the bounded convergence theorem and the fact that $z\mapsto \calE^z(f)$ is continuous and bounded between 0 and 1, for each $f\in \mathcal{C}_*$. 

We fix an arbitrary $t\in\R_{\geq 0}$. By Fubini's theorem, the duality \eqref{eq:moment duality relationship for 2-type results} 
and the first equality in \eqref{eq:eigentriple stochastic FKPP},
\begin{align}
Q_t\phi(z) 
&=\,{\bf E}_{z}\Big[\int_{\calC_{\ast}}\calE(f,Z_t)\pi(df)\Ind(\taud>t)\Big]\stackrel{{\rm Fubini}}{=}
\,\int_{\calC_{\ast}}{\bf E}_{z}\Big[\calE(f,Z_t)\Ind(\taud>t)\Big]\pi(df)\\
&\stackrel{(\ref{eq:moment duality relationship for 2-type results})}{=}
\,\int_{\calC_{\ast}}\expE_{f}\Big[\calE(u_t,z)\Ind(\tau_{\fix}>t)\Big]\pi(df) =\,\expE_{\pi}[\calE(u_t,z)\Ind(\tau_{\fix}>t)]\\
&\stackrel{(\ref{eq:eigentriple stochastic FKPP})}{=}
\,\Lambda(\pi)^t\int_{\calC_{\ast}}\calE(f,z)\pi(df)=\Lambda(\pi)^t\phi(z).
\end{align}
We therefore obtain that $\phi$ is a right eigenfunction for the killed $2$-type BCBM of eigenvalue $\Lambda(\phi)=\Lambda(\pi)$. 

It remains to prove that 
$\phi$ defined by \eqref{eq:formula for right efn of FKPP in terms of QSD of dual results} 
is strictly positive everywhere on $\chi$. We take $u\in \text{spt}(\pi)$ and $z'=(x,y)\in \bfS\times \bfS$ such that $\calE(u,z')>0$. It follows that there exists open set $ V\subseteq \bfS\times \bfS$ such that $z'\in V$ and  $\phi(z'')>0$ for all $z''\in V$. It then follows from the accessibility of $V$ - the fact that for all $z\in \chi$ there exists $t> 0$ such that $Q_t(z,V)>0$ - that for all $z\in \chi$ we have
\[
\phi(z)=(\Lambda(\pi))^{-t}Q_t\phi(z)\geq (\Lambda(\pi))^{-t}{\bf E}_z[\phi(Z_t)\Ind(Z_t\in V)]>0.
\]

\medskip

\noindent
(ii).
The proof of Part (ii) 
is identical to that of Part (i), except for the proof that $h$ is everywhere strictly positive on $\calC_{\ast}$. Given $u\in \calC_{\ast}$ we can choose a non-empty open set $V\subseteq \mathbb{S}\times \mathbb{S} \subseteq \chi$ such that $\mathcal{E}(z,V)>0$ for all $z\in V$. It then follows from the accessibility of $V$ that $\varphi(V)>0$. We therefore have that $h(u)>0$ by construction.   
\qed

\subsubsection*{Proof of Proposition \ref{prop:existence of QSD for FKPP beta=0}}
Since there is no branching in the $\beta=0$ case, the set of states corresponding to there being one red and one blue particle, is closed - i.e. $Q_t(z,\chi\setminus \mathbb{S}\times \mathbb{S})=0$ for all $z\in \mathbb{S}\times \mathbb{S}$. We need only existence of a OSD here, so
it suffices to obtain a QSD for the 2-type coalescing Brownian motion (2-type CBM) restricted to $\mathbb{S}\times \mathbb{S}$. Prior to killing, this is a process $\{Z_t=(G_t, R_t)\}_{t\in [0,\,\tau_{\partial})}$ with state space $\mathbb{S}\times \mathbb{S}$, evolving as two independent Brownian motions on $\mathbb{S}$ before the killing time $\tau_{\partial}$, with the system being killed at a rate given by $2\gamma$ times the intersection intersection local time $L^{(G,R)}$ of the two particles. 

Proposition \ref{prop:existence of QSD for FKPP beta=0} follows from the lemma below.
\begin{lemma}[QSD for CBM]\label{lem:exist of 2-part qsd}
We suppose that $\beta=0$. The  2-type CBM $(G,R)$ restricted to $\mathbb{S}\times \mathbb{S}$ has a QSD, denoted by $\varphi^{0,\rm two} \in \mathcal{P}(\mathbb{S}\times \mathbb{S})$, whose restriction off the diagonal $\Gamma$ is given by
\begin{equation}\label{eq:cts pve density 2-part qsd lem pf}
\varphi^{0,\rm two}_{\lvert_{\bfS\times \bfS\setminus\Gamma}}(dx,dy)=\rho(x,y)\,m(dx)\,m(dy),
\end{equation}
where $\rho\in \calC(\bfS\times \bfS\setminus \Gamma;\Rm_{>0})$.
\end{lemma}

\begin{proof}[Proof of Lemma \ref{lem:exist of 2-part qsd}] 
We abuse notation by writing $Q_t(z,\cdot)$ for the sub-Markovian transition kernels of the killed CBM restricted to  $\bfS\times \bfS$,  and define $Q_tf(z):=\int_{\bfS\times \bfS} f(y)\,Q_t(z,dy)$ and $\mu Q_t$ as before.

We firstly show that the semigroup $\{Q_t\}_{t\in\R_{\geq 0}}$ satisfies the strong Feller property on $\bfS\times \bfS$, by which we mean that $Q_t(\calB_b(\bfS\times \bfS))\subseteq \calC_b(\bfS\times \bfS)$ (note that we do not concern ourselves with strong continuity). 
To do this we follow the proof of \cite[Lemma 2.15]{MR3678472}. We write $Q^0_t(z,\cdot)$ for the Markovian transition kernel of a pair of Brownian motions in $\bfS\times \bfS$ without killing, which is strong Feller. Then for all $f\in \calB_b(\bfS\times \bfS)$, and $0<s<t<\infty$,
\[
\lvert Q_tf(z)-Q^0_sQ_{t-s}f(z)\rvert\leq {\bf P}_z(\tau_{\partial}>s)\,\lvert\lvert f\rvert\rvert_{\infty}\ra 0\quad\text{as}\quad s\ra 0\quad \text{uniformly in $z\in \bfS\times \bfS$.}
\]
It follows from the strong Feller property of $Q^0$ that $Q^0_sQ_{t-s}f\in \mathcal{C}(\bfS\times \bfS)$  for all $s>0$ and $f\in \calB_b(\bfS\times \bfS)$. Since the uniform limit of continuous functions is continuous,  $Q_tf\in \mathcal{C}(\bfS\times \bfS)$ and so
$(Q_t)_{t\geq 0}$ must also be strong  Feller.

We fix $t\in(0,\infty)$. We shall apply the Schauder fixed-point theorem to the map from $\calP(\bfS\times \bfS)$ to itself defined by
\begin{equation}\label{eq:fixed point equation for 2-part exis of QSD}
 \mu\mapsto \frac{\mu Q_t(\cdot)}{\mu Q_t1}=\Law_{\mu}(Z_t\lvert \tau_{\partial}>t).
\end{equation}

To do this, we first check that this map
is continuous as follows. We fix arbitrary $f\in \calC(\bfS\times \bfS)$. Then $Q_tf\in  \calC(\bfS\times \bfS)$ by the aforementioned Feller property. Hence if $\nu\to \mu$ in the weak topology, then $\langle\nu Q_t,\,f \rangle=\langle\nu ,\,Q_tf \rangle \to \langle\mu ,\,Q_tf \rangle=\langle\mu Q_t,\,f \rangle$. 
Using also that $Q_t1(z)>0$ for all $z\in \bfS\times \bfS$ and $1\in \calC(\bfS\times \bfS)$, it follows that the map \eqref{eq:fixed point equation for 2-part exis of QSD}
is well-defined and continuous (with respect to the topology of weak convergence) for all fixed $t>0$. 

Since $\calP(\bfS\times \bfS)$ is compact and convex, it follows from the Schauder fixed-point theorem that \eqref{eq:fixed point equation for 2-part exis of QSD} has a non-empty compact set of fixed points for all $t>0$, denoted as $\Phi_t$. 

Note that $\Phi_{2^{-k}}\subseteq \Phi_{2^{-\ell}}$ for all $\ell\leq k$ since $2^{-\ell}=2^{k-\ell}2^{-k}$, so that $(\Phi_{2^{-k}})_{k=1}^{\infty}$ is a descending sequence of non-empty compact sets. Therefore $\Phi:=\cap_{k=1}^{\infty}\Phi_{2^{-k}}$ is non-empty. Any element of $\Phi$ must be a fixed point of \eqref{eq:fixed point equation for 2-part exis of QSD} for all dyadic rational $t>0$, hence for all $t\in \Rm_{\geq 0}$ by continuity. Thus any element of $\Phi$ must be a QSD for the $2$-type CBM restricted to $\bfS\times \bfS$, so we have the existence part of Lemma \ref{lem:exist of 2-part qsd}.

We now take some fixed $\varphi^{0,\rm two}\in \Phi$, and define $\kappa:=-\ln\Pm_{\varphi^{0,\rm two}}(\tau_{\partial}>1)$. By considering the martingale problem associated to the $2$-type CBM restricted to $\bfS\times \bfS$, for test functions belonging to $\calC_c^{\infty}(\bfS\times \bfS\setminus \Gamma)$, we see that 
\begin{equation}\label{E:mtg_rho}
\psi(Z_t)\Ind(\tau_{\partial}>t)-\psi(Z_0)-\frac{\alpha}{2}\int_0^t\Delta\psi(Z_s)\Ind(\tau_{\partial}>s)ds
\end{equation}
is a martingale for all $\psi\in \calC_c^{\infty}(\bfS\times \bfS\setminus \Gamma)$. Taking expectation under $Z_0\sim \phi$, we see that
\[
e^{-\kappa t}\varphi^{0,\rm two}(\psi)-\varphi^{0,\rm two}(\psi)=\frac{\alpha}{2}\int_0^te^{-\kappa s}\varphi^{0,\rm two}(\Delta \psi)ds\quad \text{for all}\quad t\geq 0.
\]
Both sides are differentiable with respect to $t$. Differentiating with respect to $t$ at $t=0$, we see that $\varphi^{0,\rm two}_{\lvert_{\bfS\times \bfS\setminus \Gamma}}$ must be a non-negative weak solution of 
\begin{equation}\label{E:pde_rho}
\frac{\alpha}{2}\Delta \rho(x,y)=-\kappa \rho(x,y), \qquad (x,y)\in \bfS\times \bfS\setminus \Gamma.
\end{equation}
It follows from elliptic regularity and Harnack's inequality that $\varphi^{0,\rm two}_{\lvert_{\bfS\times \bfS\setminus \Gamma}}$ has a density belonging to $\mathcal{C}(\bfS\times\bfS\setminus \Gamma;\Rm_{>0})$.
\end{proof}

\medskip

\subsubsection*{Feller property for the stochastic FKPP and killed stochastic FKPP}

To our knowledge, the Feller property for the stochastic FKPP has not previously been established. Here, we establish the Feller property for both the stochastic FKPP and the killed stochastic FKPP (recall \eqref{fkpp_X_kill}) using duality, which are needed in the proof of Lemma \ref{L:Map_Theta1} and are of independent interest.

\begin{prop}[Feller property]\label{prop:Feller_FKPP}
For each $t\in(0,\infty)$, we have the following:
\begin{itemize}
\item[(i)] $S_tH\in  \calC_b(\calC(\mathbb{S};[0,1]))$ whenever $H\in  \calC_b(\calC(\mathbb{S};[0,1]))$; and
\item[(ii)] $P_tH\in  \calC_b(\calC_*)$ whenever $H\in  \calC_b(\calC_*)$, where $\calC_{\ast}:=\mathcal{C}(\mathbb{S};[0,1])\setminus \{{\textbf{0}},{\textbf{1}}\}$.
\end{itemize}
Here $\{S_t\}_{t\in\R_{\ge 0}}$ is the Markovian semigroup for the stochastic FKPP \eqref{fkpp_X}, and 
$\{P_t\}_{t\in\R_{\ge 0}}$ is the sub-Markovian semigroup defined in \eqref{eq:submarktranssemigroupFKPP}.
\end{prop}

\begin{proof}[Proof of Proposition \ref{prop:Feller_FKPP}]
We shall use the duality for each of these two processes (the stochastic FKPP and killed stochastic FKPP), and a generalization of the Stone-Weierstrass theorem for Tychonoff spaces, \cite[Section 44B, P.469-478]{willard2012general} and \cite[Exercise R(b) on P.245]{kelley2017general}). This provides the following.
\begin{thm}[Stone-Weierstrauss theorem for Tychonoff spaces]
Let $X$ be a Tychonoff space and $\mathcal{A}$ a unital sub-algebra of $\calC_b(X;\R)$ which separates points of $X$. Then $\mathcal{A}$ is dense in $\calC_b(X;\R)$ in the compact-open topology.
\end{thm}
Note that both $\calC(\mathbb{S};[0,1])$ and its subspace $\calC_*$ are Tychonoff spaces, since $\calC(\mathbb{S};[0,1])$ is a metric space induced by the sup-norm.

\smallskip

\noindent
{\bf (i). }
We write $\mathbb{X}:=\cup_{n=1}^{\infty}(\mathbb{S}^n/\sim)$. We fix arbitrary $\bar{x} \in \mathbb{X}$. We observe that
\begin{equation}\label{E:cts_on_CS01}
F:f\mapsto \E_{f}[D(u_t,\bar{x})]
\stackrel{(\ref{WFdual})}{=}
{\bf E}_{\bar{x}}[D(f,\bar{X}_t)] \quad\text{is a bounded, continuous map on }\calC(\mathbb{S};[0,1]),
\end{equation}
meaning that $F\in \calC_b(\calC(\bfS;[0,1]))$. This is an easy consequence of the dominated convergence theorem. We now define
\[
\mathcal{A}:=\left\{\Big(\calC(\bfS;[0,1])\ni f\mapsto a_0+\sum_{k=1}^{m}a_kD(f,\,\bar{x}^k)\Big):m\in \mathbb{N},(a_k)_{k=0}^{m}\in\R^{m+1},(\bar{x}^k)_{k=1}^{m}\in \mathbb{X}^{m} \right\}.
\]
Then $\mathcal{A}$ is a subalgebra of $\calC_b(\calC(\bfS;[0,1]))$
which separates points and contains the constant functions. To check these, note that  $\mathcal{A}$ is a subalgebra
which contains the constant functions, by construction. For any $f\neq g\in \calC(\mathbb{S};[0,1])$, there exists $x\in \mathbb{S}$ such that $f(x)\neq g(x)$. Then $D(\cdot, x)\in \mathcal{A}$ satisfies 
 $D(f, x)\neq D(g, x)$. Hence $\mathcal{A}$ separates points. It therefore follows from the Stone-Weierstrass theorem for Tychonoff spaces that $\mathcal{A}$ is dense in $\calC_b(\calC(\bfS;[0,1]))$ in the compact-open topology. In particular, for any $F\in \calC_b(\calC(\bfS;[0,1]))$ there exists a sequence $(F_m)_{m\in \mathbb{N}}$ in $\mathcal{A}$ such that $F_m\ra F$ uniformly on compact subsets of $\calC(\bfS;[0,1])$.

We now consider arbitrary $F\in \calC_b(\calC(\bfS;[0,1]))$, $t>0$ and $f_n\ra f$ in $\calC(\bfS;[0,1])$. We define $K:=\{f_n:n\in \mathbb{N}\}\cup \{f\}$, which we observe is a compact subset of $\calC(\bfS;[0,1])$. We take an aforedescribed sequence $(F_m)_{m\in \mathbb{N}}$ in $\mathcal{A}$ such that $F_m\ra F$ uniformly on compact subsets of $\calC(\bfS;[0,1])$. Then we have that 
\begin{equation}\label{eq:convergence in open-compact topology}
\sup_{f'\in K}\lvert F_m(f')-F(f')\rvert\ra 0\quad \text{as}\quad m\ra \infty.
\end{equation}
On the other hand,
\begin{align*}
\lvert S_tF(f_n)-S_tF(f)\rvert&\leq \lvert S_tF(f_n)-S_tF_m(f_n)\rvert +\lvert S_tF_m(f_n)-S_tF_m(f)\rvert+\lvert S_tF_m(f)-S_tF(f)\rvert
\\ &\leq 2\sup_{f'\in K}\lvert F_m(f')-F(f')\rvert +\lvert S_tF_m(f_n)-S_tF_m(f)\rvert.
\end{align*}
It then follows from \eqref{E:cts_on_CS01} that
\[
\limsup_{n\ra\infty}\lvert S_tF(f_n)-S_tF(f)\rvert\leq 2\sup_{f'\in K}\lvert F_m(f')-F(f')\rvert.
\]
Since $m$ is arbitrary, it follows from \eqref{eq:convergence in open-compact topology} that $S_tF(f_n)\ra S_tF(f)$ as $n\ra\infty$. Boundedness of $S_tH$ is trivial. The claim $(i)$ is therefore proven.

{\bf (ii). }
The claim $(ii)$ is proven identically to that of $(i)$, replacing $\mathcal{A}$ with the subalgebra
\[
\mathcal{A}'=\left\{\Big(\calC_{\ast}\ni f\mapsto a_0+\sum_{k=1}^{m}a_k\calE(f,\,z^k)\Big):m\in \mathbb{N},(a_k)_{k=0}^{m}\in\R^{m+1},(z^k)_{k=1}^{m}\in \chi \right\}
\]
of $\calC_b(\calC_{\ast})$. The only additional difficulty is to see that $\mathcal{A}'$ separates points, which is no longer trivial. 

To establish that $\mathcal{A}'$ separates points, we take $f\neq g$ in $\calC_{\ast}$. We assume, for the sake of contradiction, that whenever $f(x)\notin \{0,1\}$ or $g(x)\notin \{0,1\}$, then $f(x)=g(x)$. Then $\{x\in \bfS:f(x)=g(x)\}$, $\{x\in \bfS:f(x)=0,g(x)=1\}$ and $\{x\in \bfS:f(x)=1,g(x)=0\}$ are closed disjoint subsets of $\bfS$ with union $\bfS$, at least two of which must be non-empty. This is a contradiction. 

It follows that there exists $x\in \bfS$ such that $f(x)\neq g(x)$ and either $f(x)\notin \{0,1\}$ or $g(x)\notin \{0,1\}$. We define $z_1:=((x),(x)),z_2:=((x,x),(x))\in \chi$. If $\calE(f,z_1)=\calE(g,z_1)$ then $f(x)(1-f(x))=g(x)(1-g(x))$, implying that $g(x)=1-f(x)$ (since $g(x)\neq f(x)$). If also $\calE(f,z_2)=\calE(g,z_2)$, then $f(x)(1-f(x))^2=g(x)(1-g(x))^2=f(x)(1-f(x)(1-g(x))$, which is a contradiction. It follows that either $\calE(f,z_1)\neq \calE(g,z_1)$ or $\calE(f,z_2)\neq \calE(g,z_2)$, implying that $\mathcal{A}'$ separates points.
\end{proof}

\subsubsection*{Proof of Proposition \ref{prop:tightness proposition}}

Similarly to the map \eqref{eq:fixed point equation for 2-part exis of QSD}, for $t>0$ we define  the map $\Theta_t=\Theta^{\beta}_t:\,\calP(\calC_*)\to \calP(\calC_*)$ by
\begin{equation}\label{eq:automorphism theta map}
\Theta_t: \mu \mapsto \frac{\mu P_t(\cdot)}{\mu P_t1}=\Law_{\mu}(u_t\lvert \tau_{\fix}>t).
\end{equation}
Here we are using \eqref{E:LowerBound_tau_generalmu}, which ensures that $\mu P_t1>0$ for all $\mu\in\calP(\calC_{\ast})$.

By definition, a QSD of the stochastic FKPP is a fixed point of $\Theta^{\beta}_t$ for all $t\geq 0$.
We shall apply the {\it Schauder fixed-point theorem} to the mapping
$\Theta^{\beta}_t$ 
on the subset  
\begin{equation}\label{Def:calK_epsilon}
\calK_{\epsilon}:=\{\mu\in\calP(\calC_*):\mu(h^0)\geq \epsilon\}
\end{equation}
of the  topological vector space $\mathcal{M}(\calC_{*})$ (the space of finite signed Borel measures on $\calC_{*}$) equipped with the weak topology, where $\epsilon>0$, and $h^0$ is the right eigenfunction for the stochastic FKPP with $\beta=0$ constructed using Proposition \ref{prop:duality relationship for quasi-stationarity}(ii) and Proposition \ref{prop:existence of QSD for FKPP beta=0}. 

To be able to apply the Schauder fixed-point theorem and establish  Proposition \ref{prop:tightness proposition}, we need the following lemmas.

\begin{lemma}\label{L:Map_Theta2}
There exists $\epsilon_0>0$ such that $\calK_{\epsilon}$ defined in \eqref{Def:calK_epsilon} is non-empty for all $\epsilon\in(0,\,\epsilon_0)$.
For any $\epsilon\in(0,1)$, $\calK_{\epsilon}$ is a closed, convex subset of $\calP(\calC_{*})$. 
\end{lemma}

\begin{proof}[Proof of Lemma \ref{L:Map_Theta2}]
Clearly, $\calK_{\epsilon}$ is a closed, convex subset of  $\calP(\calC_*)$ for all $\epsilon>0$. This set is non-empty for all sufficiently small $\epsilon>0$,
since it is increasing as $\epsilon\downarrow 0$ and $h^0(f)>0$ for all $f\in\calC_*$. The latter 
follows  from the positivity of the density $\rho^0$ in Proposition \ref{prop:existence of QSD for FKPP beta=0}:
\begin{align}
h^0(f) =  \int_{\mathbb{S}\times {\mathbb{S}}} \mathcal{E}(f,(x,y))\,\varphi^0(dxdy)
=\int_{\mathbb{S}\times {\mathbb{S}}} (1-f(x))\,f(y)\,\rho^0(x,y)\,m(dx) m(dy)>0.\label{Def:muh0}
\end{align}
Hence $\mu(h^0)=\,\int_{\calC_*} h^0(f)\,\mu(df)>0$.

\end{proof}

\begin{lemma}\label{lem:properties of fixed point equation for stochastic FKPP}
The following hold for any fixed $\beta\in \R_{\geq 0}$:
\begin{itemize}
\item[(i)] There exists $c_{\beta}^0>0$ such that $\inf_{t\in(0,1]}(\Theta^{\beta}_t(\mu))(h^0)\geq c_{\beta}^0\mu(h^0)$ for all $\mu\in\calP(\calC_*)$. In particular,
$\Theta^{\beta}_t(\calK_{\epsilon})\subset \calK_{c_{\beta}^0\epsilon}$ for all $t\in(0,1]$ and $\epsilon\in \Rm_{>0}$. 
\item[(ii)] For all $t\in(0,1]$, there exists $\epsilon_{t}^{\beta}>0$ such that  $(\Theta^{\beta}_t(\mu))(h^0)\geq 2\mu(h^0)\wedge 2\epsilon_t^{\beta}$ for all $\mu\in\calP(\calC_*)$.
\end{itemize}
\end{lemma}

\begin{proof}[Proof of Lemma \ref{lem:properties of fixed point equation for stochastic FKPP}]
{\bf (i).}  We take $\beta=0$ for the time being. For $\mu\in\calP(\calC_*)$, since $h^0$ is a right eigenfunction for the stochastic FKPP, 
\begin{equation}\label{E:h0_righteigen}
\Theta^{0}_t(\mu)(h^0)=\frac{\lambda_0^t\mu(h^0)}{\Pm_{\mu}(\tau^0_{\fix}>t)},
\end{equation}
where $\lambda_0:=\Lambda(h^0)$ is the eigenvalue of $h^0$. This is also the equality in \eqref{E:LowerBound_tau_generalmu}.

By Girsanov's transform (Lemma \ref{L:Girsanov}), there exists for all $\beta \in\R_{\geq 0}$ some constant $k_{\beta}\in (0,\infty)$ (dependent only upon $\beta$) such that
\[
\Theta^{\beta}_t(\mu)\geq k_{\beta}\Theta_t^0(\mu)\quad\text{for all}\quad \mu\in\calP(\calC_*),\quad 0< t\leq 1.
\]
It therefore follows that
\begin{equation}\label{Fbetat_lowerBound}
\Theta^{\beta}_t(\mu)(h^0) \geq   k_{\beta}\Theta_t^0(\mu)(h^0)
=k_{\beta}\frac{\lambda_0^t\mu(h^0)}{\Pm_{\mu}(\tau^0_{\fix}>t)}
\quad\text{for all}\quad \mu\in\calP(\calC_*),\quad 0< t\leq 1.
\end{equation}
The first statement in Lemma \ref{lem:properties of fixed point equation for stochastic FKPP} then immediately follows by taking $c^0_{\beta}=k_{\beta}\,\lambda_0$, since $\lambda_0\in (0,1]$. 

\smallskip
\noindent
{\bf (ii).}  We now prove that for all $0<t\leq 1$,
\begin{equation}\label{eq:fixation prob goes to 0 as right efn does}
    \Pm_{\mu}(\tau_{\fix}^0>t)\ra 0\quad\text{as}\quad \mu(h^0)\ra 0.
\end{equation}
We recall from Proposition \ref{prop:existence of QSD for FKPP beta=0} that $\varphi^0_{\lvert \bfS\times\bfS\setminus \Gamma}$ has a continuous and strictly positive density $\rho$, so that $c_0(\epsilon):= \inf_{z\in \bfS\times \bfS\setminus B(\Gamma,\epsilon)} \rho(z)>0$ for all $\epsilon>0$. We also note that $\bar F:=\sup_{\substack{f\in \calC_{\ast}\\z\in \bfS\times \bfS}}\calE(f,z)<\infty$ (in fact $\leq 1$) and $\int_{\mathbb{S}\times {\mathbb{S}}} \calE(f,(x,y))\,dxdy= \left(\int_{\mathbb{S}}f\,dx \right)\,\left(1-\int_{\mathbb{S}}f\,dx\right)$. Hence, by
 \eqref{eq:formula for right efn of dual in terms of QSD of FKPP results} and \eqref{Def:muh0},
\begin{align}
\mu(h^0)&= \int_{\calC_*} \int_{\mathbb{S}\times {\mathbb{S}}} \calE(f,(x,y))\,\varphi^0(dxdy)\,\mu(df)\\
&\geq  c_0(\epsilon)
\int_{\calC_*} \int_{\mathbb{S}\times {\mathbb{S}}} \calE(f,(x,y))\,dxdy\,\mu(df)\\
&-\int_{\calC_*} \int_{B(\Gamma,\epsilon)} \calE(f,(x,y))[c_0(\epsilon)
-\rho(x,y)]\Ind(\rho(x,y)< c_0(\epsilon))\,dxdy\,\mu(df)\\
&\geq  c_0(\epsilon)
\Big[\int_{\calC_*} \int_{\mathbb{S}\times {\mathbb{S}}} \calE(f,(x,y))\,dxdy\,\mu(df)-\text{Vol}(B(\Gamma,\epsilon)) \bar F\Big]. \label{muh0_lowerbound}
\end{align}

Since $\text{Vol}(B(\Gamma,\epsilon))\ra 0$ as $\epsilon \ra 0$, and $\int_{\mathbb{S}\times {\mathbb{S}}} \calE(f,(x,y))\,dxdy= \left(\int_{\mathbb{S}}f\,dx \right)\,\left(1-\int_{\mathbb{S}}f\,dx\right)$, it follows that
\[
\int_{\calC_*} \left(\int_{\mathbb{S}}f\,dx \right)\,\left(1-\int_{\mathbb{S}}f\,dx\right)\,\mu(df) \ra 0\quad\text{as}\quad \mu(h^0)\ra 0.
\]
From this and from Lemma \ref{L:extinct_at_t}, we have \eqref{eq:fixation prob goes to 0 as right efn does} for all $0<t\leq 1$.

It follows from \eqref{Fbetat_lowerBound} and \eqref{eq:fixation prob goes to 0 as right efn does} that for all $0<t\leq 1$ and $\beta\in \R_{\geq 0}$ there exists $\epsilon_t^{\beta}>0$ such that $(\Theta_t(\mu))(h^0)\geq 2\mu(h^0)$ whenever $\mu(h^0)\leq \epsilon_t^{\beta}$. It therefore follows from Part $(i)$ that $(\Theta_t(\mu))(h^0)\geq 2\mu(h^0)\wedge c^0_{\beta}\epsilon_t^{\beta}$ for all $\mu\in\calP(\calC_{\ast})$. Reducing $\epsilon_t^{\beta}$, we obtain Part $(ii)$.
\end{proof}

We have that the closure of $\calK_{\epsilon}$ in $\calP(\calC(\bfS;[0,1]))$ is given by $\cl({\calP(\calC(\bfS;[0,1]))})=\{\mu\in \calP(\calC(\bfS;[0,1])):\mu(h^0)\geq \epsilon\}$. This is not a subset of $\calP(\calC_{\ast})$, but is a subset of 
\[
\widetilde{\calP}:=\{\mu\in \calP(\calC(\bfS;[0,1])):\mu(\calC_{\ast})>0\}. 
\]
That is, $\mu\in\cl({\calK_{\epsilon}})$ need not apply all of its mass to $\calC_{\ast}$, but it must apply some mass to $\calC_{\ast}$. We consider $\widetilde{\calP}$ to be a subset of $\calP(\calC(\bfS;[0,1]))$, equipping it with the subspace topology. We observe that the map $\Theta^{\beta}_t$ defined in \eqref{eq:automorphism theta map} is well-defined as a map $\widetilde{\calP}\ra \calP(\calC_{\ast})$.
\begin{lemma}\label{L:Map_Theta1}
The maps $\Theta^{\beta}_t:\,\calP(\calC_*)\to \calP(\calC_*)$ and $\Theta^{\beta}_t:\tilde{\calP}\to \calP(\calC_*)$ are continuous for each $t\in(0,\infty)$. Moreover, $\Theta^{\beta}_{t+s}=\Theta^{\beta}_{t}\circ\Theta^{\beta}_{s}$ for $s,t\in(0,\infty)$, where $\circ$ is composition.
\end{lemma}
\begin{proof}[Proof of Lemma \ref{L:Map_Theta1}]
By definition, for any $H\in \calC_b(\calC_*)$,
\[
\left(\Theta^{\beta}_t(\mu)\right)(H)=\E_{\mu}[H(u_t)\,|\,\tau_{\fix}>t]=\frac{\mu(P_tH)}{\mu(P_t1)}.
\]
The continuity of $\Theta^{\beta}_t:\,\calP(\calC_*)\to \calP(\calC_*)$ follows from (i) the Feller property of the killed stochastic FKPP established in Proposition \ref{prop:Feller_FKPP} (note that $1\in \calC_b(\calC_*)$) and (ii)  the fact $\mu(P_t1)>0$ for all $\mu\in \calP(\calC_*)$.

We have from Lemma \ref{L:extinct_at_t} that $\Pm_u(\tau_{\fix}>t)\ra 0$ as $u\ra {\bf{0}},{\bf{1}}$. It follows that $P_tH\in \calC_b(\calC(\bfS;[0,1]))$ for $H\in \calC_b(\calC_{\ast})$, with $P_tH$ vanishing on $\{{\bf{0}},{\bf{1}}\}$. The proof that $\Theta^{\beta}_t:\tilde{\calP}\to \calP(\calC_*)$ is continuous is then identical to the above proof that $\Theta^{\beta}_t:\calP(\calC_*)\to \calP(\calC_*)$ is continuous.

The fact that $\Theta^{\beta}_{t+s}=\Theta^{\beta}_{t}\circ\Theta^{\beta}_{s}$ follows from  the semigroup property of $(P_t)_{t\geq 0}$.
\end{proof}

\begin{lemma}\label{lem:compactness lemma}
We fix any $\beta\in \R_{\geq 0}$, $t\in(0,1]$ and $\epsilon>0$. Then $\Theta^{\beta}_t(\calK_{\epsilon})$ is relatively compact in  both  $\calP(\calC(\mathbb{S};[0,1]))$ and in $\calP(\calC_{\ast})$.
\end{lemma}
\begin{proof}[Proof of Lemma \ref{lem:compactness lemma}]
We firstly show that $\Theta^{\beta}_t(\calK_{\epsilon})$ is tight in $\calP(\calC(\mathbb{S};[0,1]))$. That is,  for all $\zeta>0$, there exists a compact subset $\Gamma_{\zeta}\subset \calC(\mathbb{S};[0,1])$ such that
\begin{align}\label{TightInC}
\inf_{\mu\in\calP(\calC_*):\,\mu(h^0)\geq \epsilon}\P_{\mu}\left(u^{\beta}_t\in \Gamma_{\zeta}\,\big|\,\tau^{\beta}_{\fix}>t\right) \,>&\,1-\zeta.
\end{align}
Since $\calC(\mathbb{S};[0,1])$ is a Polish space under the  uniform topology and since $\|u_t\|_{\infty}\leq 1$, it suffices (see \cite[Theorem 7.3]{billingsley2013convergence}) to show that for any $\eta,\eta'>0$, there exists $\delta\in(0,1)$ such that 
\begin{align}\label{E:Tight_condition_1}
\sup_{\mu\in\calP(\calC_*):\,\mu(h^0)\geq \epsilon}\P_{\mu}\left( 
\omega(u^{\beta}_t;\delta)\,\geq \eta\,\big|\,\tau^{\beta}_{\fix}>t\right) \,\leq &\,\eta',
\end{align}
whereby $\omega(f;\delta):= \sup_{x,y\in\mathbb{S}:\,|x-y|\leq \delta}|f(x)-f(y)|$. In Lemma \ref{L:uniform continuity}, we show that \eqref{E:Tight_condition_1} holds  if we get rid of the conditioning. From this, \eqref{E:Tight_condition_1} itself holds because
\begin{equation}\label{E:LowerBound_tau}
\inf_{\mu\in\calP(\calC_*):\,\mu(h^0)\geq \epsilon}\P_{\mu}\left( 
\tau^{\beta}_{\fix}>t\right)
\stackrel{ (\ref{E:LowerBound_tau_generalmu})}{>}
0.
\end{equation}
We have now established that $\Theta^{\beta}_t(\calK_{\epsilon})$ is tight in $\calP(\calC(\bfS;[0,1]))$. 

It follows that the closure of $\Theta^{\beta}_t(\calK_{\epsilon})$ in $\calP(\calC(\bfS;[0,1]))$, $\cl({\Theta^{\beta}_t(\calK_{\epsilon})})$, is a compact subset of $\calP(\calC(\bfS;[0,1]))$. Since also  $\Theta^{\beta}_t(\calK_{\epsilon})\subseteq \calK_{c^{\beta}_0\epsilon}$ and $\cl({\calK_{c^{\beta}_0\epsilon}})\subseteq \tilde{\calP}$, $\cl({\Theta^{\beta}_t(\calK_{\epsilon})})$ is a compact subset of $\tilde{\calP}$. We established in Lemma \ref{L:Map_Theta1} that $\Theta^{\beta}_t:\tilde{\calP}\ra \calP(\calC_{\ast})$ is continuous. Thus $\Theta^{\beta}_t(\cl({\Theta^{\beta}_t(\calK_{\epsilon})}))$ is a compact subset of $\calP(\calC_{\ast})$, hence $\Theta^{\beta}_{2t}(\calK_{\epsilon})\subseteq \Theta^{\beta}_t(\cl({\Theta^{\beta}_t(\calK_{\epsilon})}))$ is a relatively compact subset of $\calP(\calC_{\ast})$. Since $t>0$ is arbitrary, we are done.
\end{proof}

\smallskip

We now return to the proof of Proposition \ref{prop:tightness proposition}.
We fix $t=2^{-k}>0$ for the time being, for arbitrary $k\in \mathbb{N}$. By Lemmas \ref{L:Map_Theta2} and \ref{lem:properties of fixed point equation for stochastic FKPP}(i),
if $0<\epsilon<\epsilon_t^{\beta} \wedge \epsilon_0$, then $\Theta^{\beta}_t(\calK_{\epsilon})\subset \calK_{\epsilon}$ with $\calK_{\epsilon}$ a non-empty, closed, convex subset of $\calP(\calC_{\ast})$. We henceforth fix such an $\epsilon$. By Lemma \ref{lem:compactness lemma},  $\Theta^{\beta}_t(\calK_{\epsilon })$ is  relatively compact in $\calK_{\epsilon }$.
Hence, using Lemma \ref{L:Map_Theta1}, it follows from the Schauder fixed-point theorem that $\Pi^{\beta}_k$ is non-empty and compact for all $k\geq 0$ large enough, where $\Pi^{\beta}_k$ is defined  to be the set of fixed points of the map
\[
\Theta^{\beta}_{2^{-k}}:\calK_{\epsilon }\rightarrow \calK_{\epsilon }.
\]
Since $\epsilon<\epsilon_t^{\beta}$, if $\mu\notin \calK_{\epsilon}$ then $(\Theta^{\beta}_t(\mu))(h^0)\geq 2\mu(h^0)$ by Part $(ii)$ of Lemma \ref{lem:properties of fixed point equation for stochastic FKPP}, so that $\Theta^{\beta}_t(\mu)\neq \mu$. Therefore $\Pi^{\beta}_k$ is the set of fixed points of the map
\[
\Theta^{\beta}_{2^{-k}}:\calP(\calC_{\ast})\rightarrow \calP(\calC_{\ast}),
\]
so in particular does not depend upon the choice of $\epsilon<\epsilon_t^{\beta}\wedge \epsilon_0$.

It therefore follows that $(\Pi^{\beta}_k)_{k\geq 0}$ is a descending sequence of non-empty compact sets, so have non-empty intersection, which we define to be $\Pi^{\beta}:=\cap_k\Pi^{\beta}_k$. 

We now fix an arbitrary $\pi^{\beta}\in \Pi^{\beta}$. It follows that $\Theta^{\beta}_t(\pi^{\beta})=\pi^{\beta}$ for all dyadic rational $t\geq 0$, hence for every $t\in \Rm_{\geq 0}$ by continuity of the map $t\mapsto \Theta^{\beta}_t(\pi^{\beta})$ from $\R_{\geq 0}$ to $\calP(\calC_{*})$. We have therefore established the existence of a QSD, $\pi^{\beta}$, for the stochastic FKPP, for all $\beta\in \R_{\geq 0}$.

We now fix an arbitrary $\mu\in\calP(\calC_*)$, and show that $\{\Theta_t(\mu)\}_{t\geq 1}$  is tight in $\calP(\calC_*)$. Using Part (ii) of Lemma \ref{lem:properties of fixed point equation for stochastic FKPP}, we observe by induction that either $\Theta_n(\mu)(h^0)\geq 2^n\mu(h^0)$ for all $n\in \mathbb{N}$, or there exists $n=n(\mu)<\infty$ such that $\Theta_n(\mu)(h^0)\in \calK_{\epsilon^{\beta}_1}$. Since $h^0$ is bounded, we cannot have $\Theta_n(\mu)(h^0)\ra \infty$ as $n\ra\infty$, so there must exist $n=n(\mu)$ such that $\Theta_n(\mu)\in \calK_{\epsilon^{\beta}_1}$. It follows from Lemma \ref{lem:properties of fixed point equation for stochastic FKPP} that there exists $\epsilon=\epsilon(\mu)>0$ such that $\Theta_t(\mu)\in \calK_{\epsilon}$ for all $t\geq n(\mu)$. It also follows from Lemma \ref{lem:properties of fixed point equation for stochastic FKPP} that, reducing $\epsilon(\mu)>0$ if necessary, we have $\Theta_t(\mu)\in \calK_{\epsilon}$ for all $t\leq n(\mu)$. Therefore there exists $\epsilon=\epsilon(\mu)>0$ such that $\Theta_t(\mu)\in \calK_{\epsilon}$ for all $t\geq 0$. Therefore $\Theta_t(\mu)\in \Theta_1(\calK_{\epsilon})$ for all $t\geq 1$, which is a precompact subset of $\calP(\calC_{\ast})$ by Lemma \ref{lem:compactness lemma}.

\qed

\subsubsection*{Proof of Proposition \ref{prop:convergence to QSD for 2-type BCBM}}

We fix an arbitrary $\beta\in\R_{\ge 0}$ and a right eigenfunction $\phi=\phi^{\beta}\in \calC_b(\chi;\Rm_{>0})$ of the killed $2$-type BCBM, and we let $\lambda:=\Lambda(\phi)>0$ be the corresponding eigenvalue. Then $Q_t\phi=\lambda^t\phi$ for all $t\in\R_{\ge 0}$. Note that this eigenpair $(\phi, \lambda)$ exists by
Propositions \ref{prop:duality relationship for quasi-stationarity} and \ref{prop:tightness proposition}.

Following  Doob's transformation, we define the following Markovian transition kernel $\overline{Q}$ on $\chi$, 
\begin{equation}
\overline{Q}_t(z,dz'):=\frac{\phi(z')}{\phi(z)}\,\lambda^{-t}Q_t(z,dz'),\quad t\in \R_{\ge 0},\;z\in \chi,
\end{equation}
and also the corresponding semigroup operators $\overline{Q}_tf=\lambda^{-t}\frac{Q_t(\phi f)}{\phi}=\frac{Q_t(\phi f)}{Q_t(\phi)}$ for all $f\in \mathcal{B}_b(\chi)$. In particular, $\overline{Q}_t 1=1$ for all $t\in\R_{\ge 0}$. The probabilistic interpretation is that $\{\overline{Q}_t\}_{t\ge 0}$ is the transition kernel of the $2$-type BCBM  conditioned on never being killed  \cite{doob1957conditional, chetrite2015nonequilibrium}. In the literature this conditioned process is typically referred to as the ``$Q$-process'', see for instance \cite{champagnat2016exponential}. We will apply Harris' ergodic theorem  to the discrete-time chain$(\overline{Q}_n)_{n\in \Zm_{\geq 0}}$. We organize our proof into three steps below.

{\bf Step 1: Existence and uniqueness of the QSD for the killed 2-type BCBM. }
We suppose for the time being that Assumptions 1-2 of \cite[Theorem 1.2]{Hairer2011} hold for the discrete-time Markov chain $(\overline{Q}_n)_{n\in \Zm_{\geq 0}}$ (we will verify them at the end of this proof). Then 
there exists a unique stationary distribution for $\overline{Q}_1$, which we denote by
$\overline{\mu}\in \calP(\chi)$. For any other $t>0$, since 
\[
\mu \overline{Q}_t\overline{Q}_1=\mu \overline{Q}_1\overline{Q}_t=\mu \overline{Q}_t, 
\]
we see that $\mu \overline{Q}_t$ must also be a stationary distribution for $\overline{Q}_1$. By the uniqueness of stationary distributions for $\overline{Q}_1$, it follows that $\mu\overline{Q}_t=\mu$, so that $\mu$ is a stationary distribution for $\overline{Q}_t)_{t\geq 0}$ (in fact, the unique one).

We now write $t=n+h$ for $0\leq h<1$ and an integer $n$. Then for any initial condition $\nu$ we can write
\[
\nu \overline{Q}_t-\mu = \nu \overline{Q}_n(\overline{Q}_h-1)+\nu \overline{Q}_n - \mu = (\nu \overline{Q}_n-\mu)(\overline{Q}_h-1) + (\nu \overline{Q}_n-\mu).
\]
It therefore follows from \cite[Theorem 1.2]{Hairer2011} that
\begin{equation}\label{eq:convergence to stationary for Q process}
\lvert\lvert \nu \bar Q_t-\mu\rvert\rvert\leq 3\lvert\lvert \nu \bar Q_{\lfloor t\rfloor}-\mu\rvert\rvert\ra 0\quad\text{as}\quad t\ra\infty.
\end{equation}

We now define $\varphi\in \mathcal{P}(\chi)$ to be 
\begin{equation}\label{Def:varphi}
\varphi(dz)=
\frac{1}{C_R}\frac{\overline{\mu}(dz)}{\phi(z)}, \quad \text{where}\quad C_R=\int \frac{1}{\phi(z)}\overline{\mu}(dz).    
\end{equation}
Then  $\varphi$ is a QSD for the 2-type BCBM because, for all $f\in \mathcal{B}_b(\chi)$ and $t\in\R_{\ge 0}$,
\begin{align}
\varphi Q_t (\phi f)=  \varphi (Q_t(\phi f))=& \frac{1}{C_R}\int \frac{Q_t(\phi f)}{\phi}\,\overline{\mu}(dz)\\
=& \frac{\lambda^t}{C_R} \int  \overline{Q}_t f(z)\,\overline{\mu}(dz) = \frac{\lambda^t}{C_R}\int f(z)\,\overline{\mu}(dz) \qquad \text{as }\overline{\mu} \text{ is a stationary distribution.}\\
=&\lambda^t\,\varphi(\phi f).
\end{align}
Conversely, if $\widehat{\varphi}\in \mathcal{P}(\chi)$ is a QSD of the 2-type BCBM, then
$\frac{\phi(z)\,\widehat{\varphi}(dz)}{\widehat{\varphi}(\phi)}\in \calP(\chi)$ is a stationary distribution for $(\overline{Q}_t)_{t\in \R_{\ge 0}}$ because for all $f\in \calB_b(\chi)$ we have
\[\langle \, \phi\,\widehat{\varphi},\overline{Q}_tf\rangle = \lambda^{-t}\int_{\chi}\frac{Q_t(\phi f)}{\phi}\,\phi\,d\widehat{\varphi} = \lambda^{-t}\widehat{\varphi} \left( Q_t(\phi f)\right)=\lambda^{-t} \left(\widehat{\varphi} Q_t\right)(\phi f)=\lambda^{-t}\lambda^{t}\widehat{\varphi} (\phi f)= \langle \phi\,\widehat{\varphi},\, f\rangle\]
for all $t \in \R_{\ge 0}$, where we used the fact that $\widehat{\varphi}$ is a QSD in the penultimate equality. Hence by the uniqueness of $\overline{\mu}$ (\cite[Theorem 1.2]{Hairer2011}) we have $\overline{\mu} = \frac{\phi(z)\,\widehat{\varphi}(dz)}{\widehat{\varphi}(\phi)}$. This implies that $\widehat{\varphi}=\varphi\in\mathcal{P}(\chi)$. Furthermore, from the above we see that  $\Lambda(\varphi)=\Lambda(\phi)=\lambda$. Hence \eqref{E:EqualEigenvalues} holds.

\smallskip

\noindent
{\bf Step 2: Convergence \eqref{eq:Perron-Frobenius for dual} for each $\beta\in\R_{\ge 0}$. }  
Step 1 above and \eqref{eq:convergence to stationary for Q process} imply that  $\frac{\phi\,d\varphi}{\varphi(\phi)}\in \calP(\chi)$ is the unique stationary distribution for $(\overline{Q}_t)_{t\in\R_{\ge 0}}$, and
\[
\left\| \nu\overline{Q}_t(\cdot)-\frac{\phi\,d\varphi}{\varphi(\phi)} \right\|_{\TV}\ra 0\quad \text{as}\quad t\ra\infty \qquad \text{for all}\qquad \nu\in \calP(\chi).
\]
Since $\phi$ is bounded and everywhere strictly positive, for any $\mu\in \calP(\chi)$ we may define the probability measure 
\[
\nu(dz):=\frac{\phi(z)\mu(dz)}{\mu(\phi)}.
\]
It follows that 
\[
\nu \overline{Q}_t(dz')=\lambda^{-t}\frac{\phi(z')}{\mu(\phi)}\mu Q_t(dz').
\]

To establish \eqref{eq:Perron-Frobenius for dual}, we now define the compact sets $K_m:=\{z\in\chi:\,N(z)\leq m\}=\cup_{k=2}^m \chi_k$ for $1\leq m<\infty$, where $\chi_k:=\{z\in\chi:\,N(z)=k\}$. Since $\phi$ is bounded away from $0$ on compacts, it follows that for all $\mu\in \calP(\chi)$ we have
\begin{equation}\label{eq:total variation convergence on compacts for BCBM}
\lvert\lvert \lambda^{-t}\mu Q_t(\cdot)_{\lvert_{K_m}}-\frac{\mu(\phi)}{\varphi(\phi)}\varphi (\cdot)_{\lvert_{K_m}}\rvert\rvert_{\TV}\ra 0\quad \text{as}\quad t\ra\infty,\quad\text{for every}\quad m<\infty.
\end{equation}
For $\mu\in \calP(\chi)$, $0\leq t<\infty$ and $m<\infty$ we define
\[
d^{\mu,m}_t:=\lambda^{-t}\mu Q_t\Ind_{K_m^c}.
\]
It follows from Lemma \ref{L:supx_nt} that, for all $\epsilon>0$, there exists $m\in \mathbb{N}$ such that for all $\mu\in \chi$,
\[
d^{\mu,m}_{t+1}\leq \frac{\epsilon}{\lambda}\lambda^{-t}\mu Q_t\Ind_{\chi}\leq \frac{\epsilon}{\lambda}\big[\lambda^{-t}\mu Q_t\Ind_{K_m}+d^{\mu,m}_{t}\big].
\]
It then follows from \eqref{eq:total variation convergence on compacts for BCBM} that
\[
\limsup_{m\ra\infty}\limsup_{t\ra\infty} d_t^{\mu,m}=0.
\]
Combining this with \eqref{eq:total variation convergence on compacts for BCBM}, we have established \eqref{eq:Perron-Frobenius for dual}. 

\smallskip

\noindent
{\bf Step 3: Checking the assumptions of Harris' ergodic theorem. }Finally, we check Assumptions 1-2 of \cite[Theorem 1.2]{Hairer2011} for the discrete-time Markov chain $(\overline{Q}_{n})_{n\in \Z_{\ge 0}}$, for $h\in (0,\infty)$ is arbitrarily fixed, to complete the proof. Without loss of generality, we let $h=1$.
For $z\in \chi$ we define $N(z)$ to be the total number of particles in the configuration $z$, that is
\[
N(z):=n+m\quad\text{whereby}\quad z=((x_1,\ldots,x_n),(y_1,\ldots,y_m))\in \chi.
\]
From Gronwall's inequality applied to \eqref{eq:difference between number with and without branching}, we see that $\sup_{\bar{x}\in \mathbb{S}^n}{\bf E}_{\bar{x}}\left[N^{(\beta)}_t\right]\leq n(1+\beta e^{\beta t})$ for all $n\in\mathbb{N}$ and $t\in\R_{\ge 0}$, where $N^{(\beta)}_t=N(Z_t)$ is the number of particles of the 2-type BCBM at time $t$. Hence, taking $\epsilon=\frac{\lambda}{2}$ in \eqref{supx_Expectednt} in Lemma \ref{L:supx_nt}, there exists a constant $n' =n'(\lambda)$ such that
\[
Q_1N(z) \leq {\bf E}_{z}[N(Z_1)] \leq  \frac{\lambda}{2}N(z) +  n'(1+\beta e^{\beta})\,\Ind_{\{N(z)\leq n'\}} \quad \text{for}\quad z\in\chi.
\]
Since $\phi=\phi^{\beta}$ is strictly positive and continuous, We can define the Lyapunov function $V$ and the constant $C<\infty$ by 
\begin{equation}\label{Def:V and C}
V:=\frac{N}{\phi},\quad C:= \frac{n'(1+\beta e^{\beta})}{\lambda}\,\frac{1}{\inf\{\phi(z):\,N(z)\leq n'\}}
\end{equation}
which, by the above inequality, satisfy
\begin{align}
\overline{Q}_1V(z) =\frac{Q_1N (z)}{\lambda\,\phi(z)} \leq \frac{1}{2}V(z) + C.
\end{align}
Therefore $\overline{Q}_1$ satisfies \cite[Assumption 1]{Hairer2011}. 


Next, we  check that $\overline{Q}_1$ also satisfies the Dobrushin condition, \cite[Assumption 2]{Hairer2011}. That is, we check that
there exist a constant $a_1\in(0,1)$ and $\nu\in \calP(\chi)$ such that
\begin{equation}\label{A:assumption2inHM11}
\inf_{\{z\in\chi:\,V(z)\leq C_1\}}\overline{Q}_1(z,\cdot)\geq a_1\,\nu (\cdot)
\end{equation}
for some $C_1> 4C$, where $C$ is the constant in \eqref{Def:V and C}. 

For this, we note that $\{z\in\chi:\,V(z)\leq C_1\}\subset K_M:=\cup_{k\leq M}\chi_k$, where $M$ is any integer greater than $C_1 \|\phi\|_{\infty}$ and $\chi_k:=\{z\in\chi:\,N(z)=k\}$. It is clear that $\inf_{z\in K_M}\overline{Q}_{\frac{1}{2}}(z,\chi_2)>0$. Moreover, it follows from the parabolic Harnack inequality and the fact that $\phi$ is bounded and bounded away from $0$ that there exists $c_0>0$ and an open set $U\subseteq \chi_2$  such that $\overline{Q}_{\frac{1}{2}}(z,dz')\geq c_0\Leb_{\lvert_U}(dz')$ for all $z\in \chi_2$, where $\Leb_{\lvert_U}$ is Lebesgue measure restricted to $U$. It follows that
\[
\overline{Q}_1(z,dz')\geq \inf_{z''\in K_M}\overline{Q}_{\frac{1}{2}}(z'',\chi_2)c_0\Leb_{\lvert_U}(dz')\quad\text{for all}\quad z\in K_M,
\]
whence we obtain \eqref{A:assumption2inHM11} and thus $\overline{Q}_1$ satisfies \cite[Assumption 2]{Hairer2011}. 

The proof of Proposition \ref{prop:convergence to QSD for 2-type BCBM} is complete.
\qed

\subsubsection*{Proof of Proposition \ref{prop:uniqueness QSD for FKPP}}

We fix arbitrary $\mu\in\calP(\calC_*)$. Using the simple fact $(\mu P_t)(\calE^{z})=\mu(P_t \calE^z)$, we obtain
 \begin{align}
(\mu P_t)(\calE^{z})
=\mu(P_t \calE^z)
=&\int_{\calC_*} (P_t\calE^{z})(f)\,\mu(df)
\stackrel{ (\ref{eq:moment duality relationship for 2-type results})}{=}
 \,\int_{\calC_*}\left(Q_t\calE^{\bullet}(f)\right)(z)\,\mu(df) \notag\\
=&\,\int_{\chi} \mu(\calE^{w})\,Q_t(z, dw) \,=\,Q_t\left(\mu(\calE^{\bullet})\right)(z). \label{E:moment duality_2}
 \end{align}

It therefore follows 
from \eqref{E:moment duality_2} and \eqref{eq:Perron-Frobenius for dual} 
that for all $z\in \chi$, as $t\to\infty$,
\begin{align}\label{eq:convergence of moments of FKPP ic mu}
\lambda^{-t}(\mu P^{\beta}_t)(\calE^{z}) 
\stackrel{(\ref{E:moment duality_2})}{=}
\lambda^{-t}Q^{\beta}_t\left(\mu(\calE^{\bullet})\right)(z)
\stackrel{(\ref{eq:Perron-Frobenius for dual})}{\rightarrow}
 \phi^{\beta}(z)\,\varphi^{\beta}\left(\mu(\calE^{\bullet})\right)\,=\,\pi^{\beta}(\calE^{z})\,\mu(h^{\beta}) \in (0,\infty),
\end{align}
where the last equality follows from the facts $\phi^{\beta}(z)=\pi^{\beta}(\calE^{z})$ and $\varphi^{\beta}\left(\mu(\calE^{\bullet})\right)=\mu(h^{\beta})$ according to \eqref{eq:formula for right efn of FKPP in terms of QSD of dual results} and 
\eqref{eq:formula for right efn of dual in terms of QSD of FKPP results} respectively.
Since both $\phi^{\beta}$ and $h^{\beta}$ are everywhere strictly positive (by Proposition \ref{prop:duality relationship for quasi-stationarity}), we see that the limit in \eqref{eq:convergence of moments of FKPP ic mu} must be strictly positive for any $z\in \chi$ and $\beta\in\R_{\geq 0}$. 

From \eqref{eq:convergence of moments of FKPP ic mu} and Lemma \ref{lem:moments determine measures},
we see that the QSD for the stochastic FKPP is uniquely determined by \eqref{QSD_FKPP_char} for each $\beta\geq 0$.

\qed

\subsection{Proofs for Section \ref{S:Insights}}\label{S:InsightsProofs}

We keep in mind that in all the proofs for Section \ref{S:Insights}, $\beta=0$ whilst $\alpha,\gamma\in \Rm_{>0}$ are fixed and arbitrary.

For the proof below we identify $\mathbb{S}\simeq \Rm/\Zm$ and define the operations $\pm$ in the standard manner.

\subsubsection*{Proof of Theorem \ref{thm:explicit expressions when neutral}}
We have established in Propositions \ref{prop:existence of QSD for FKPP beta=0} and \ref{prop:convergence to QSD for 2-type BCBM} that the  $2$-type CBM has a unique QSD $\varphi^0$ which is supported on $\bfS\times \bfS$, and that $\varphi^0_{\lvert_{\bfS\times \bfS\setminus \Gamma}}$ has a $\calC(\bfS\times \bfS\setminus \Gamma;\Rm_{>0})$ density on $\bfS\times \bfS\setminus \Gamma$. In fact, in the proof of Proposition \ref{prop:existence of QSD for FKPP beta=0} we establish more than this, we establish that $\varphi^0_{\lvert_{\bfS\times \bfS\setminus \Gamma}}$ has a $C^{\infty}(\bfS\times \bfS\setminus \Gamma;\Rm_{>0})$ density which is a classical solution of $\frac{\alpha}{2}\Delta\varphi^0=\kappa_0\varphi^0$ (abusing notation by writing $\varphi^0$ for its density), whereby $\kappa_0:=-\ln(\lambda_0)$. Moreover since the transition kernel of $(G_t,R_t)_{t<\tau_{\partial}}$ is dominated by that of Brownian motion on $\bfS\times \bfS$ without killing, it follows that $\varphi^0$ has a bounded density everywhere, so $\varphi$ has a density on $\bfS\times \bfS$ belonging to $\calB_b(\bfS\times \bfS)\cap \calC_b(\bfS\times \bfS\setminus \Gamma)$. Since $\varphi^0(\Gamma)=0$, we may define $\varphi^0$ to be a constant $c$ on $\Gamma$. We shall choose suitable $c$ below, and we denote this density as $\rho^0$.

Next, we observe that  there exists a function $f:\bfS\ra \Rm$ such that $\rho^0(x,y)=f(y-x)$ for $x,y\in \bfS$.
This is because $\varphi^0$ is the unique QSD, and $\rho^0$ is invariant under reflection about the diagonal $\Gamma$ and under shifting along $\Gamma$ (by symmetry), in the sense that
\begin{equation}\label{eq:symmetries of QSD}
\rho^0(x,y)=\rho^0(y,x)\quad \text{and}\qquad \rho^0(x+z,y+z)=\rho^0(x,y) \quad \text{for all}\quad x,y,z\in \bfS.
\end{equation}
This function $f$ is
continuous on $\bfS\setminus \{0\}$, equal to $c$ at $0$, and satisfies $f(x)=f(1-x)$ for $x\in \bfS\setminus \{0\} \simeq(0,\,1)$.

It remains to find this function $f$.
Recall that $\rho^0$ is a classical solution on $\bfS\times \bfS\setminus \Gamma$ of $\frac{\alpha}{2}\Delta\rho^0=-\kappa_0 \rho$, where $\kappa_0:=-\ln(\lambda_0)$. It follows that $f$ is a non-negative solution of $\Delta f=-\frac{\kappa_0}{\alpha} f$ on $\bfS\setminus \{0\}$. Upon solving this equation, we obtain
\[
f(u)=A\sin(\pi(a+(1-2a)u)),\quad u\in \bfS\setminus \{0\}
\]
for some constant $a\in [0,\frac{1}{2}]$, where $A$ is the normalising constant
\[
A=\frac{\pi (1-2a)}{2\cos(\pi a)},
\]
so that $\int_0^1f(u)du=1$ (as QSD must have mass $1$).
Therefore, we choose $c:=A\sin(\pi a)$, so that  both $\rho^0$ and $f$ are  continuous everywhere it their respective domains.
It remains to find the constant $a$.


Observe that the process $(G_t-R_t)_{0\leq t<\tau_{\partial}}$ is a rate $2\alpha$ Brownian motion $B_t$ on the circle $\mathbb{S}\simeq [0,1)$, 
killed at rate $\frac{\gamma }{2\alpha}dL^{0}_t$ 
(i.e. at rate $\frac{\gamma}{2\alpha}$ according to the local time $L^0_t$ at $0\in\mathbb{S}$). Then for all test functions $g\in \calC_b(\bfS)\cap C^{\infty}(\bfS\setminus \{0\})$ such that $g_+(0)=-g_-(0)$, we have that
\[
\begin{split}
g(B_t)\Ind(\tau_{\partial}>t)-g(B_0)-\int_0^t\alpha \Delta g(B_s)\Ind(\tau_{\partial}>s)ds
+\int_0^t\Ind(\tau_{\partial}>s)\Big[\frac{\gamma}{2\alpha}g(0)-g'(0_+)\Big] dL^0_s
\end{split}
\]
is a martingale. By specifying also that $g'(0_+)=\frac{\gamma}{2\alpha}g(0)>0$ and taking $(G_0,R_0)\sim \varphi_0$ (so that $B_0\sim f$), we obtain that for all $t\geq 0$ and all such test functions $g$,
\[
e^{-\kappa_0 t}\langle f,g\rangle -\langle f,g\rangle =\int_0^te^{-\kappa_0s}\alpha \langle f,\Delta g\rangle ds.
\]
It follows that $\alpha \langle \Delta f,g\rangle=-\kappa_0\langle f,g\rangle=\alpha \langle f,\Delta g\rangle$ for all such test functions $g$. It follows from integration by parts that for all such test functions $g$ we have that
\[
g(0)[f(1_-)-f(0_+)]=f(0)[g'(1_-)-g'(0_+)]=\frac{-\gamma}{\alpha} g(0)f(0).
\]
Therefore $f'(0_+)-f'(1_-)=\frac{\gamma}{\alpha} f(0)$. Therefore we have that $a$ satisfies
\[
2(1-2a)\pi \cos(\pi a)=\frac{\gamma}{\alpha} \sin(\pi a).
\]
We now take $\theta_{\ast}:=(\frac{1}{2}-a)\pi$. It then follows that 
\[
f(u)=A\cos (2  \theta_{\ast}(\frac{1}{2}-u)), \quad A=\frac{ \theta_{\ast}}{\sin(\theta_{\ast})},\quad \theta_{\ast}\tan(\theta_{\ast})=\frac{\gamma}{4\alpha}.
\]

We have now established \eqref{eq:QSD density for 2-particle CBM} and \eqref{eq:principal eigenvalue no selection}. 

Since $\bfS\times \bfS$ is a closed communication class, $\phi^0_{\lvert_{\bfS\times \bfS}}$ is a right eigenfunction for the killed $2$-type CBM restricted to the state space $\bfS\times \bfS$. It follows that $\phi^0_{\lvert_{\bfS\times \bfS}}$ is a solution to $\frac{1}{2}\alpha\Delta y=-\kappa_0 y$ on $\bfS \times \bfS\setminus \Gamma$. Moreover $\phi^0_{\lvert_{\bfS\times \bfS}}$ must have the same symmetries as $\rho^0$ in \eqref{eq:symmetries of QSD}. It follows that $\phi^0_{\lvert_{\bfS\times \bfS}}=c\rho^0$ for some scaling constant $c>0$. We therefore obtain \eqref{eq:expression for right e-fn of CBM}.

By Lemma \ref{lem:early stopping reduces spectral radius},  the process
$\lambda^{-t}\phi^0(Z_t)\Ind(\tau_{\partial}>t)$ 
is a martingale, and ${\bf E}_{z}\left[e^{\alpha \theta_{\ast}^2\tau^{Z}}\right]<\infty$, for all $z\in\chi$. By the optional stopping theorem, 
\begin{equation}\label{E:phiz_proof}
\phi^{0}(z)={\bf E}_{z}\left[\phi^{0}(Z_{\tau^Z})e^{4\alpha \theta_{\ast}^2\tau^Z} \right]=M_{\ast}\cos(\theta_{\ast})\,{\bf E}_{z}\left[e^{4\alpha \theta_{\ast}^2\tau^{Z}}\right], \quad z\in \chi,
\end{equation}
where the second equality in  \eqref{E:phiz_proof} follows from the observation that $\phi^{0}((x,x))=M_{\ast}\cos(\theta_{\ast})$ for all $x\in \bfS$, by \eqref{eq:expression for right e-fn of CBM}. Hence \eqref{E:phiz} holds.

\qed

\subsubsection*{Proof of Theorem \ref{T:M*}}

Our goal is  to show that $M_{\ast}$ in \eqref{eq:expression for right e-fn of CBM} is given by \eqref{eq:Mast constant}.

We take an arbitrary dense sequence in $\bfS$, $\underline{x}=(x_1,\ldots)\in \bfS^{\mathbb{N}}$, and the right eigenfunction $\phi^0$ given by 
\eqref{eq:formula for right efn of FKPP in terms of QSD of dual results}. 
Let $z^{(n)}=(x_1,(x_2,\ldots,x_n))$.  As $n\ra\infty$,
\begin{equation}\label{eq:phi0(zn) limit}
\phi^0(z^{(n)})=\expE_{u\sim \pi^0}\left[(1-u(x_1))\Big(1-\prod_{j=2}^n(1-u(x_j))\Big)\right]\ra \expE_{u\sim \pi^0}[1-u(x_1)]=\frac{1}{2},
\end{equation}
where we used the fact that $\expE_{u\sim \pi^0}[u(x)]=\frac{1}{2}$ for all $x\in \bfS$ by  symmetry. 

Consider a 2-type CBM $Z$
with initial condition $z^{(n)}$, and recall that  $\tau^{Z}$ is the first time that $Z^{(n)}$ consists of one green and one red particle, both at the same position. 
Using the fact that $Z$ can only have one green particle at all times, we  observe the following. Suppose that
we remove the color designation from the particle system,
obtaining a $1$-type coalescing Brownian motion up to the time $\tau^{(n)}$, the first time when there are only $2$ particles, and both are at the same position. Then
$\tau^Z$ and $\tau^{(n)}$ have the same law.
It therefore follows from \eqref{E:phiz} that for all $n\geq 2$, 
\begin{equation}\label{E:phizn}
\phi^{0}(z^{(n)})=M_{\ast}\cos(\theta_{\ast})\,{\bf E}_{z^{(n)}}[e^{4\alpha \theta_{\ast}^2\tau^{(n)}}].
\end{equation}

By \eqref{eq:phi0(zn) limit} and \eqref{E:phizn}, we obtain that the limit $\lim_{n\ra\infty}{\bf E}_{z^{(n)}}[e^{4\alpha \theta_{\ast}^2\tau^{(n)}}]$ exists in $\R$ and satisfies
\begin{equation}\label{denseS}
\frac{1}{2}=M_{\ast}\cos(\theta_{\ast})\lim_{n\ra\infty}{\bf E}_{z^{(n)}}[e^{4\alpha \theta_{\ast}^2\tau^{(n)}}].
\end{equation}
We therefore obtain \eqref{eq:Mast constant}. Since $M_{\ast}\cos(\theta_{\ast})$ cannot depend upon the choice of $\underline{x}$, it follows that the value of $\lim_{n\ra\infty}{\bf E}_{z^{(n)}}[e^{\alpha \theta_{\ast}^2\tau^{(n)}}]$ lies in $[1,\infty)$ and does not depend upon the choice of $\underline{x}\in \Sigma_{\mathbb{S}}$. 

\qed

\subsubsection*{Proof of Theorem \ref{thm:neutral QSD doesn't fixate on closed set}}

The proof of Theorem \ref{thm:neutral QSD doesn't fixate on closed set} follows in the same manner as that of
Theorem \ref{T:M*}. We fix the non-empty closed set $F\subset \bfS$ and an element $\underline{x}=(x_1,\ldots)\in \Sigma_F$. 
As before we  define $z^{(n)}:=(x_1,(x_2,\ldots,x_n))$ for all $n\geq 2$. Then for all $u\in \calC_{\ast}$,
\begin{equation}\label{Conv_product1}
\prod_{j=2}^n(1-u(x_j))\ra 
\begin{cases}
    1,\quad u\equiv 0\text{ on $F$}\\
    0,\quad\text{otherwise}
\end{cases}\quad\text{as}\quad n\ra\infty.
\end{equation}
The above convergence to $0$ follows because $\sum_{j=2}^{n}u(x_j)$ diverges due to the fact that any point in $F$ is an accumulation point in $\underline{x}$ by our definition of $\Sigma_F$.

Therefore, as $n\ra\infty$, as in \eqref{eq:phi0(zn) limit} we have
\begin{equation}\label{Conv_product2}
\phi^0(z^{(n)})=\expE_{u\sim \pi^0}\left[(1-u(x_1))\Big(1-\prod_{j=2}^n(1-u(x_j))\Big)\right]\ra \expE_{u\sim \pi^0}\big[(1-u(x_1))\Ind(\{u\equiv 0\text{ on $F$}\}^c)\big],
\end{equation}
where  $\phi^0$ is the right eigenfunction given by \eqref{eq:formula for right efn of FKPP in terms of QSD of dual results} as before.

By the above convergence and 
the fact that $\expE_{u\sim \pi^0}[1-u(x_1)]=\frac{1}{2}$ for all $x_1\in \bfS$, we have
\begin{equation}\label{eq:prob of QSD being 0 on F}
\Pm_{u\sim \pi^0}(u\equiv 0\text{ on $F$})=\expE_{u\sim \pi^0}[(1-u(x_1))\Ind(u\equiv 0\text{ on $F$})]=\frac{1}{2}-\lim_{n\ra\infty}\phi^{0}(z^{(n)}),
\end{equation}
where the first equality follows since $x_1\in F$.

As in \eqref{E:phizn},  
$\phi^{0}(z^{(n)})=M_{\ast}\cos(\theta_{\ast})\,{\bf E}_{z^{(n)}}[e^{4\alpha\theta_{\ast}^2\tau^{(n)}_1}]$ 
for all $n\geq 2$. Since the left-hand side of \eqref{eq:prob of QSD being 0 on F} does not depend upon the choice of $\underline{x}\in \Sigma_F$, it follows that the limit 
$\lim_{n\ra\infty}{\bf E}_{z^{(n)}}[e^{4\alpha\theta_{\ast}^2\tau^{(n)}_1}]$ exists and does not depend upon the choice of $\underline{x}\in \Sigma_F$. It belongs to $[1,\infty]$  by definition, and it must be finite due to \eqref{eq:prob of QSD being 0 on F}.

We have established that ${\bf E}_{F}\left[e^{4\alpha\theta_{\ast}^2\tau_1}\right]$ is well-defined and takes values in $\Rm_{\geq 1}$ and satisfies
\begin{equation}\label{denseF}
\Pm_{u\sim \pi^0}(u\equiv 0\text{ on $F$})=\frac{1}{2}-M_{\ast}\cos(\theta_{\ast})\,{\bf E}_{F}\left[e^{4\alpha\theta_{\ast}^2\tau_1}\right]
\end{equation}
for all closed set $F\subset \mathbb{S}$,
which  generalizes \eqref{denseS}.

We recall that $M_{\ast}=\frac{1}{2\cos(\theta_{\ast})\,{\bf E}_{\bfS}\left[e^{4\alpha\theta_{\ast}^2\tau_1}\right]}$. It follows from \eqref{denseF} that \eqref{eq:prob neutral QSD zero on F} holds.
Since $\Pm_{u\sim \pi^0}(u\equiv 0\text{ on $F$})=\Pm_{u\sim \pi^0}(u\equiv 1\text{ on $F$})$ by symmetry, we have established the equality in \eqref{eq:prob neutral QSD not fixed on F}. The value in \eqref{eq:prob neutral QSD not fixed on F} lies in the interval $(0,1)$ for the following reason. 

\smallskip

\begin{remark}\rm\label{rmk:proof of Lemma first invocation}
At this point we would like to invoke Lemma \ref{L:Laplace_tau1}. We will, therefore, not make use of any of the rest of the following proof when we come to prove Lemma \ref{L:Laplace_tau1}.
\end{remark}

\smallskip

We have that ${\bf E}_{F}\left[e^{4\alpha\theta_{\ast}^2\tau_1}\right]\in[1,\infty)$ for all non-empty closed $F$ by Lemma \ref{L:Laplace_tau1}, so that by \eqref{ineq:EF strictly increasing in F} we have
\begin{equation}
    \Pm_{u\sim \pi^0}(u\text{ is fixed on $F$})\in (0,1)\qquad\text{for all closed }\qquad \emptyset\neq F\subsetneq \bfS.
\end{equation}

Finally, for all closed subsets $\emptyset \neq F\subsetneq F'\subseteq \mathbb{S}$, it holds that
\begin{align}
&\Pm(\text{$u$ is fixed on $F$ but not on $F'$})\\
=&\,\Pm(\text{$u$ is fixed on $F$})-\Pm(\text{$u$ is fixed on $F'$})\\
=&\,\Pm(\text{$u$ is not fixed on $F'$})-\Pm(\text{$u$ is not fixed on $F$}).
\end{align}

It therefore follows from  \eqref{ineq:EF strictly increasing in F} that \eqref{eq:probabliity of fixation on set but not on larger set strictly positive} holds. The proof of Theorem \ref{thm:neutral QSD doesn't fixate on closed set} is complete.

\qed
\subsubsection*{Proof of Lemma \ref{L:Laplace_tau1}}
Lemma \ref{L:Laplace_tau1} was employed in the proof of Theorem \ref{thm:neutral QSD doesn't fixate on closed set}, immediately after Remark \ref{rmk:proof of Lemma first invocation}. Consequentially, while we shall employ elements from the proof of Theorem \ref{thm:neutral QSD doesn't fixate on closed set} prior to Remark \ref{rmk:proof of Lemma first invocation}, we must be careful to avoid using any elements thereafter.

The fact that  ${\bf E}_{F}\left[e^{4\alpha\theta_{\ast}^2\tau_1}\right]={\bf E}_{\underline{x}}\left[e^{4\alpha\theta_{\ast}^2\tau_1}\right]\in [1,\infty)$ is well-defined and is the same for all
$\underline{x}\in \Sigma_F$
follows directly from the proof of Theorem \ref{thm:neutral QSD doesn't fixate on closed set}, as mentioned in the paragraph immediately after \eqref{eq:prob of QSD being 0 on F}. This number is strictly larger than 1 since $\tau_1>0$  almost surely even if the CBM starts with only three particles. 

By \eqref{denseF}, we have the monotonicity   
\begin{equation}\label{ineq:EF increasing in F}
{\bf E}_{F}\left[e^{4\alpha\theta_{\ast}^2\tau_1}\right]\leq {\bf E}_{F'}\left[e^{4\alpha\theta_{\ast}^2\tau_1}\right] \qquad \text{whenever }F\subset F'\subset \bfS.
\end{equation}
It remains to show the strict inequality \eqref{ineq:EF strictly increasing in F}. By the monotonicity \eqref{ineq:EF increasing in F}, 
it suffices to show this strict inequality when
$F'=F\cup \{w\}$ for some $w\notin F$. 
We establish this for the rest of this proof, by using the fact that $1$-type coalescing Brownian motion (CBM) comes down from infinity  \cite{hobson2005duality,barnes2022coming}.

We take $\underline{x}=(x_1,\ldots)\in \Sigma_F$ and $\underline{y}=(y_1,\ldots)\in \Sigma_{F'}$, the latter defined by
\[
y_k=\begin{cases}
    x_{\frac{k}{2}},\quad &k\text{ even}\\
    w,\quad &k\text{ odd}
\end{cases}.
\]
We define $\bar X_t$ and $\bar Y_t$ to be two $1$-type CBMs with entrance laws $\underline{x}$ and $\underline{y}$ respectively, as described in Remark \ref{Rk:EntranceLaw}, possibly in two probability spaces. Each of these systems of $1$-type CBM comes down from infinity  \cite{hobson2005duality,barnes2022coming} and thus has state space
\begin{equation}
\Xi:=\cup_{n\geq 1}\bfS^n/\sim
\end{equation}
at all positive times, where $\sim$ is the equivalence relationship on $\bfS^n$ such that
$\underline{x}=(x_1,\ldots,x_n)\sim \underline{y}=(y_1,\ldots,y_n)$ if $\underline{x}$ can be obtained from $\underline{y}$ by permuting the coordinates. 


We define on $\Xi$ the partial order $\underline{x}=(x_1,\ldots,x_n)\leq \underline{y}=(y_1,\ldots,y_m)$ if the multi-set $[x_1,\ldots,x_n]$ is a subset (as a multi-set) of $[y_1,\ldots,y_m]$ (recall that multi-sets count multiplicities, with the notion of subset defined accordingly). 
For two probability measures $\mu_1,\mu_2\in \calP(E)$, we say that $\mu_2$ stochastically dominates $\mu_1$, written $\mu_1\leq_{\st}\mu_2$, if 
there exists a $\Xi\times \Xi$-valued random variable $(A,B)$ (defined on some probability space) such that $A\sim \mu_1$, $B\sim \mu_2$ and $A\leq B$ almost surely. If, in addition, there is a positive probability that $A\neq B$ (equivalently, if also $\mu_1\neq \mu_2$), then we say that $\mu_2$ strictly stochastically dominates $\mu_1$ and write
 $\mu_1<_{\st}\mu_2$. We shall employ the following simple lemma.
 \begin{lemma}\label{lem:stochastic dominance preserved in limits}
     Suppose that $(\mu^{(n)}_1)_{n\in \mathbb{N}}$ and $(\mu^{(n)}_2)_{n\in \mathbb{N}}$ are two sequences in $\calP(\Xi)$  that  converge in $\calP(\Xi)$ to  $\mu_1$ and $\mu_2$ respectively, and  $\mu_1^{(n)}\leq_{\st}\mu_2^{(n)}$ for all $n\in \mathbb{N}$. Then $\mu_1\leq_{\st}\mu_2$.
 \end{lemma}

 \begin{proof}[Proof of Lemma \ref{lem:stochastic dominance preserved in limits}]
We define $F:=\{(x,y)\in \Xi\times \Xi:x\leq y\}$, which we observe is a closed subset of $\Xi\times \Xi$, so is a complete and separable metric space. Then for each $n\in\mathbb{N}$ there exists a coupling of $\mu^{(n)}_1$ and $\mu^{(n)}_2$, $(A^{(n)},B^{(n)})$, supported on $F$. Since we have convergence in distribution of each marginal, $\{\Law((A^{(n)},B^{(n)})):\,n\in\mathbb{N}\}$ is tight. Taking a subsequential limit (using Prokhorov's theorem) we obtain a coupling $(A,B)$ of $\mu_1$ and $\mu_2$ supported on $F$.   
 \end{proof}
 
For $n\in \mathbb{N}$ we write $\bar X^{(n)}_t$ and ${\bar Y}^{(n)}_t$ for coalescing Brownian motions started from $(x_1,\ldots,x_n)$ and $(y_1,\ldots,y_{2n})$ respectively. Since $\underline{x}\leq \underline{y}$, by permuting the elements of $\underline{y}$ so that $\underline{y}=(x_1,\ldots,x_n,w,\ldots,w)$ and declaring that particle $i$ kills particle $j$ for $i<j$ (as in Remark \ref{Rk:EntranceLaw}), we see that $\bar X^{(n)}_t\leq \bar Y^{(n)}_t$ almost surely
for all $n\in\mathbb{N}$ and $t>0$, so that $\Law(\bar X^{(n)}_t)\leq_{\st} \Law(\bar Y^{(n)}_t)$ for all $n\in\mathbb{N}$ and $t>0$. Then since $\bar X^{(n)}_t$ (respectively $\bar Y^{(n)}_t$) converges in distribution to $\bar X_t$ (respectively $\bar Y_t$) for fixed $t>0$, it follows from Lemma \ref{lem:stochastic dominance preserved in limits} that $\Law(\bar X_t)\leq_{\st}\Law(\bar Y_t)$.

\commentout{
 
 We let $\bar X^{(n)}$ (respectively $\bar X$
 
coupling for two random variables $A$ and $B$, we say $B$ stochastically dominates $A$ and write
 $A\leq_{\st}B$ if
$\P(B\leq x)\leq \P(A\leq x)$ for all $x\in \Xi$.
If, furthermore, $\P(B\leq x) <\P(A\leq x)$ for some $x\in \Xi$, we say that $B$ strictly stochastically dominates $A$ and write
 $A<_{\st}B$.


By construction, we see that $\bar X_t\leq_{\st}X_t'$ for all $t>0$. }

Clearly if $\bar X_t\eqd \bar Y_t$ for some $t>0$, then $\bar X_s\eqd \bar Y_s$ for all $s\in (t,\infty)$, where $\eqd$ denotes equality in distribution. We claim that if $\Law(\bar X_t)<_{\st} \Law(\bar Y_t)$ for some $t>0$, then $\Law(\bar X_s)<_{\st} \Law(\bar Y_s)$ for all $s\in (t,\infty)$. To see this, take a coupling $(A^{(t)},B^{(t)})$ of $\Law(\bar X_t)$ and $\Law(\bar Y_t)$ such that $A^{(t)}\leq \bar B^{(t)}$ almost surely. We can permute the elements of $B^{(t)}$ as in the previous paragraph, as $B^{(t)}\in \Xi$ has finitely many particles (by coming down from infinity). We initiate coalescing Brownian motions $(\bar X^{(t)}_s)_{s\geq t}$ and $(\bar Y^{(t)}_s)_{s\geq t}$ from time $t$ and initial conditions $A^{(t)}$ and $B^{(t)}$ respectively, coupled as in the previous paragraph. Then for $s> t$, $\bar X^{(t)}_s\leq \bar Y^{(t)}_s$ by construction. Moreover, by assumption there is a positive probability that $A ^{(t)}\neq  B^{(t)}$, in which case $B^{(t)}$ contains at least one extra particle compared to $A^{(t)}$. There is a positive probability that this extra particle then survives up to time $s$, so there is a positive probability that $\bar X^{(t)}_s\neq \bar Y^{(t)}_s$. We have established the claim.


There are therefore two possibilities:
\begin{enumerate}
    \item $\Law(\bar X_t)<_{\st}\Law(\bar Y_t)$ for all $t>0$;\label{enum:possibility strict stochastic domination for all times}
    \item $\bar X_t\eqd \bar Y_t$ for all $t>0$.\label{enum:possibility equal in law for all times}
\end{enumerate}
We suppose for contradiction that possibility \ref{enum:possibility equal in law for all times} holds. Then by moment duality \eqref{eq:moment duality relationship for 2-type results} it would follow that
\[
\expE_{u_0}\left[\prod_{i=1}^{\infty}(1-u_t(x_i))\right]=\expE_{u_0}\left[\prod_{i=1}^{\infty}(1-u_t(y_i))\right],
\]
for any initial condition $u_0$ for the stochastic FKPP. This is equivalent to
\[
\P_{u_0}\left(F\cap {\rm supt}(u_t) =\emptyset,\;w\notin  {\rm supt}(u_t) \right)= \P_{u_0}\left(F\cap {\rm supt}(u_t) =\emptyset\right),
\]
for all $t>0$ and all $u_0\in \calC_*$. This can be rewritten as
\begin{equation}\label{ContradictionE}
\Pm_{u_0}\left(w\in {\rm supt}(u_t)\subseteq F^c\right)=0,
\end{equation}
for all $t>0$ and all $u_0\in \calC_*$. We see that this is false because
the support
${\rm supt}(u_t)$ and ${\rm supt}(1-u_t)$ are stochastically continuous in $t$, by 
the argument in \cite[Proposition 3.2]{tribe1995large}. More precisely,
by first choosing $u_0\in \calC_*$ to approximate the indicator function of a small neighborhood $V_w$ of $w$ (for example, when
 $u_0=1$ in a connected open neighborhood $V_w\subset \bfS\setminus F$ 
and $F\cap {\rm supt}(u_0) =\emptyset$), and then choosing $t>0$ small enough, the probability $Pm_{u_0}\left(w\in {\rm supt}(u_t)\subseteq F^c\right)$ is strictly positive.

It follows that we cannot have possibility \ref{enum:possibility equal in law for all times}, hence possibility \ref{enum:possibility strict stochastic domination for all times} holds (i.e. $\Law(\bar X_t)<_{\st}\Law(\bar Y_t)$ for all $t>0$). It follows from this (by coupling and using coming down from infinity) that $\Law_{F'}(\tau_1)>_{\st}\Law_F(\tau_1)$, so that we have the strict inequality \eqref{ineq:EF strictly increasing in F}.
\qed

\subsubsection*{Proof of Corollary \ref{cor:martingale corollary}}

Let $\phi^0$ be the right eigenfunction defined in \eqref{eq:formula for right efn of FKPP in terms of QSD of dual results}. The killed $2$-type CBM is $(Z_t)_{0\leq t<\tau_{\partial}}$. Then the process
$\Lambda(\phi^0)^{-t}\phi^0(Z_t)\Ind(\tau_{\partial}>t)$
is a martingale, by Lemma \ref{lem:early stopping reduces spectral radius}. We recall from Theorem \ref{thm:explicit expressions when neutral} that $\Lambda(\phi^0)=e^{-4\alpha\theta_{\ast}^2}$.
Let $\overline{\phi^0}$ be the extension of $\phi^0$  that vanishes on $\bfS\times \{\emptyset\}$. Then $\phi^0(Z_t)=0$ for $t\geq \tau_{\partial}$ and
\[
e^{4\alpha\theta_{\ast}^2}\,\overline{\phi^0}(Z_t)
\]
is a martingale with respect to the natural filtration of $Z$ (with no killing) for all time.

The subset $\bfS\times(\cup_{m\geq 1}\bfS^m/\sim)\subset \chi$, corresponding to there only being one green particle, is closed for coalescing Brownian motion. For $n\geq 2$ and $z=(x_1,(x_2,\ldots,x_n))\in \bfS\times(\cup_{m\geq 1}\bfS^m/\sim)$ we have that
\begin{equation}\label{phi0formula}
\phi^0(z)=\expE_{u\sim \pi^0}\left[(1-u(x_1))-\prod_{i=1}^n(1-u(x_i)) \right]=\frac{1}{2}-\expE_{u\sim \pi^0}\left[\prod_{i=1}^n(1-u(x_i))\right].
\end{equation}
We note that the expression on the right vanishes if we take $z=(x_1,\emptyset)$ (i.e. $n=1$). Hence the extended function $\overline{\phi^0}(z)$ is equal to the right hand side of \eqref{phi0formula}. 


Given that $Z_t=(X^1_t,(X^2_t,\ldots,X^{N_t}_t))$, we can remove the colour designation and consider the usual 1-type coalescing Brownian motion $\vec{X}_t:=(X^1_t,\ldots,X^{N_t}_t)$. It follows that the process
\[
e^{4\alpha \theta_{\ast}^2t}\left(\frac{1}{2}-{\bf E}_{u\sim \pi^0}\left[\prod_{i=1}^{N_t}(1-u(X_t^i))\right]\right)\quad\text{is a martingale.}
\]
\qed

\appendix

\section{Appendix}

\subsection{Basic facts about the stochastic FKPP}\label{appendix section:mild solutions}

Consider the stochastic PDE
\begin{equation}\label{SPDE}
   	\partial_t u(t,x)      \,= \frac{\alpha}{2}\Delta u +b(u)+ \sigma(u)\,\dot{W}, \qquad (t,x)\in(0,\infty)\times \mathbb{S},
\end{equation}
where $b$ and $\sigma$ are $\R$-valued Borel measurable functions on $\R$, and
$\dot{W}$ is the space-time Gaussian white noise on $\R_{\geq 0}\times \mathbb{S}$ defined by specifying that $\{W(f):\,f\in L^2(\R_{\geq 0}\times \mathbb{S})\}$ is a Gaussian family with mean zero and covariance 
\[
\E[W(f)\,W(g)]=\int_{0}^{\infty}\int_{\mathbb{S}}f(s,x)\,g(s,x)\,m(dx)\,ds,
\]
where $m(dx)$ denotes the Lebesgue measure on $\mathbb{S}$.

We adopt Walsh's theory \cite{MR876085} to regard the stochastic FKPP equation \eqref{fkpp_X} as a shorthand for an integral equation \eqref{E:MildSol_u} below. 
A process  $u=(u_t)_{t\geq 0}$ taking values in $\mathcal{B}(\mathbb{S};\,\R)$, defined on some probability space $(\Omega,\,\mathcal{F},\,\P)$, 
is said to be a  \textbf{mild solution} to equation \eqref{fkpp_X} with initial condition $u_0$ if there is a space-time white noise $\dot{W}$ such that  $u$ is adapted to the filtration  generated by $\dot{W}$ and that, $\P$-almost surely, $u$ satisfies the integral equation 
\begin{align}\label{E:MildSol_u}
u_t(x)= \int_{\mathbb{S}} p( t,x,y)\,u_0(y)\,m(dy) &+ \int_0^t\int_{\mathbb{S}}p(t-s,x,z)\,b(u_s(z))\,m(dz)\,ds   \notag\\
&+ \int_{\mathbb{S}\times [0,t]}p(t-s,x,z)\,
\sigma\big(u_s(z)\big)\,dW(z,s)
\end{align}
for all $t\in\R_{\ge 0}$, 
where 
$p(t,x,z)=p^{\mathbb{S},\alpha}(t,x,z)$ is the transition density
 of a  Brownian motion $B$  on $\mathbb{S}$ with variance $\alpha$ and
with respect to the 1-dimensional Lebesgue measure $m(dz)$. 

Assume the coefficients $b$ and $\sigma$  are such that there exists a mild solution to \eqref{fkpp_X},  we show in Lemma \ref{L:rescale}  how to pass from an SPDE on the circle
$\mathbb{S}$ to that on another circle $\mathbb{S}_{\ell}=\R/\ell\Z$ with  circumference $\ell$, for arbitrary $\ell\in(0,\infty)$.
\begin{lemma}[Rescaling  on a circle]\label{L:rescale}
Let $u=(u_t)_{t\in \R_{\geq 0}}$ be a mild solution to \eqref{SPDE} and $c,\ell\in(0,\infty)$ be constants. Define
 $v(t,\tilde x):=u\left(\frac{c^2\,t}{\ell^2},\,\frac{\tilde x}{\ell}\right)$ for $(t,\tilde x)\in \R_{\geq 0}\times S_{\ell}$. Then $v$ is a mild solution to 
\begin{equation}\label{SPDE_ell}
   	\partial_t v(t, \tilde x)      \,= \frac{\alpha\,c^2}{2}\Delta v +\frac{c^2}{\ell^2}\,b(v)+ \frac{c}{\sqrt{\ell}}\sigma(v)\,\dot{W}_{\mathbb{S}_\ell}, \qquad (t, \tilde x)\in(0,\infty)\times \mathbb{S}_{\ell}
\end{equation}
with initial condition $v_0(\tilde x)=u_0(\frac{\tilde x}{\ell})$,
where  $\dot{W}_{\mathbb{S}_\ell}$ is the space-time Gaussian white noise on $\R_{\geq 0}\times \mathbb{S}_{\ell}$. 
\end{lemma}

\begin{proof}[Proof of Lemma \ref{L:rescale}]
Identifying $\mathbb{S}_{\ell}$ with the interval $[0,\ell)$, the transition density for the Brownian motion on $\mathbb{S}_{\ell}$ with variance $\alpha$ is explicitly given by
\begin{equation}\label{densityBM_S}
p^{\mathbb{S}_{\ell},\alpha}(t,x,y)=\frac{1}{\sqrt{2\pi \alpha t}} \sum_{k\in\Z} e^{\frac{-(y-x+ k \ell)^2}{2\alpha t}}=p^{\mathbb{S}_{\ell},1}(\alpha t,x,y), \qquad x,y\in [0,\ell)\simeq \mathbb{S}_{\ell}.
\end{equation}

Note that $\mathbb{S}=\mathbb{S}_1$ and $p=p^{\mathbb{S},\alpha}$ in \eqref{E:MildSol_u}. Observe that we have the relations for all $\ell\in(0,\infty)$:
\begin{align}
 p^{\mathbb{S},1}(t,x,y)\,=&\,
\ell\,p^{\mathbb{S}_{\ell},1}(\ell^2 t,\ell x,\ell y), \qquad x,y\in [0,1)\simeq \mathbb{S} \label{cv1}\\
\int_{\mathbb{S}}\phi(y)\,m(dy)\,=&\, \frac{1}{\ell}\int_{\mathbb{S}_{\ell}}\phi(\frac{\tilde y}{\ell})\,m(d\tilde{y}) \label{cv2}\\
\int_0^t\int_{\mathbb{S}}\psi(y,s)\,W_{\mathbb{S}}(dy,ds)\,\eqd\,&\frac{1}{\sqrt{\ell}}\int_0^t\int_{\mathbb{S}_{\ell}}\psi(\frac{\tilde{y}}{\ell},s)\,W_{\mathbb{S}_\ell}(d\tilde{y},ds) \label{cv3}\\
\frac{W_{\mathbb{S}_\ell}(d\tilde{y},a ds)}{\sqrt{a}}\,\eqd\,&W_{\mathbb{S}_\ell}(d\tilde{y}, ds), \qquad a\in (0,\infty) \label{cv4}
\end{align}

The rest of the proof follows from  change of variables, see for instance \cite[Section 4.1]{mueller2021speed}. To give some detail, we let  $p(t,x,z)=p^{\mathbb{S},\alpha}(t,x,z)$. Since $u$ is a mild solution, by
\eqref{cv1} we have
\begin{align}
v_t(\tilde x)=&\, \int_{\mathbb{S}} p( \frac{c^2 t}{\ell^2},\frac{\tilde x}{\ell},y)\,u_0(y)\,m(dy) + \int_0^{\frac{c^2 t}{\ell^2}}\int_{\mathbb{S}}p(\frac{c^2 t}{\ell^2}-s,\frac{\tilde x}{\ell},z)\,b(u_s(z))\,m(dz)\,ds   \notag\\
&\qquad + \int_{\mathbb{S}\times [0,\frac{c^2 t}{\ell^2}]}p(\frac{c^2 t}{\ell^2}-s,\frac{\tilde x}{\ell},z)\,
\sigma\big(u_s(z)\big)\,dW(z,s)\\
=&\,\int_{\mathbb{S}} \ell \,p^{\mathbb{S}_{\ell},\alpha}(c^2 t,\tilde x,\,\ell y)\,u_0(y)\,m(dy) + \int_0^{\frac{c^2 t}{\ell^2}}\int_{\mathbb{S}}\ell\, p^{\mathbb{S}_{\ell},\alpha}(c^2 t-\ell^2 s,\tilde x, \ell z)\,b(u_s(z))\,m(dz)\,ds   \notag\\
&\qquad + \int_{\mathbb{S}\times [0,\frac{c^2 t}{\ell^2}]}\ell\,p^{\mathbb{S}_{\ell},\alpha}(c^2 t -\ell^2 s,\tilde x, \ell z)\,
\sigma\big(u_s(z)\big)\,dW(z,s). \label{E:Mild_v}
\end{align}
The first term in \eqref{E:Mild_v} is, by \eqref{cv2}, $\int_{\mathbb{S_{\ell}}} \,p^{\mathbb{S}_{\ell},\alpha c^2}(r,\tilde x,\,\ell y)\,v_0(\tilde y)\,m(d \tilde y)$. The second  term in \eqref{E:Mild_v} is equal to
$\frac{c^2}{\ell^2}\int_0^{t}\int_{\mathbb{S}_{\ell}}p^{\mathbb{S}_{\ell},\alpha c^2}(t-r,\tilde x, \tilde z)\,b(v_s(\tilde z))\,m(d \tilde z)\,dr$,
by \eqref{cv2} and the change of variable $s=\frac{c^2}{\ell^2} r$.
The third  term in \eqref{E:Mild_v} is, by \eqref{cv3} and then \eqref{cv4},
\[
\frac{1}{\sqrt{\ell}}\int_{\mathbb{S}_{\ell}\times [0,\frac{c^2 t}{\ell^2}]}\ell\,p^{\mathbb{S}_{\ell},\alpha}(c^2 t-\ell^2 s,\tilde x, \tilde z)
\sigma\big(u_s(\frac{\tilde z}{\ell})\big)W_{\mathbb{S}_\ell}(d\tilde z,ds) \eqd \frac{c}{\sqrt{\ell}}
\int_{\mathbb{S}_{\ell}\times [0,t]}p^{\mathbb{S}_{\ell},\alpha c^2}(t-r,\tilde x, \tilde z)
\sigma\big(v_r(\tilde z)\big)W_{\mathbb{S}_\ell}(d\tilde z,dr).
\]
The proof is complete.
\end{proof}

\begin{remark}\rm \label{Rk:wellposedFKPP}
The stochastic FKPP \eqref{fkpp_X} corresponds to $b(u)=\beta\,u(1-u)$ and $\sigma(u)=\sqrt{\gamma\,u(1-u)}$. In this case there exists a  mild solution that is unique in law, for any initial condition $u_0\in \calB(\bfS;[0,1])$; see \cite{MR1271224} and \cite[Remark 1]{hobson2005duality}.

\end{remark}

We now restrict our attention to the stochastic FKPP, so fix $b(u):=\beta u(1-u)$ and $\sigma(u):=\sqrt{\gamma u(1-u)}$ for the time being, for fixed and arbitrary constants $\beta\in \Rm$ and $\gamma\in \Rm_{>0}$ (recall that $\alpha\in \Rm_{>0}$ is also fixed and arbitrary).

\begin{lemma}[Girsanov's transform for FKPP]\label{L:Girsanov}
Let $\P^{\beta}_{t}$ be the measure induced on the
canonical path space 
up to time $t$ by the stochastic FKPP \eqref{fkpp_X} with initial condition $u_0\in \mathcal{B}(\mathbb{S};[0,1])$ and selection coefficient $\beta$ (where $\alpha$ and $\gamma$ are fixed positive numbers).
For all $t\in \R_{\geq 0}$ and for any event $A \in \calF_t$,
\[
\exp{ \left\{-\left(\frac{\beta }{\gamma}+\frac{\beta^2}{8\gamma}t\right) \right\}}\,\P^{0}_{t}(A)\;\leq\; \P^{\beta}_{t}( A) \;\leq\;  \exp{ \left\{\frac{\beta }{\gamma}\right\}}\,\P^{0}_{t}(A).
\]
\end{lemma}

\begin{proof}[Proof of Lemma \ref{L:Girsanov}]
We follow \cite[Section 2.2]{mueller2021speed} to use a version the Girsanov theorem for stochastic PDE in \cite[Theorem IV.1.6]{perkins2002part}. Namely,  $\P^{\beta}_{t}$ is absolutely continuous with respect to $\P^{0}_{t}$ and
\begin{align*}
{\frac{d\P^{\beta}_{t}}{d\P^{0}_{t}}}{\bigg\lvert_{\calF_t}}=&\,
\exp{\left\{\int_0^t\int_{\mathbb{S}}\frac{\beta u_s(1-u_s)}{\sqrt{\gamma u_s(1-u_s)}} W(dy,ds) -\frac{1}{2}\int_0^t\int_{\mathbb{S}}\frac{\beta^2}{\gamma}u_s(1-u_s)\,dy\,ds\right\}}\\
=&\,    
\exp{\left\{ \frac{\beta}{\gamma}\int_{\mathbb{S}}[u_t(y)-u_0(y)]\,dy -\frac{1}{2}\int_0^t\int_{\mathbb{S}}\frac{\beta^2}{\gamma}u_s(1-u_s)\,dy\,ds\right\}},
\end{align*}
where $u$ solves the stochastic FKPP with $\beta=0$; i.e., $\partial_t u= \frac{\alpha}{2}\Delta u + \sqrt{\gamma u(1-u)}\,\dot{W}$. The last display
is bounded 
between
$\exp{ \left\{-\left(\frac{\beta }{\gamma}+\frac{\beta^2}{8\gamma}t\right) \right\}}$ and $\exp{ \left\{\frac{\beta }{\gamma}\right\}}$ almost surely under $\P^{\beta}_{t}$ for all $\beta\in[0,\infty)$, because  $0\leq u(s,x)\leq 1$ for all $(s,x)\in[0,t]\times\mathbb{S}$ almost surely under $\P^{0}_{t}$.
\end{proof}

\begin{lemma}[Initial mass]\label{L:extinct_at_t}
We consider a sequence of initial conditions $(f_n)_{n\in \mathbb{N}}\subset \calC_{\ast}$. Then for any $t>0$ we have:
\begin{itemize}
    \item[(i)] if $\int_{\mathbb{S}}f_n(x)\,dx\to 0$ then $\P_{f_n}(u_t \equiv {\bf 0} ) \to 1$,
\item[(ii)] if $\int_{\mathbb{S}}f_n(x)\,dx\to 1$ then $\P_{f_n}(u_t \equiv {\bf 1} ) \to 1$, and 
\item[(iii)] if $\left(\int_{\mathbb{S}}f_n(x)\,dx\right)\left(1-\int_{\mathbb{S}}f_n(x)\,dx\right)\to 0$ then $\P_{f_n}(u_t \in \{{\bf 0},\,{\bf 1}\} ) \to 1$.
\end{itemize}
\end{lemma}

\begin{proof}[Proof of Lemma \ref{L:extinct_at_t}]
We first prove (i). 
By Girsanov's transform (Lemma \ref{L:Girsanov}), it suffices to  prove this for the case $\beta=0$.
Fix $u_0\in \mathcal{B}(\mathbb{S};\,[0,1])$ and let $u$ be a mild solution to the stochastic FKPP with initial condition $u_0$. 
Below we give a proof using duality and the property of ``coming down from infinity" of the coalescing Brownian motion on the circle \cite{hobson2005duality, barnes2022coming}.

We fix arbitrary $t>0$. For each $n\in\mathbb{N}$, we let $\{X^{(n),i}_0\}$ be a Poisson point process on $\mathbb{S}$ with intensity $n\cdot m(dx)$, independent of $u_0$. By superposition, we can couple $\{X^{(n),i}_0\}$ for all $n\in\mathbb{N}$ such that $\{X^{(n),i}_0\} \subset \{X^{(n+1),i}_0\}$. Then  the set of points $X^{\infty}:=\cup_{n\in\mathbb{N}}\{X^{(n),i}_0\}$ is countable and dense in $\mathbb{S}$ almost surely. 
By the duality in \cite[eqn.(10)]{hobson2005duality},
\begin{equation}\label{E:extinct_at_t}
\P_{u_0}(u_t \equiv {\bf 0} )={\bf E}_{\infty}\left[ \prod_{i=1}^{n_t}(1-u_0(X^i_t)) \right] ,
\end{equation}
where  ${\bf E}_{\infty}$ is the expectation of the system of coalescing Brownian motions with  initial locations $X^{\infty}$,  
averaging over the randomness of both the initial location $X^{\infty}$ and the system of coalescing Brownian motions at time $t$. 


By \cite{hobson2005duality}, the number $n_t$ of particles alive at time $t$ is finite almost surely under ${\bf P}_{\infty}$. In the follwing, we define $n_0:=+\infty$, so that $n_0>k$ for all $k<\infty$.
The right of \eqref{E:extinct_at_t} is therefore equal to
\begin{align}
\sum_{k=1}^{\infty}\sum_{\ell=1}^{\infty} {\bf E}_{\infty}\left[ \prod_{i=1}^k(1-u_0(X^i_t))\,\Big|\,n_t=n_{(1-2^{-\ell})t}=k\right]\,{\bf P}_{\infty}(n_t=k,n_{(1-2^{-\ell})t}=k,n_{(1-2^{-\ell+1})t}>k),
\end{align}
which we claim tends to 1
as $\int_{\mathbb{S}}u_0(x)\,m(dx)\to 0$.

Indeed, ${\bf P}_{\infty}\left((X^1_t,\ldots,X^k_t)\in \cdot \,|\,n_t=n_{(1-2^{-\ell})t}=k \right)$
has a bounded density on $\mathbb{S}^k$ with respect to the Lebesque measure by the parabolic Harnack inequality. It therefore follows that 
\[
{\bf E}_{\infty}\left[ \prod_{i=1}^k(1-u_0(X^i_t))\,\Big|\,n_t=n_{(1-2^{-\ell})t}=k\right]\ra 1 \qquad \text{as }\int_{\mathbb{S}}u_0(x)\,m(dx)\to 0,
\]
whence our claim follows by the bounded convergence theorem.

The proof of  (i) is complete. The proof of (ii) follows by considering $v:=1-u$ and applying (i). Claim (iii) follows from (i) and (ii).

\end{proof}

The  $p$-moment estimate for space and time increments in Lemma \ref{L:Moment_hatu} below is known (see, for instance, \cite[Lemma 4]{fan2017stochastic}). It
implies, via \cite[Theorem 1.1]{MR876085} and by taking $p$ large enough, that the unique solution $u$ to \eqref{fkpp_X} is H\"older continuous with exponent $<1/2$ in space and exponent $<1/4$ in time.

	\begin{lemma}
		\label{L:Moment_hatu}
		For any $0<T_1\leq T_2<\infty$ and $p\in[2,\infty)$, there exists a constant  $C=C_{T_1,T_2,p}\in(0,\infty)$ not depending on the initial condition $u_0$ such that
		\begin{align}
			\label{E:Moment_hatu}
			&\E_{u_0}\left[\,|u(t_1,x_1)- u(t_2,x_2)|^p \,\right] \leq C\,\Big(|t_1-t_2|^{p/4}+|{\rm dist}(x_1,x_2)|^{p/2} \Big) 
		\end{align}
		for all $t_1, t_2\in[T_1,T_2]$, $x_1,\,x_2\in \mathbb{S}$ and $u_0\in \mathcal{B}(\mathbb{S};[0,1])$. 
	\end{lemma}

With Lemma \ref{L:Moment_hatu},  we obtain 
the following continuity result that is uniform over all initial conditions $u_0\in \mathcal{B}(\mathbb{S};[0,1])$.
\begin{lemma}\label{L:uniform continuity}
For any $t\in(0,\infty)$ and $\eta,\,\eta'\in(0,1)$, there exists $\delta=\delta(t,\eta,\eta')\in(0,1)$ such that
\begin{equation}\label{E:uniform continuity}
    \sup_{u_0\in \mathcal{B}(\mathbb{S};[0,1])}\P_{u_0}\left(\omega(u_t;\delta)\;\geq\, \eta \right)\leq \eta', \qquad \text{where}\quad \omega(f;\delta):= \sup_{x,y\in\mathbb{S}:\,|x-y|\leq \delta}|f(x)-f(y)|.
\end{equation}
That is, $\omega(u_t;\delta)\to 0$ in probability under $\P_{u_0}$, uniformly over all initial conditions $u_0\in \mathcal{B}(\mathbb{S};[0,1])$, as $\delta\to 0$. In particular,  $\{\P_{f}(u_t\in\cdot):\,f\in \mathcal{B}(\mathbb{S};[0,1])\}$ is tight in $\calP(\mathcal{C}(\mathbb{S};[0,1]))$.
\end{lemma}

\begin{proof}[Proof of Lemma \ref{L:uniform continuity}]
We fix $t>0$ throughout the proof.
Note that $\omega(u_t;\delta)\to 0$ almost surely under $\P_{u_0}$ for each $u_0\in \mathcal{B}(\mathbb{S};[0,1])$ as $\delta\to 0$, because 
Lemma \ref{L:Moment_hatu} implies, via \cite[Theorem 1.1]{MR876085}, that any mild solution $u$ to \eqref{fkpp_X} is continuous (i.e. an element of $\mathcal{C}((0,\infty)\times \mathbb{S},[0,1])$) almost surely. 

We identify $\mathbb{S}$ with the unit interval $[0,1)$.
For $n\in\mathbb{N}$, we let
$D_n:=\{k 2^{-n}:\,k=0,1,\cdots,2^n-1\}\subset \mathbb{S}$
be the $n$-th dyadic partition of $\mathbb{S}$, and let $\mathbb{D}=\cup_{n=1}^{\infty}D_n$ be the set of dyadic rationals in $\mathbb{S}$. 
Since $u_t\in \mathcal{C}(\mathbb{S},[0,1])$, it is enough to show that for all  $\eta>0$, 
\begin{equation}\label{E:uniform continuity2}
   \P_{u_0}\left(\sup_{x,y\in \mathbb{D}:\,|x-y|\leq 2^{1-m}}|u_t(x)-u_t(y)|\;\geq\, \eta \right)\to 0 \qquad\text{as }m\to\infty
\end{equation}
uniformly in $u_0\in \mathcal{B}(\mathbb{S};[0,1])$. In the following, $\lambda>0$ and $p>2$ are constants to be determined. For $x,y\in \mathbb{D}$ with $|x-y|\leq 2^{1-m}$, 
\begin{align}
|u_t(x)-u_t(y)|\leq&\, \sum_{n:\,n\geq m}\max_{x\in D_n}|u_t(x+2^{-n})-u_t(x)|\\
\leq &\,\sum_{n:\,n\geq m}\left(\frac{S_{p,\lambda}(m)}{2^{n\lambda}}\right)^{1/p} \,=\,S_{p,\lambda}(m)^{\frac{1}{p}}\,\frac{2^{\frac{-m\lambda}{p} }}{1-2^{\frac{-\lambda}{p} }}\label{E:uniform continuity3},
\end{align}
where 
\begin{align}
S_{p,\lambda}(m):=&\,\sum_{n:\,n\geq m} 2^{n\lambda} \sum_{x\in D_n} |u_t(x+2^{-n})-u_t(x)|^p \\
\geq& \,2^{n\lambda} \max_{x\in D_n} |u_t(x+2^{-n})-u_t(x)|^p  \quad\text{ for each }n\geq m.
\end{align}

By Lemma \ref{L:uniform continuity} and \eqref{E:uniform continuity3}, the probability on the left-hand side of \eqref{E:uniform continuity2} is bounded above by
\begin{align}\label{E:uniform continuity4}
   \P_{u_0}\left(S_{p,\lambda}(m)\;\geq\, \eta^p \,(1-2^{\frac{-\lambda}{p} })^p\,2^{m\lambda}\right) \leq&\, \frac{\E_{u_0}[S_{p,\lambda}(m)]}{\eta^p \,(1-2^{\frac{-\lambda}{p} })^p\,2^{m\lambda}}\\
   = &\,C'\,2^{-m\lambda}\,\sum_{n:\,n\geq m} 2^{n\lambda} \sum_{x\in D_n} \E_{u_0}[|u_t(x+2^{-n})-u_t(x)|^p]\\
   \leq&\,C'\,C_{t,p}\,2^{-m\lambda}\,\sum_{n:\,n\geq m} 2^{n\lambda} \sum_{x\in D_n} (2^{-n})^{\frac{p}{2}}\\
   \leq&\,C'\,C_{t,p}\,2^{-m\lambda}\,\sum_{n:\,n\geq m} 2^{n(\lambda+1-\frac{p}{2})}
\end{align}
where we define $C':=\left[\eta^p \,(1-2^{\frac{-\lambda}{p} })^p\right]^{-1}$. The right-hand side of the above does not depend upon $u_0$ and converges to 0 as $m\to\infty$, provided that $\lambda+1-\frac{p}{2}<0$ (i.e. $p>2(\lambda+1)$) and $\lambda>0$. The proof is complete.
\end{proof}
 The following lemma \ref{L:ODE} makes precise the statement that \eqref{fkpp_X} is a spatial version of the 1-dimensional Wright-Fisher diffusion \eqref{1dWF}. We recall that $b(u):=\beta u(1-u)$ and $\sigma(u):=\sqrt{\gamma u(1-u)}$. 
 
\begin{lemma}[Reduction to the well-mixed case]\label{L:ODE}
We let $u$ be a mild solution to \eqref{fkpp_X}. Then for all $t\in(0,\infty)$,  as $\alpha\to\infty$, $u_t$ converges to a constant function on $\mathbb{S}$ whose values $V(t)$ (for time $t\in\R_{\geq 0}$) are a weak solution to the stochastic ODE 
$$dV(t)=b(V(t))\,dt +\sigma( V(t))\,dB_t,$$ 
where $B$ is the standard Brownian motion. Precisely, for all $0<T'<T<\infty$, as $\alpha\to\infty$,
\begin{equation}\label{E:ODE1}
\sup_{t\in[T',\,T]}\sup_{x\in \mathbb{S}}\left|u_t(x)-\int_{\mathbb{S}}u_t(y)\,m(dy)\right|\,\longrightarrow \,0  \quad\text{in }L^2(\P),
\end{equation}
and
\begin{equation}\label{E:ODE2}
\left(\int_{\mathbb{S}}u_t(y)\,m(dy) \right)_{t\in[0,T]}\,\longrightarrow \,\left(V(t) \right)_{t\in[0,T]}\quad\text{in distribution in }\mathcal{C}([0,T]). 
\end{equation}
\end{lemma}

\begin{proof}[Proof of Lemma \ref{L:ODE}]
For all $t\in(0,\infty)$, we have $u_t\in \mathcal{C}(\mathbb{S};\,[0,1])$.
Since $\mathbb{S}$ is compact, $p(t,x,y)\to 1$ uniformly on $[T',T]\times \mathbb{S}\times \mathbb{S}$ for all $0<T'<T<\infty$.
Using this local uniform convergence, boundedness of $b$ and $\sigma$,  \eqref{E:MildSol_u} and  the Burkholder-Davis-Gundy inequality (see, for instance \cite[Theorem 26.12]{kallenberg1997foundations}), one can show that 
$$\sup_{t\in[T',T]}|u_t(x^*)-u_t(x_*)|\to 0  \quad\text{in }L^2(\P),$$ 
where $x^*$ and $x_*$ are (random and time dependent) points on $\mathbb{S}$ such that $u_t(x^*)=\sup_{x\in \mathbb{S}}u_t(x)$ and $u_t(x_*)=\inf_{x\in \mathbb{S}}u_t(x)$.
The latter convergence implies \eqref{E:ODE1}.

The second convergence \eqref{E:ODE2} will follow from tightness of the left hand side in $\mathcal{C}([0,T])$ and convergence in finite-dimensional distributions. This tightness follows from the moment estimate \eqref{E:Moment_hatu}: for all $p\geq 2$,
\begin{align}
\E_{u_0}\left[\Big|\int_{\mathbb{S}}u_{t_1}(y)\,m(dy)-\int_{\mathbb{S}}u_{t_2}(y)\,m(dy) \Big|^p\right]  
\leq&\,\E_{u_0}\left[\int_{\mathbb{S}}|u_{t_1}(y)-u_{t_2}(y)|^p\,m(dy)\right]
\leq \, C\,|t_1-t_2|^{p/4}.
\end{align}
Tightness in $\mathcal{C}([0,T])$ follows if we take $p>4$.

Using \eqref{E:ODE1}, the continuity of $b$ and $\sigma$, and the aforementioned uniform convergence $p(t,x,y)\to 1$ on $[T',T]\times \mathbb{S}\times \mathbb{S}$ for all $0<T'<T<\infty$, we see that as $\alpha\to\infty$, the last term in \eqref{E:MildSol_u} satisfies
\[
\int_{\mathbb{S}\times [0,t]}p(t-s,x,z)\,
\sigma\big(u_s(z)\big)\,dW(z,s) - \int_{\mathbb{S}\times [0,t]}
\sigma\Big(\int_{\mathbb{S}}u_s(y)\,m(dy)\Big)\,dW(z,s)\to 0 \quad\text{in }L^2(\P).
\]
On the other hand, if we let $B_t:=\int_{\mathbb{S}\times [0,t]} \,dW(z,s)$ for $t\in\R_{\geq 0}$, then
\[
\left(\int_{\mathbb{S}\times [0,t]}
\sigma\Big(\int_{\mathbb{S}}u_s(y)\,m(dy)\Big)\,dW(z,s)\right)_{t\in\R_{\geq 0}}\eqd 
\left(\int_{0}^t
\sigma\Big(\int_{\mathbb{S}}u_s(y)\,m(dy)\Big)\,dB_s\right)_{t\in\R_{\geq 0}}.
\]
Convergence of $\int_{\mathbb{S}}u_t(y)\,m(dy)$ to $V(t)$ in distribution in $\R$,  for a fixed $t$,  now follows by integrating both sides of \eqref{E:MildSol_u} with respect to $x\in \mathbb{S}$ under $m(dx)$. Convergence of finite dimensional distributions can be shown similarly, by using the Markov property of $(u_t)_{t\in\R_{\geq 0}}$.
We have proved the second convergence \eqref{E:ODE2}.
\end{proof}

\subsection{Proof of Lemma \ref{lem:moments determine measures}}

Denote by $\mathcal{M}_{+}(E)$ the space of finite non-negative measures on a set $E$, equipped with the weak topology.
Suppose $\mu_1,\mu_2\in\mathcal{M}_{+}(\calC_*)$ are such that 
\begin{equation}\label{A:moments determine measures}
\int_{\calC_*}\Big[\prod_{i=1}^n(1-f(x_i))\Big]\Big[1-\prod_{j=1}^m(1-f(y_j))\Big]\mu_1(df)=\int_{\calC_*}\Big[\prod_{i=1}^n(1-f(x_i))\Big]\Big[1-\prod_{j=1}^m(1-f(y_j))\Big]\mu_2(df)
\end{equation}
for all $\{x_i\}_{i=1}^n\in \mathbb{S}^n/\sim$ and $\{y_j\}_{j=1}^m\in \mathbb{S}^m/\sim$,  for all $n,m\geq 1$.  
Let $(y_j)_{j=1}^{\infty}$ be dense subset $\bfS$. Then for every $f$ that is not equal to 0 almost everywhere on $\mathbb{S}$, we have
$1-\prod_{j=1}^m(1-f(y_j))\ra 1$ as $m\to\infty$. 
Hence, by the dominated convergence theorem,
\[
\int_{\calC_*}\Big[\prod_{i=1}^n(1-f(x_i))\Big]\mu_1(df)=\int_{\calC_*}\Big[\prod_{i=1}^n(1-f(x_i))\Big]\mu_2(df).
\]
This implies that, for any $k\in \mathbb{N}$ and any vector $(w_i)_{i=1}^k\in \mathbb{S}^k$, all cross moments of
$\{f(w_i)\}_{i=1}^k$ under $\mu_1(df)$ are the same as those under $\mu_2(df)$. Since the Carleman's condition is satisfied for the variables $\{f(w_i)\}_{i=1}^k$ since $f$ is bounded, the restrictions $\mu_1|_{(w_i)_{i=1}^k}$ and $\mu_2|_{(w_i)_{i=1}^k}$ are the same element of $\mathcal{M}_{+}([0,1]^k)$.

We have shown that any finite dimensional projection of $\mu_1$ is equal to that of $\mu_2$. Hence  $\mu_1=\mu_2$ in $\mathcal{M}_{+}(\calC_*)$ by Dynkin's $\pi-\lambda$ theorem, since the collection of cylindrical subsets of $\calC_*$ is a $\pi$ system generating the Borel $\sigma$-algebra on $\calC_{\ast}$.
\qed

\subsection{Early stopping reduces spectral radius}

Lemma \ref{lem:early stopping reduces spectral radius}  implies that if we kill a process at an earlier time, then the spectral radius (the eigenvalue of the semigroup of the killed process) will be smaller.  It is used in the proof of Theorem \ref{T:M*}.

Let $(X_t)_{t\in\R_{\ge 0}}$ be a  Markov process taking values in  a topological space $E$ and  $\tau_{\partial}$ be a stopping time with respect to the natural filtration $\{\mathcal{F}^X_t\}_{t\in \R_{\ge 0}}$ of $X$. Consider
the killed (or absorbed) process $(\widetilde{X}_t)_{0\leq t<\tau_{\partial}}$, defined by $\widetilde{X}_t=X_t$ for $t< \tau_{\partial}$ and  $\widetilde{X}_t$ being a separate (cemetery) state for $t\geq \tau_{\partial}$.
\begin{lemma}\label{lem:early stopping reduces spectral radius}
We suppose that the sub-Markovian kernel of the killed process has a bounded, non-negative right eigenfunction $f\in \calB_b(E; \R_{\ge 0})$ and a corresponding eigenvalue over time $1$, $\lambda\in (0,1]$. 
Then the process $(M_t)_{t\in \R_{\ge 0}}$ defined by 
\[
M_t:=\lambda^{-t}\,f(X_t)\Ind(\tau_{\partial}>t)
\]
is a martingale for all initial conditions $x\in E$. 
Furthermore, suppose that there exists a stopping time $\tau'$ with respect to $\{\mathcal{F}^X_t\}_{t\in \R_{\ge 0}}$, and positive constant $c'>0$, such that $\tau'<\tau_{\partial}$ and $\inf_{t\in[0,\tau']}f(X_t)\geq c'$, almost surely under $\Pm_{\mu}$. Then
\begin{equation}\label{eq:finite power to stopping time for lemma}
    \expE_{\mu}[\lambda^{-\tau'}]<\infty.
\end{equation}
\end{lemma}

\begin{proof}[Proof of Lemma \ref{lem:early stopping reduces spectral radius}]
The fact that $(M_t)_{t\in \R_{\ge 0}}$ is a martingale follows from the Markov property of $X$ and the definition of right eigenfunction $f$, as follows. Let $(Q_t)_{t\in\R_{\ge 0}}$ be the sub-Markovian transition semigroup associated to $(X_t)_{0\leq t<\tau_{\partial}}$ and $0\leq s\leq t$. Then
$$\E_x[M_t\,|\,\mathcal{F}^X_s]=\lambda^{-t}\,Q_{t-s}f(X_s)\,\Ind(\tau_{\partial}>s) = 
\lambda^{-t}\,\lambda^{t-s}f(X_s)\,\Ind(\tau_{\partial}>s)=M_s.
$$

If $\lambda=1$ then \eqref{eq:finite power to stopping time for lemma} obviously holds, so we may assume that $\lambda\in (0,1)$. We may therefore take $T<\infty$ such that $\lambda^{-T}\geq \frac{\lvert\lvert f\rvert\rvert_{\infty}}{c'}$. We take $X_0\sim \mu$ and define the following discrete-time process
\[
K_n:=\begin{cases}
    f(X_0),\quad &n=0\\
    \lambda^{-nT}f(X_{nT}),\quad &n\geq 1,\; nT<\tau'\\
    \lambda^{-\lceil \frac{\tau'}{T}\rceil T}f(X_{\tau'}),\quad  & nT\geq \tau'
\end{cases}
\]
We observe that $(K_n)_{n\geq 0}$ is a non-decreasing sequence of non-negative random variables, so that $\lim_{n\ra \infty}K_n$ exists in $(0,\infty]$ and $\expE[\lim_{n\ra \infty}K_n]=\lim_{n\ra\infty}\expE[K_n]$, by the monotone convergence theorem. Furthermore,
\[
\expE[K_n]\leq \lambda^{-T}\expE[\lambda^{-(nT\wedge \tau')}f(X_{nT\wedge \tau'})]\leq \lambda^{-T}\expE[f(X_0)]\quad\text{for all}\quad n\geq 0,
\]
and 
\[
\lambda^{-\tau'}f(X_{\tau'})\leq \lambda^{-T}\lim_{n\ra\infty}K_n.
\]
It follows that $c'\expE[\lambda^{-\tau'}]\leq \expE[\lambda^{-\tau'}f(X_{\tau'})]\leq \lambda^{-2T}\expE[f(X_0)]<\infty$, so that $\lambda^{-\tau'}$ is integrable.
\end{proof}

\subsection{QSD for the $1$-dimensional Wright-Fisher diffusion}\label{SS:1-dim}


In this subsection,  we will demonstrate how to prove convergence to a QSD for the classical Wright-Fisher diffusion using standard spectral methods, the most common technique for studying convergence to a QSD.
The reader may observe that these arguments are applicable to a general finite dimensional diffusion, but would typically fail for stochastic PDEs; see Remark \ref{Rk:ChallengeInfDim}.
More precisely, we will establish existence and uniqueness (and more) of the QSD for the 
classical Wright Fisher equation \eqref{1dWF}  
\begin{equation*}
dX_t=\beta X_t(1-X_t)dt+\sqrt{\gamma\,X_t(1-X_t)}dB_t,\quad t\in (0,\infty),
\end{equation*}
that is the 1-dimensional analogue of the SPDE \eqref{fkpp_X}. Most of these results are known in the literature, but we make precise the statements and give a self-contained proof. We also characterize the small noise weak$-*$ limit of the QSD at the end of this section.
The notation in this subsection is self-contained and should not be confused with thise used in other parts of the paper.

This diffusion process $X$ has state space $[0,1]$ and two absorbing states $\{0,1\}$. The absorption time (or fixation time) is again denoted by $\tau_{\fix}=\inf\{t>0:X_t\in \{0,1\}\}$. Since $\{0,1\}$ is a cemetery set, we may consider this to be an absorbed Markov process, which we denote as $(X_t)_{0\leq t<\tau_{\partial}}$. We then define the associated sub-Markovian transition kernel $(P_t)_{t\in\R_{\ge 0}}$ by
\[
P_tf(x)=\expE_x[f(X_t)\Ind(\tau_{\fix}>t],\quad f\in \calB_b((0,1)).
\]
In the following lemma, we establish that $\{P_t\}_{t\in\R_{\ge 0}}$ is a strongly continuous semigroup on  $\calC_0:=\calC_0((0,1))=\{f\in \calC((0,1)):\,f(0+)=f(1-)=0\}$, the  space of real-valued continuous functions on $(0,1)$ that vanish at infinity. We will employ at times the following identification of  spaces, 
\begin{equation}\label{E:identification_C0}
\calC_0 
\simeq \{f\in \calC([0,1]):\,f(0)=f(1)=0\}.
\end{equation}

\begin{lemma}[Compactness of killed semigroup]\label{L:C0Compact}
The killed Wright-Fisher diffusion
$(X_t)_{0\leq t<\tau_{\partial}}$ is a  $\calC_0$-Feller process. Furthermore,
for all $t>0$,  $P_t:\,\calC_0\ra \calC_0$ is compact with spectral radius $r(P_t)\in (0,1)$.
\end{lemma}

\begin{proof}[Proof of Lemma \ref{L:C0Compact}]
First, we note that the killed process $(X_t)_{0\leq t<\tau_{\fix}}$ has a transition density $p(t,x,y)$ which is continuous on $(t,x,y)\in (0,\infty)\times (0,1)\times (0,1)$ 
(see Kunita and Ichihara \cite[Theorem 3]{Ichihara1974}). Note that this is a consequence of hypoellipticity, and is true for degenerate parabolic diffusions satisfying parabolic H\"ormander conditions by H\"ormander's theorem. Note also that we have not established boundedness of the transition density for any fixed $t>0$. It follows that for all $t,\epsilon>0$,
\[
P_t:(0,1)\ni x\mapsto P_t(x,\cdot)_{\lvert_{[\epsilon,1-\epsilon]}}\in (\calP([\epsilon,1-\epsilon]),\lvert\lvert \cdot\rvert\rvert_{\TV}))\subset (\calP((0,1)),\lvert\lvert \cdot\rvert\rvert_{\TV}))\quad \text{is continuous.}
\]
For all $0<s<\infty$, 
we have  $\Pm_x(\tau_{\fix}>s)\ra 0$ as $x\ra \{0,1\}$ since both boundary points are regular by Feller’s classification \cite[P.329]{durrett2008probability}, from which we can conclude that \begin{enumerate}
\item $P_t(x,\cdot)\ra 0$ in total variation as $x\ra \{0,1\}$;
\item $\sup_{x\in (0,1)}P_t(x,(0,1)\setminus [\epsilon,1-\epsilon])\ra 0$ as $\epsilon \ra 0$.
\end{enumerate}
It therefore follows that, for all $t>0$ fixed,
\[x\mapsto \begin{cases}
P_t(x,\cdot),\quad x\in (0,1)\\
0,\quad \qquad x\in \{0,1\}
\end{cases}\in (\calP((0,1)),\lvert\lvert \cdot\rvert\rvert_{\TV})
\]
is a  continuous map from $[0,1]$ to $(\calP((0,1)),\lvert\lvert \cdot\rvert\rvert_{\TV})$.
For all $x,y\in [0,1]$ and  $f\in\calB_b([0,1])$,
\begin{align}
\lvert P_tf(x)-P_tf(y)\rvert\leq \lvert\lvert P_t(x,\cdot)-P_t(y,\cdot)\rvert\rvert_{\TV}\lvert\,\lvert f\rvert\rvert_{\infty}.
\end{align}
By a standard argument as in \cite[Corollary 4.8 in Chapter II]{bass1994probabilistic},  we have $P_t(\calC_b((0,1)))\subset \calC_0$ and
\begin{equation}
\{P_tf:\, f\in \calC_0((0,1)) \mbox{ with }\|f\|_{\infty}\leq 1\} \quad\text{ is equicontinuous.} 
\end{equation}
Therefore, 
$P_t:\,\calC_0\ra \calC_0$ is compact for all $t>0$ by the Arzela-Ascoli theorem, using the identification 
\eqref{E:identification_C0}.

\medskip

It follows from the spectral radius formula that the following limit exists and gives the spectral radius,
\[
r(P_t)=\lim_{n\ra \infty}\lvert\lvert P_{nt}\rvert\rvert_{C_0\ra C_0}^{\frac{1}{n}}
=\lim_{n\ra \infty}\left(\sup_{x\in (0,1)}P_{nt}1(x) \right)^{\frac{1}{n}}=\lim_{n\ra \infty}\left(\sup_{x\in (0,1)}\Pm_x(\tau_{\fix}>nt)\right)^{\frac{1}{n}}.
\]
Since $\inf_{x\in (\frac{1}{3},\frac{2}{3})}\Pm_x(X_t\in (\frac{1}{3},\frac{2}{3}))>0$ and $\sup_{x\in(0,1)}\Pm_x(\tau_{\fix}> 1)<1$, it follows that $r(P_t)\in (0,1)$. Hence $P_t$ is compact for each $t>0$.
\end{proof}

\begin{remark}\rm \label{Rk:ChallengeInfDim}
The proof of this lemma relies on the Arzela-Ascoli theorem on $\calC([0,1])$ (note that $[0,1]$ is compact) and \eqref{E:identification_C0}.
Such an approach is \textit{\it not applicable} in an infinite dimensional setting like the stochastic FKPP, for instance because $\mathcal{C}(\mathbb{S};[0,1])$ is not compact (in fact not even locally compact), since
the unit ball in an infinite-dimensional normed space is never compact.
 \end{remark}
\begin{remark}
    It follows from Lemma \ref{L:C0Compact} that the killed Wright-Fisher diffusion, $(X_t)_{0\leq t<\tau_{\fix}}$, has an infinitesimal generator defined on a dense subspace of $\calC_0$, which we denote by $L^{\text{WF}}$. We will therefore refer to the eigenvalue of a right eigenfunction and of a QSD to be the eigenvalue with respect to the infinitesimal generator, in contrast to the stochastic FKPP case where we had to use eigenvalue to denote the eigenvalue with respect to $P_1$.
\end{remark}

\begin{prop}[Eigenfunction and time to absorption]\label{prop:WF eigen}
Fix an arbitrary $\beta\in \Rm$ and $\gamma\in(0,\infty)$ in \eqref{1dWF} . 
\begin{enumerate}
\item The  generator $L^{\text{WF}}$ on $\calC_0$ has a non-empty pure point spectrum $\sigma(L^{\text{WF}})$ which is finite on $\{z\in\mathbb{C}:\,\text{Re}(z)>-c\}$ for all $c>0$. Define $\kappa_0:=-\sup\{\text{Re}(z):z\in \sigma(L^{\text{WF}})\}$ and $\kappa_1:=-\sup\{\text{Re}(z)<-\kappa_0:z\in \sigma(L^{\text{WF}})\}$. Then $0<\kappa_0<\kappa_1\leq \infty$.
\item   $-\kappa_0$ is the unique eigenvalue of the generator belonging to the boundary spectrum $\{z\in \sigma(L^{\text{WF}}):\text{Re}(z)=-\kappa_0\}$. Also, $L^{\text{WF}}$ has a unique (up to constant multiple) eigenfunction for the eigenvalue $-\kappa_0$. This eigenfunction belongs to $\calC_0((0,1);\Rm_{>0})$, and is denoted as $h$.
\item   For all $0<r<\kappa_1-\kappa_0$, there exists $C=C(r)<\infty$ such that
\begin{equation}\label{eq:WF convergence of fixation prob}
    \lvert e^{\kappa_0 t}\Pm_{\mu}(\tau_{\fix}>t)-\mu(h)\rvert\leq Ce^{-rt}\quad\text{for all}\quad t\in(0,\infty),\quad \mu\in \calP((0,1)).
\end{equation}
\end{enumerate}
\end{prop}

\begin{proof}[Proof of Proposition \ref{prop:WF eigen}]
Given a strongly continuous, eventually compact semigroup on a Banach space, \cite[Theorem 2.1, p.209]{Arendt1986} characterises the asymptotics of the semigroup in terms of the spectrum of its infinitesimal generator. Lemma \ref{L:C0Compact} ensures that  $(P_t)_{t\geq 0}$ is eventually (in fact immediately) compact.
Our strategy is to  characterise the spectrum of its infinitesimal generator, apply \cite[Theorem 2.1]{Arendt1986} to characterise the asymptotics of $(P_t)_{t\geq 0}$, then conclude \eqref{eq:WF convergence of fixation prob} by taking the adjoint of this characterisation.

We first show that $(P_t)_{t\geq 0}$ is irreducible as a positive semigroup on the Banach lattice $(\calC_0((0,1)),\leq, \lvert\lvert \cdot\rvert\rvert_{\infty})$; see \cite[p.234-235]{Arendt1986} or
\cite[p.3]{Grobler1995} for the definition of a Banach lattice, and
\cite[Definition 3, p.306]{Arendt1986} for the definition of irreducibility.
We have from \cite[Definition 3 (iii), p.306]{Arendt1986} and the Riesz-Markov-Kakutani representation theorem that $(P_t)_{t\geq 0}$ is irreducible (in the sense of \cite[Definition 3, p.306]{Arendt1986}) if and only if for every $\mu\in \calP((0,1))$ and $f\in \calC_0((0,1);\Rm_{\geq 0})\setminus \{{\bf 0}\}$. 
We have $\mu(P_tf)>0$ for some $t=t(\mu,f)>0$. This is clearly the case as $\mu(P_tf)=\E_{\mu}[f(X_t)\Ind(\tau_{\fix}>t)]>0$ for all $t>0$, by the irreducibility of $(X_t)_{0\leq t<\tau_{\fix}}$. 
We now use the above to characterise the spectrum of $L^{\text{WF}}$. From the spectral mapping theorems \cite[(6.4), p.85 and Corollary 6.7 (i), p.87-88]{Arendt1986}, and the spectral theorem for compact operators, $L^{\text{WF}}$ has a non-empty pure point spectrum with all eigenvalues having strictly negative real part, and only finitely many eigenvalues with real part at least $-\kappa'$, for all $\kappa'\in \Rm$. 
This implies that $0<\kappa_0<\kappa_1\leq \infty$.

Then \cite[Theorem 3.8]{Arendt1986} implies that $-\kappa_0$ is the unique eigenvalue with real part $-\kappa_0$. Moreover $-\kappa_0\in \sigma(L^{\text{WF}})$ is an algebraically simple pole of the resolvent by \cite[Theorem 3.12, p.315]{Arendt1986}, hence $-\kappa_0$ has geometric multiplicity $1$ by \cite[(3.4), p.73]{Arendt1986}. The residue at $-\kappa_0$ is then given by $\pi\otimes h$ for some $\pi\in \calP((0,1))$ and $h\in \calC_0((0,1);\Rm_{>0})$ by \cite[Proposition 3.5, p.310]{Arendt1986}. We may 
identify $h\in \calC_0((0,1);\Rm_{>0})$ as the unique eigenfunction of eigenvalue $-\kappa_0$.

We now apply the above with \cite[Theorem 2.1, p.209]{Arendt1986} to conclude that for all $0<r<\kappa_1-\kappa_0$ there exists $C=C(r)<\infty$ such that
\[
\lvert\lvert e^{\kappa_0 t}P_t-h\otimes \pi\rvert\rvert_{\calC_0\to\calC_0}\leq Ce^{-rt},\quad 0\leq t<\infty.
\]

Taking the adjoint of the above under the Riesz representation theorem, we obtain that
\begin{equation}\label{E:Adjoint of convergence}
\lvert\lvert e^{\kappa_0 t}\mu P_t-\mu(h)\pi\rvert\rvert_{\TV}\leq Ce^{-rt},\quad 0\leq t<\infty,\quad \mu\in\calP((0,1)).
\end{equation}
We immediately conclude \eqref{eq:WF convergence of fixation prob}. 
\end{proof}

\begin{prop}[Convergence to the unique QSD]\label{prop:WF diffusion convergence to QSD}
For any $\beta\in \Rm$ and $\gamma\in(0,\infty)$, the classical Wright-Fisher diffusion \eqref{1dWF}  has a unique QSD, which we denote as $\pi$. Furthermore for all $0<r<\kappa_1-\kappa_0$ there exists a fixed constant $C=C(r)<\infty$, and a time $T=T(r,\mu)<\infty$ that is continuously dependent upon $\mu\in \calP((0,1))$, such that
\begin{equation}\label{eq:WF convergence to QSD}
\lvert\lvert \Law_{\mu}(X_t\lvert \tau_{\fix}>t)-\pi\rvert\rvert_{\rm TV}\leq \frac{Ce^{-rt}}{\mu(h)}\quad\text{for all}\quad t\in [T,\infty),\quad \mu\in \calP((0,1)).
\end{equation}
\end{prop}

\begin{proof}[Proof of Proposition \ref{prop:WF diffusion convergence to QSD}]
Let $\pi$ and $h$ be given by the proof of Proposition \ref{prop:WF eigen} part 2.

We  obtain \eqref{eq:WF convergence to QSD} by algebraic manipulation of \eqref{E:Adjoint of convergence}. Since $\pi\in\calP((0,1))$ is the unique quasi-limiting distribution, it must be a quasi-stationary distribution, and be unique.
\end{proof}

{\bf Properties of the QSD and spectrum of $L^{\text{WF}}$. }
The QSD $\pi$ is a left eigenmeasure of $L^{\text{WF}}$ \cite[Proposition 4, P.349]{meleard2012quasi} with eigenvalue $-\kappa_0$, hence on the interior $(0,1)$ it is a classical solution of
\begin{align}\label{ODEpi_neutralWF}
\frac{\gamma}{2}\,[x(1-x)\pi(x)]'' - [\beta x(1-x)\pi(x)]' +\kappa_0\, \pi(x) =&\,0.
\end{align}
The right eigenfunctions $h$ are solutions of $L^{\text{WF}}h(x)=-\kappa h(x)$ for $x\in(0,1)$ with Dirichlet boundary condition $h(0)=h(1)=0$. Solving this provides for the eigenvalues.

For equation \eqref{1dWF}, 
an explicit expression for the transition density for the killed Wright-Fisher diffusion 
$(X_t)_{0\leq t<\tau_{\partial}}$
was provided in  \cite[eqn.(8)]{kimura1954stochastic} and in \cite[eqn. (1.6)-(1.7)]{mano2009duality}. 
Note that $\gamma$ in \eqref{1dWF} is $1/(2N)$ in Kimura \cite[eqn.(4)]{kimura1954stochastic}, and $\beta$ in \eqref{1dWF} is $s=c/N$ in \cite{kimura1954stochastic}. The eigenvalues $0<\kappa_0<\kappa_1<\cdots$ of the generator satisfy
$\kappa_k=\frac{1}{4N}\left(c^2-C_{1,k}\right)$
where $C_{1,k}$ are the separation constants \cite{stratton1941elliptic} and $c=\frac{\beta}{2\gamma}$. In the weak selection regime, the principle eigenvalue $\kappa_0>0$ satisfies the asymptotic
\[
\kappa_0 
=\gamma\Big(1+\frac{\beta^2}{10\gamma^2}-\frac{\beta^4}{7000\gamma^4}+\ldots\Big) ,\qquad \text{as}\qquad  c=\frac{\beta}{2\gamma}\to 0.
\]

In the other extreme, namely $\frac{\beta}{2\gamma}\to\infty$, it is straightforward to establish the following small-noise limit of the QSD (when $\beta\in \R_{\ge 0}$ is fixed).
\begin{prop}\label{prop:weak*wellmixed}
As  $\gamma\to 0$, the principle eigenvalue $\kappa\to 0$ and  $\pi$ converges weakly to $\pi_*\in\mathcal{P}([0,1])$
that satisfies
\begin{equation}
\pi_*=
\begin{cases}
\delta_1 \quad &{\rm for }\; \beta>0\\
{\rm Unif}((0,1)) \quad &{\rm for }\; \beta=0\\
\end{cases}
\end{equation}  
\end{prop}

Such a small noise  limit  for the support of the QSD of the Wright Fisher diffusion \eqref{1dWF} 
highlights the contrast between the neutral case and the case with selection. Similar asymptotic results for the support of QSD can be found in \cite{ramanan1999quasi} and \cite[Section 3.3]{faure2014quasi}.

\section*{Acknowledgements}
Research supported by National Science Foundation grant DMS-2152103 and Office of Naval Research grant N00014-20-1-2411 to W.-T. Fan. The work of OT was partially supported both by grant 200020 196999 from the Swiss National
Foundation and by the EPSRC MathRad programme
grant EP/W026899/. We thank Daniel Rickert for assistance in producing the figures. 

\bibliographystyle{alpha}
\bibliography{QuasiSPDE}
\end{document}